\documentclass{article}

\usepackage{graphicx}
\usepackage{xcolor}
\usepackage{amsthm}
\usepackage{amsmath}
\usepackage{amssymb}
\usepackage{fullpage}
\usepackage{subcaption}
\usepackage{hyperref}

\usepackage[backend=biber,url=false,eprint=false,maxbibnames=99]{biblatex}
\newtheorem{definition}{Definition}[section]
\newtheorem{lemma}[definition]{Lemma}
\newtheorem{proposition}[definition]{Proposition}
\newtheorem{theorem}[definition]{Theorem}
\newtheorem{remark}[definition]{Remark}
\newtheorem{example}[definition]{Example}

\DeclareMathOperator{\sign}{sgn}
\DeclareMathOperator{\real}{Re}
\DeclareMathOperator{\imag}{Im}

\DeclareMathOperator{\polylog}{Li}
\DeclareMathOperator{\sinc}{sinc}
\DeclareMathOperator{\erfi}{erfi}

\newcommand{\HT}{\mathbf{H}} 
\newcommand{\Hop}[1]{\mathcal{H}^{#1}}

\newcommand{\HopGeneric}{\mathcal{H}}

\newcommand{\inter}[1]{\textbf{\textsl{#1}}}
\newcommand{\lo}[1]{\underline{#1}}
\newcommand{\hi}[1]{\overline{#1}}

\graphicspath{{figures/}}

\title{Highest Cusped Waves for the Burgers-Hilbert equation}
\author{Joel Dahne\and Javier G\'omez-Serrano}

\begin{document}
\maketitle

\begin{abstract}
  In this paper we prove the existence of a periodic highest, cusped,
  traveling wave solution for the Burgers-Hilbert equation
  \(f_{t} + f f_{x} = \HT[f]\) and give its asymptotic behaviour at
  \(0\). The proof combines careful asymptotic analysis and a
  computer-assisted approach.
\end{abstract}

\section{Introduction}
\label{sec:introduction}
The Burgers-Hilbert equation \cite{hunter18:burger-hilbert} is a
nonlinear wave model, in the periodic setting given by
\begin{equation}
  \label{eq:BH}
  f_{t} + f f_{x} = \HT[f],\quad \text{ for } (x, t) \in \mathbb{T} \times \mathbb{R}.
\end{equation}
Here \(\HT\) is the Hilbert transform which, for
\(f: \mathbb{T} \to \mathbb{R}\), is defined by
\begin{equation*}
  \HT[f](x) = \frac{1}{2\pi}p.v. \int_{-\pi}^{\pi}\cot\left(\frac{x - y}{2}\right)f(y)\ dy,\quad
  \widehat{\HT[f]}(k) = -i \sign(k)\widehat{f}(k).
\end{equation*}
The equation was first used by Marsden and Weinstein in 1983 as a
second order approximation for the evolution of the boundary of a
simply connected vortex patch in two dimensions \cite{MARSDEN1983305}.
More recently Biello and Hunter used it to serve as an approximation
for small slope vorticity fronts \cite{biellohunter10:nonlinear}. The
validity of this approximation was recently proved
\cite{hunter21_approx_vortic_front_by_equat}.

For small initial data in \(H^{2}(\mathbb{R})\), estimates for the
lifespan were proved by Hunter and Ifrim
\cite{hunterifrim12:enhanced}, see also~\cite{hunter2015long}.

Global existence of weak solutions for initial data in
\(L^{2}(\mathbb{R})\) was established by Bressan and Nguyen, in which
case the solution lies in
\(L^{2}(\mathbb{R}) \cap L^{\infty}(\mathbb{R})\) for \(t > 0\)
\cite{Bressan2014}. Bressan and Zhang constructed locally in time
piecewise smooth solutions with a single, logarithmic, shock
\cite{Bressan2017}. Stability and uniqueness of these solutions in a
larger class of solutions were shown by Krupa and Vasseur
\cite{Krupa2020}.

Numerical simulations have shown formation of shocks in finite time
\cite{biellohunter10:nonlinear, hunter18:burger-hilbert}. Castro,
Córdoba and Gancedo proved finite time blow up of the
\(C^{1,\delta}\)-norm with \(0 < \delta < 1\) for initial data
\(f_{0} \in L^{\infty}(\mathbb{R}) \cap C^{1,\delta}(\mathbb{R})\)
satisfying that there exists a point \(x_{0}\) with
\(H[f_{0}](x_{0}) > 0\) and
\(u_{0}(x_{0}) \geq (32\pi\|u_{0}\|^{2}_{L^{2}(\mathbb{R})})^{1 / 3}\)
\cite{Castro2010}. Saut and Wang proved shock formation in finite time
\cite{saut20:whith}. Solutions that develop an asymptotic self-similar
shock at a single point with an explicitly computable blowup profile
were constructed by Yang \cite{yang20:shock_burger_hilber}.

The Burgers-Hilbert equation occurs as a special case in the family of
fractional KdV equations given by
\begin{equation}
  \label{eq:fkdv}
  f_{t} + f f_{x} = |D|^{\alpha}f_{x}.
\end{equation}
Here
\begin{equation*}
  \widehat{|D|^{\alpha}f}(\xi) = |\xi|^{\alpha}\widehat{f}(\xi)
\end{equation*}
and the parameter \(\alpha\) may in general take any real value. For
\(\alpha = 2\) and \(\alpha = 1\) we get the classical KdV and
Benjamin-Ono equations. For \(\alpha = -1\) it reduces to the
Burgers-Hilbert equation.

For \(\alpha \in (-1, 0)\) the fractional KdV equation exhibits finite
time blow \cite{Castro2010,Hur2012}. That this blowup happens in terms
of wave breaking was proved for \(\alpha \in (-1, -1 / 3)\) by Hur and
Tao \cite{Hur2014,Hur2017} and for \(\alpha \in (-1, 0)\) by Oh and
Pasqualotto \cite{oh21:_gradien_burger}. See \cite{Klein2015} for a
numerical study and also \cite{chickering21:_asymp_burger}. For other
results see e.g. \cite{ehrnstromwang19:enhanced} for existence time,
\cite{Riao2021} for well-posedness and \cite{klein20:_kortew_vries,
  Klein2017} for some results for related equations.

In this work we are concerned with traveling wave solutions of the
Burgers-Hilbert equation \eqref{eq:BH}. The study of traveling waves
is an important topic in fluid dynamics, see e.g. \cite{Haziot2022}
for a recent overview of traveling water waves. The traveling wave
assumption \(f(x, t) = \varphi(x - ct)\), where \(c > 0\) denotes the
wave speed, gives us
\begin{equation}
  \label{eq:bh-wave}
  -c \varphi' + \varphi\varphi' = \HT[\varphi].
\end{equation}
The Burgers-Hilbert equation has an analytic branch of even,
zero-mean, \(2\pi\)-periodic, smooth traveling wave solutions
bifurcating from constant solutions \cite{hunter18:burger-hilbert}. If
we let \(\epsilon\) be the bifurcation parameter and
\(\varphi_{\epsilon}\), \(c_{\epsilon}\) be the solution with its
corresponding wave speed, then as \(\epsilon \to 0\) we have
\begin{align*}
  \varphi_{\epsilon}(x) &= \epsilon \cos(x) + \mathcal{O}(\epsilon^{2}),\\
  c_{\epsilon} &= -1 + \mathcal{O}(\epsilon^{2}).
\end{align*}
Castro, Córdoba and Zheng \cite{castro2021stability} proved that this
branch exists in the range \((0, \epsilon^{*})\) with
\(\epsilon^{*} \sim 0.23\) and fails to exist for
\(\epsilon > \frac{2}{e}\). Moreover they proved an enhanced lifespan
estimate for perturbations of \(\varphi_{\epsilon}\) compared to the
results in \cite{hunterifrim12:enhanced}. In
\cite{hunter18:burger-hilbert}, Hunter remarks that this branch
presumably ends in a highest wave which is not smooth at its crest, as
is common for equations of this type. In this paper we prove the
existence of a highest cusped wave and give its behaviour at the
crest. More precisely we prove the following theorem:
\begin{theorem}
  \label{thm:main}
  There is a \(2\pi\)-periodic traveling wave \(\varphi\) of
  \eqref{eq:bh-wave}, which behaves asymptotically at \(x = 0\) as
  \begin{equation*}
    \varphi(x) = c + \frac{1}{\pi}|x|\log|x| + \mathcal{O}(|x|\sqrt{\log|x|}).
  \end{equation*}
\end{theorem}
\begin{remark}
  We don't expect the remainder term
  \(\mathcal{O}(|x|\sqrt{\log|x|})\) in Theorem \ref{thm:main} to be
  sharp. The estimate follows from the choice of our space.
\end{remark}
The notion of highest traveling waves exist for a large number of
equations. For the free boundary Euler equation Stokes argued that if
there exists a singular solution with a steady profile it must have an
interior angle of \(120^{\circ}\) at the crest
\cite{stokes2018mathematical}. This is known as the Stokes conjecture
and was proved in 1982 \cite{Amick1982}. For the Whitham equation
\cite{Whitham1967} the existence of a highest cusped traveling wave
was conjectured by Whitham in~\cite{Whitham1999}. Its existence,
together with its \(C^{1 / 2}\) regularity, was recently proved by
Ehrnström and Wahlén \cite{Ehrnstrm2019}. They conjectured that the
wave is convex between its crests and also its precise asymptotic
behaviour. This conjecture was proved by Enciso, Gómez-Serrano and
Vergara \cite{enciso2018convexity}. See also \cite{ehrnstrom22:_whith}
for some recent remarks. For the fractional KdV equations the
traveling waves assumption allows us to write the equation as
\begin{equation}
  \label{eq:fkdv-wave}
  -c \varphi' + \varphi\varphi' = |D|^{\alpha} \varphi'.
\end{equation}
There has recently been much progress related to highest waves to this
family of equations. For \(\alpha < -1\) Bruell and Dhara proved the
existence of highest traveling waves which are Lipschitz at their cusp
\cite{bruell18:_waves}. Very recently Ørke proved their existence for
\(-1 < \alpha < 0\) for the (inhomogeneous) fractional KdV equations
as well as the fractional Degasperis-Procesi equations, together with
their optimal \(-\alpha\)-Hölder regularity
\cite{orke22:_highes_kortew}. Hildrum and Xue prove a similar result
for another class of equations, including the (homogeneous) fractional
KdV equations for \(-1 < \alpha < 0\)
\cite{hildrum22:_period_holder_kdv}.

The results in \cite{Ehrnstrm2019, bruell18:_waves,
  orke22:_highes_kortew, hildrum22:_period_holder_kdv} are all based
on global bifurcation arguments, bifurcating from the constant
solution and proving that the branch must end in a highest wave which
is not smooth at its crest. The proof of the convexity of the highest
cusped wave for the Whitham equation in \cite{enciso2018convexity}
uses a completely different approach. The problem is first rephrased
as a fixed point problem, the existence and properties of the fixed
point is then related to inequalities for certain constants that
appear in the reduction. These inequalities are then checked using a a
computer assisted proof. In this paper we use a similar approach for
proving Theorem \ref{thm:main}.

The ansatz \(\varphi(x) = c - u(x)\) allows us to rewrite
\eqref{eq:bh-wave} as an equation that does not explicitly depend on
the wave speed \(c\). Proving the existence of a solution \(u\) can be
rewritten as a fixed point problem by considering the ansatz
\begin{equation*}
  u(x) = u_{0}(x) + w(x)v(x)
\end{equation*}
where \(u_{0}(x)\) is an explicit, carefully chosen, approximate
solution and \(w(x)\) is an explicit weight factor. Proving the
existence of a fixed point can be reduced to checking an inequality
involving three constants, \(D_{0}\), \(\delta_{0}\) and \(n_{0}\),
that only depend on the choice of \(u_{0}\) and \(w\), see Proposition
\ref{prop:contraction}. This inequality is checked by bounding
\(D_{0}\), \(\delta_{0}\) and \(n_{0}\) using a computer assisted
proof, see Lemmas \ref{lemma:bh-bounds-n}, \ref{lemma:bh-bounds-delta}
and \ref{lemma:bh-bounds-D}. These bounds are highly non-trivial, in
particular \(\delta_{0}\) is given by the supremum of a function on
the interval \([0, \pi]\) which attains its maximum around
\(10^{-5000}\). A plot of the function \(u\) is given in Figure
\ref{fig:BH-u}.

\begin{figure}
  \centering
  \includegraphics[width=0.7\textwidth]{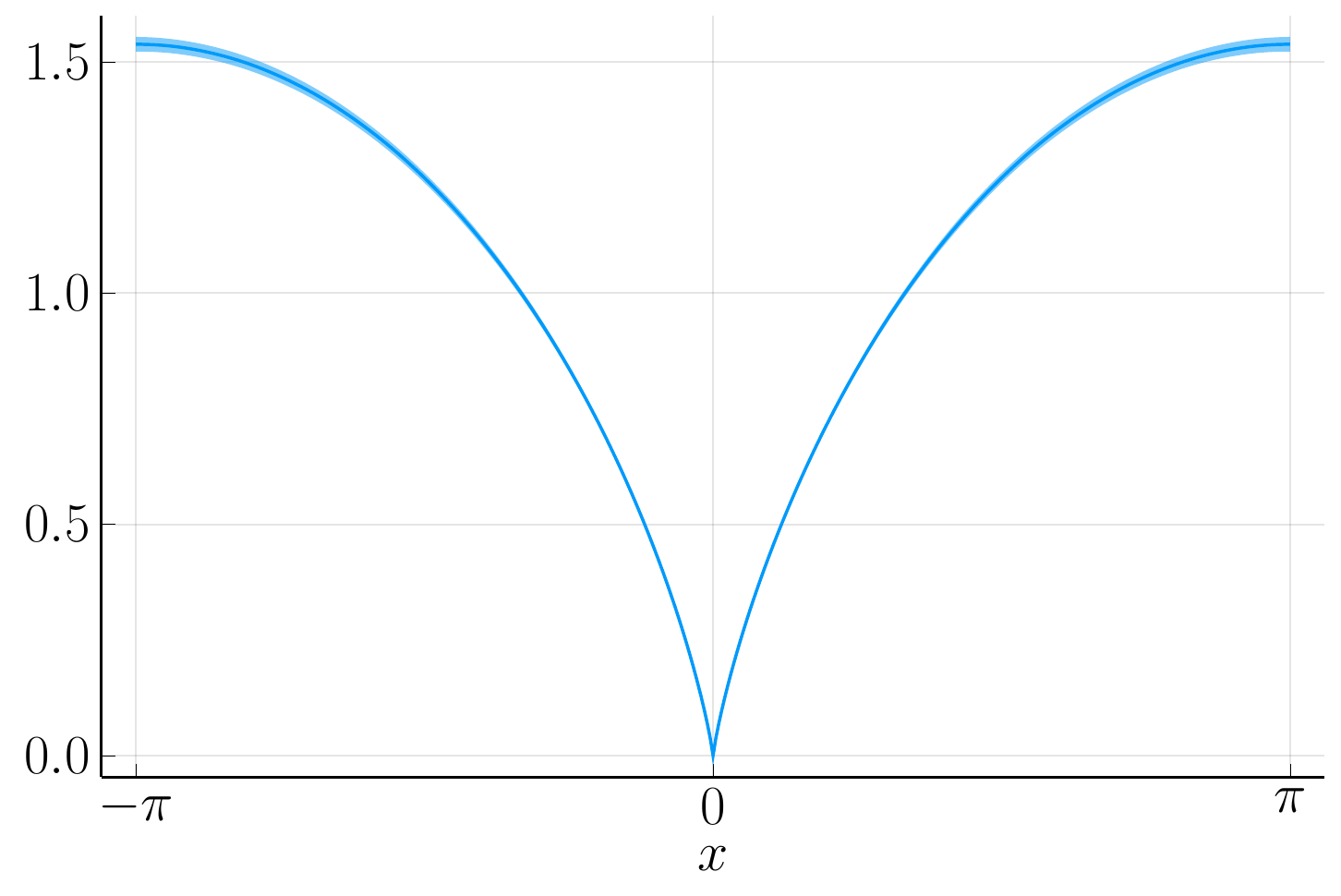}
  \caption{An enclosure of the function \(u(x)\) on the interval
    \([-\pi, \pi]\). The thin line is the approximation \(u_{0}(x)\)
    and the width of the thick line is computed using the bound of
    \(\|v\|_{L^{\infty}(\mathbb{T})}\).}
  \label{fig:BH-u}
\end{figure}

One of the key difficulties is the construction of the approximate
solution \(u_{0}\). Due to the singularity at \(x = 0\) it is not
possible to use a trigonometric polynomial alone, it would converge
very slowly and have the wrong asymptotic behaviour. Pure products of
powers and logarithms, \(|x|^{a}\log^{b} |x|\), have the issue that
they are not periodic and do not interact well with the operator
\(\HopGeneric\). Instead we take inspiration from the construction in
\cite{enciso2018convexity} and consider a combination of trigonometric
polynomials and Clausen functions of different orders, defined as
\begin{equation*}
  C_{s}(x) = \sum_{n = 1}^{\infty}\frac{\cos(nx)}{n^{s}},\quad
  S_{s}(x) = \sum_{n = 1}^{\infty}\frac{\sin(nx)}{n^{s}},
\end{equation*}
for \(s > 1\) and by analytic continuation otherwise. We also make use
of their derivatives with respect to the order, for which we use the
notation
\begin{equation*}
  C_{s}^{(\beta)}(x) := \frac{d^{\beta}}{ds^{\beta}} C_{s}(x),\quad
  S_{s}^{(\beta)}(x) := \frac{d^{\beta}}{ds^{\beta}} S_{s}(x).
\end{equation*}
These functions are \(2\pi\)-periodic, non-analytic at \(x = 0\) and
behave well with respect to \(\HopGeneric\). In particular
\(C_{2}^{(1)}(x) - C_{2}^{(1)}(0) \sim |x|\log |x|\), which
corresponds to the behaviour we expect in Theorem \ref{thm:main}.
However, directly applying the ideas from \cite{enciso2018convexity}
for the construction does not work since it is not easy to determine
the next term in the expansion. Instead we go through the fractional
KdV equations for \(-1 < \alpha < 0\), for this family of equations
the same ideas do work and we then study the limit \(\alpha \to -1\).

\begin{remark}
  From the representation of \(u_{0}\) and the bounds for \(v\), given
  in Section \ref{sec:proof-main-theorem}, it is possible to compute
  quantitative bounds for properties of the solution.

  It is possible to enclose the first Fourier coefficient of the
  solution \(u\), and hence also \(\varphi\). For \(u_{0}\) the first
  Fourier coefficient is given by \(-0.54771699...\), taking into
  account the bounds for \(v\) we can compute the enclosure
  \([0.534, 0.561]\) for the first Fourier coefficient of \(\varphi\).
  This agrees with the results from Castro, Córdoba and Zheng
  \cite{castro2021stability} that the branch of solutions break down
  for \(\epsilon\) somewhere between \(\epsilon^{*} \sim 0.23\) and
  \(\frac{2}{e}\). It also agrees with their (unpublished) numerical
  results indicating that the branch ends around \(\epsilon = 0.548\)
  \cite{private-castro2021}.

  It is also possible to enclose the mean of \(u\), which gives the
  wavespeed \(c\) for a corresponding zero-mean \(\varphi\). From the
  expression of \(u_{0}\) (see \eqref{eq:u0}) it is straightforward to
  compute the mean \(1.11041\dots\) of \(u_{0}\). Taking into account
  the bounds for \(v\) we get the enclosure \([1.1, 1.121]\) for the
  mean of \(u\).
\end{remark}

An important part of our work is the interplay between traditional
mathematical tools and rigorous computer calculations. Traditional
numerical methods typically only compute approximate results, to be
able to use the results in a proof we need the results to be
rigorously verified. The basis for rigorous calculations is interval
arithmetic, pioneered by Moore in the 1970's \cite{Moore1979}. Due to
improvements in both computational power as well as well as great
improvements in software it has become possible to use computer
assisted tools in many more problems. The main idea with interval
arithmetic is to do arithmetic not directly on real numbers but on
intervals with computer representable endpoints. Given a function
\(f:\mathbb{R} \to \mathbb{R}\), an interval extension of \(f\) is an
extension to intervals satisfying that for an interval \(\inter{x}\),
\(f(\inter{x})\) is an interval satisfying \(f(x) \in f(\inter{x})\)
for all \(x \in \inter{x}\). In particular this allows us to prove
inequalities for the function \(f\), for example the right endpoint of
\(f(\inter{x})\) gives an upper bound of \(f\) on the interval
\(\inter{x}\). For an introduction to interval arithmetic and rigorous
numerics we refer the reader to the books \cite{Moore1979,Tucker2011}
and to the survey \cite{GomezSerrano2019} for a specific treatment of
computer assisted proofs in PDE. For all the calculations in this
paper we make use of the Arb library \cite{Johansson2017arb} for ball
(intervals represented as a midpoint and radius) arithmetic. It has
good support for many of the special functions we use
\cite{Johansson2014hurwitz, Johansson2014thesis, Johansson2019},
Taylor arithmetic (see e.g. \cite{Johansson2015reversion}) as well as
rigorous integration~\cite{Johansson2018numerical}.

The paper is organized as follows. In Section
\ref{sec:reduct-fixed-point} we reduce the proof of Theorem
\ref{thm:main} to a fixed point problem. In Section
\ref{sec:clausen-functions} we give a brief overview of properties of
the Clausen functions that are relevant for the construction of
\(u_{0}\), in Section \ref{sec:construction} we give the construction
of \(u_{0}\). Section \ref{sec:bounding-n-delta} is devoted to the
approach for bounding \(n_{0}\) and \(\delta_{0}\), Section
\ref{sec:analysis-T} to studying the linear operator that appears in
the construction of the fixed point problem. The computer assisted
proofs giving bounds for \(n_{0}\), \(\delta_{0}\) and \(D_{0}\) are
given in Section \ref{sec:bounds-for-values}. Finally we give the
proof of Theorem \ref{thm:main} in Section
\ref{sec:proof-main-theorem}.

Three appendices are given at the end of the paper. Appendix
\ref{sec:removable-singularities} gives some technical details for how
to compute enclosures of functions around removable singularities.
Appendix \ref{sec:comp-encl-claus} is concerned with computing
enclosures of the Clausen functions and Appendix
\ref{sec:rigorous-integration} with the rigorous numerical integration
needed for bounding \(D_{0}\).

\section{Reduction to a fixed point problem}
\label{sec:reduct-fixed-point}
In this section we reduce the problem of proving Theorem
\ref{thm:main} to proving the existence of a fixed point for a certain
operator. We start with the following lemma which motivates the notion
of highest wave (see also~\cite[Theorem
3.4]{hildrum22:_period_holder_kdv})
\begin{lemma}
  Let \(\varphi \in C^{1}\) be a nonconstant, even solution of
  \eqref{eq:bh-wave} which is nondecreasing on \((-\pi, 0)\), then
  \begin{equation*}
    \varphi' > 0 \quad \text{ and }\quad \varphi < c
  \end{equation*}
  on \((-\pi, 0)\).
\end{lemma}
\begin{proof}
  We start by proving that under these assumptions
  \(\HT[\varphi] < 0\). Since \(\varphi\) is even we can write the
  Hilbert transform of \(\varphi\) as
  \begin{equation*}
    \HT[\varphi](x) = \frac{1}{2\pi} \int_{-\pi}^{0}\left(\cot\left(\frac{x - y}{2}\right) + \cot\left(\frac{x + y}{2}\right)\right)\varphi(y)\ dy.
  \end{equation*}
  Integration by parts gives us
  \begin{equation*}
    \HT[\varphi](x) = \frac{1}{2\pi} \int_{-\pi}^{0}\left(\log\left|\sin\left(\frac{x - y}{2}\right)\right| - \log\left|\sin\left(\frac{x + y}{2}\right)\right|\right)\varphi'(y)\ dy.
  \end{equation*}
  For \(x \in (-\pi, 0)\) and \(y \in (-\pi, 0)\) we have
  \begin{equation*}
    \left|\sin\left(\frac{x - y}{2}\right)\right| < \left|\sin\left(\frac{x + y}{2}\right)\right|.
  \end{equation*}
  For \(y > x\) this follows from that \(|\sin(x)|\) is increasing as
  a distance to a multiple of \(\pi\) and that \(\frac{x - y}{2}\) is
  closer to zero than \(\frac{x + y}{2}\) is to zero or \(-\pi\), the
  case \(y < x\) can be reduced to the previous case by switching
  \(x\) and \(y\) and using that \(\sin\) is odd. It follows that
  \begin{equation*}
    \log\left|\sin\left(\frac{x - y}{2}\right)\right| - \log\left|\sin\left(\frac{x + y}{2}\right)\right| < 0.
  \end{equation*}
  Since \(\varphi' \geq 0\) we get \(\HT[\varphi] \leq 0\).
  Furthermore \(\varphi\) is nonconstant and continuous so we have
  \(\varphi' > 0\) in some open set, giving us \(\HT[\varphi] < 0\).

  Now, writing \eqref{eq:bh-wave} as
  \begin{equation*}
    \varphi'(\varphi - c) = \HT[\varphi]
  \end{equation*}
  and using \(\HT[\varphi] < 0\) we get
  \begin{equation*}
    \varphi'(\varphi - c) < 0.
  \end{equation*}
  Which implies that \(\varphi' > 0\) and \(\varphi < c\).
\end{proof}
As a consequence, any continuous, nonconstant, even function which is
nondecreasing on \((-\pi, 0)\) that satisfy \eqref{eq:bh-wave} almost
everywhere must satisfy \(\varphi \leq c\). The maximal possible
height is thus given by \(c\) and due to the function being even and
nondecreasing on \((-\pi, 0)\) the maximal height has to be attained
at \(x = 0\).

Now, the ansatz \(\varphi(x) = c - u(x)\) inserted in
\eqref{eq:bh-wave} gives an equation which does not explicitly depend
on the wave speed \(c\). Indeed, inserting this gives us
\begin{equation}
  \label{eq:main-derivative}
  uu' = -\HT[u].
\end{equation}
Note that a solution of this equation gives a solution of
\eqref{eq:bh-wave} for any wave speed \(c\). This is to be expected
due to the Galilean change of variables
\begin{equation*}
  \varphi \mapsto \varphi + \gamma,\ c \mapsto c + \gamma
\end{equation*}
which leaves \eqref{eq:bh-wave} invariant. In particular, taking \(c\)
equal to the mean of \(u\) gives a zero mean solution. For a highest
wave we expect to have \(\varphi(0) = c\), giving us \(u(0) = 0\).
Integrating \eqref{eq:main-derivative} gives us
\begin{equation}
  \label{eq:main}
  \frac{1}{2}u^{2} = -\HopGeneric[u].
\end{equation}
Here \(\HopGeneric\) is the operator
\begin{equation}
  \label{eq:H}
  \HopGeneric[f](x) = \frac{1}{\pi}p.v. \int_{-\pi}^{\pi}\left(
    \log\left|\sin\left(\frac{x - y}{2}\right)\right| - \log\left|\sin\left(\frac{y}{2}\right)\right|
  \right)f(y)\ dy.
\end{equation}
It is the integral of the Hilbert transform with the constant of
integration taken such that \(\HopGeneric[f](0) = 0\), this ensures
that any solution of \eqref{eq:main} satisfies \(u(0) = 0\). Note that
any solution of \eqref{eq:main} is a solution of
\eqref{eq:main-derivative} and hence gives a solution to
\eqref{eq:bh-wave}.

To reduce the problem to a fixed point problem we write \(u\) as one,
explicit, approximate solution of \eqref{eq:main} and one unknown
term. More precisely we make the ansatz
\begin{equation}
  \label{eq:main-ansatz}
  u(x) = u_{0}(x) + w(x)v(x)
\end{equation}
where \(u_{0}(x)\) is an explicit, carefully chosen, approximate
solution of \eqref{eq:main} and
\(w(x) = x\sqrt{\log\left(1 + \frac{1}{|x|}\right)}\). By taking
\(u_{0}(x) \sim \frac{1}{\pi}|x|\log|x|\), proving Theorem
\ref{thm:main} reduces to proving existence of
\(v \in L^{\infty}(\mathbb{T})\) such that the given ansatz is a
solution of \eqref{eq:main}.

Inserting the ansatz \eqref{eq:main-ansatz} into \eqref{eq:main} gives
us
\begin{equation*}
  \frac{1}{2}(u_{0} + wv)^{2} = -\HopGeneric[u_{0} + wv]\\
  \iff \frac{1}{2}u_{0}^{2} + u_{0}wv + \frac{1}{2}w^{2}v^{2} = -\HopGeneric[u_{0}] - \HopGeneric[wv].
\end{equation*}
By collecting all the linear terms in \(v\) we can write this as
\begin{equation*}
  u_{0}wv + \HopGeneric[wv] = -\HopGeneric[u_{0}] - \frac{1}{2}u_{0}^{2} - \frac{1}{2}w^{2}v^{2}\\
  \iff v + \frac{1}{wu_{0}}\HopGeneric[wv] = -\frac{1}{wu_{0}}\left(\HopGeneric[u_{0}] + \frac{1}{2}u_{0}^{2}\right) - \frac{w}{2u_{0}}v^{2}.
\end{equation*}
Now let \(T\) denote the operator
\begin{equation}
  \label{eq:T_alpha}
  T[v] = -\frac{1}{wu_{0}}\HopGeneric[wv].
\end{equation}
Denote the weighted defect of the approximate solution \(u_{0}(x)\) by
\begin{equation*}
  F(x) = \frac{1}{w(x)u_{0}(x)}\left(\HopGeneric[u_{0}](x) + \frac{1}{2}u_{0}(x)^{2}\right),
\end{equation*}
and let
\begin{equation*}
  N(x) = \frac{w(x)}{2u_{0}(x)}.
\end{equation*}
Then we can write the above as
\begin{equation*}
  (I - T)v = -F - Nv^{2}.
\end{equation*}
Assuming that \(I - T\) is invertible we can rewrite this as
\begin{equation}
  \label{eq:G}
  v = (I - T)^{-1}\left(-F - Nv^{2}\right) =: G[v].
\end{equation}
Hence proving the existence of \(v\) such that that \(u_{0} + wv\) is
a solution to \eqref{eq:main} reduces to proving the existence of a
fixed point of the operator \(G\).

Next we reduce the problem of proving that \(G\) has a fixed point to
checking an inequality for three numbers that depend only on the
choice of \(u_{0}\) and \(w\). We let \(\|T\|\) denote the
\(L^{\infty}(\mathbb{T}) \to L^{\infty}(\mathbb{T})\) norm of a linear
operator \(T\).
\begin{proposition}
  \label{prop:contraction}
  Let \(D_{0} = \|T\|\),
  \(\delta_{0} = \|F\|_{L^{\infty}(\mathbb{T})}\) and
  \(n_{0} = \|N\|_{L^{\infty}(\mathbb{T})}\). If \(D_{0} < 1\) and
  they satisfy the inequality
  \begin{equation*}
    \delta_{0} < \frac{(1 - D_{0})^{2}}{4n_{0}}
  \end{equation*}
  then for
  \begin{equation*}
    \epsilon = \frac{1 - D_{0} - \sqrt{(1 - D_{0})^{2} - 4\delta_{0}n_{0}}}{2n_{0}}
  \end{equation*}
  and
  \begin{equation*}
    X_{\epsilon} = \{v \in L^{\infty}(\mathbb{T}): v(x) = v(-x), \|v\|_{L^{\infty}(\mathbb{T})} \leq \epsilon\}
  \end{equation*}
  we have
  \begin{enumerate}
  \item \(G(X_{\epsilon}) \subseteq X_{\epsilon}\);
  \item
    \(\|G[v] - G[w]\|_{L^{\infty}(\mathbb{T})} \leq k_{0}\|v -
    w\|_{L^{\infty}(\mathbb{T})}\) with \(k_{0} < 1\) for all
    \(v, w \in X_{\epsilon}\).
  \end{enumerate}
\end{proposition}
\begin{proof}
  Using that \(N\) and \(F\) are even it can be checked that
  \begin{equation*}
    G(X_{\epsilon}) \subseteq (I - T)^{-1}X_{\delta_{0} + n_{0}\epsilon^{2}}.
  \end{equation*}
  Since \(\|T\| < 1\) the operator \(I - T\) is invertible and an
  upper bound of the norm of the inverse is given by
  \(\frac{1}{1 - D_{0}}\), moreover \(T\) takes even functions to even
  functions and hence so will \((I - T)^{-1}\). This gives us
  \begin{equation*}
    G(X_{\epsilon}) \subseteq (I - T)^{-1}X_{\delta_{0} + n_{0}\epsilon^{2}} \subseteq X_{\frac{\delta_{0} + n_{0}\epsilon^{2}}{1 - D_{0}}}.
  \end{equation*}
  The choice of \(\epsilon\) then gives
  \begin{equation*}
    \frac{\delta_{0} + n_{0}\epsilon^{2}}{1 - D_{0}} = \epsilon.
  \end{equation*}

  Next we have \(G[v] - G[w] = (I - T)^{-1}(-N(v^{2} - w^{2}))\) and hence
  \begin{equation*}
    \|G[v] - G[w]\|_{L^{\infty}(\mathbb{T})}
    \leq \frac{n_{0}}{1 - D_{0}}\|v^{2} - w^{2}\|_{L^{\infty}(\mathbb{T})}
    \leq \frac{2n_{0}\epsilon}{1 - D_{0}}\|v - w\|_{L^{\infty}(\mathbb{T})}.
  \end{equation*}
  Where \(k_{0} = \frac{2n_{0}\epsilon}{1 - D_{0}} < 1\) since
  \(\epsilon < \frac{1 - D_{0}}{2n_{0}}\).
\end{proof}

\section{Clausen functions}
\label{sec:clausen-functions}
We here give definitions and properties of the Clausen functions that
are used in Sections \ref{sec:construction},
\ref{sec:bounding-n-delta} and \ref{sec:analysis-T}. For more details
about the Clausen functions see Appendix \ref{sec:comp-encl-claus}.

The Clausen functions are related to the polylogarithm through
\begin{align*}
  C_{s}(x) = \frac{1}{2}\left(\polylog_{s}(e^{ix}) + \polylog_{s}(e^{-ix})\right) = \real\left(\polylog_{s}(e^{ix})\right),\\
  S_{s}(x) = \frac{1}{2}\left(\polylog_{s}(e^{ix}) - \polylog_{s}(e^{-ix})\right) = \imag\left(\polylog_{s}(e^{ix})\right).
\end{align*}
They behave nicely with respect to the Hilbert transform, for which we
have
\begin{equation*}
  \HT[C_{s}](x) = S_{s}(x),\quad \HT[S_{s}](x) = -C_{s}(x).
\end{equation*}
In many cases we want to work with functions which are normalised to
be zero at \(x = 0\), for which we use the notation
\begin{equation*}
  \tilde{C}_{s}(x) = C_{s}(x) - C_{s}(0),\quad
  \tilde{C}_{s}^{(\beta)}(x) = C_{s}^{(\beta)}(x) - C_{s}^{(\beta)}(0).
\end{equation*}
With this notation we get for the operator \(\HopGeneric\),
\begin{equation*}
  \HopGeneric[\tilde{C}_{s}](x) = -\tilde{C}_{s - 1}(x),\quad \HopGeneric[S_{s}](x) = -S_{s - 1}(x).
\end{equation*}

From \cite{enciso2018convexity} we have the following expansion for
\(C_{s}(x)\) and \(S_{s}(x)\), valid for \(|x| < 2\pi\),
\begin{align*}
  C_{s}(x) &= \Gamma(1 - s)\sin\left(\frac{\pi}{2}s\right)|x|^{s - 1}
             + \sum_{m = 0}^{\infty} (-1)^{m}\zeta(s - 2m)\frac{x^{2m}}{(2m)!};\\
  S_{s}(x) &= \Gamma(1 - s)\cos\left(\frac{\pi}{2}s\right)\sign(x)|x|^{s - 1}
             + \sum_{m = 0}^{\infty} (-1)^{m}\zeta(s - 2m - 1)\frac{x^{2m + 1}}{(2m + 1)!}.
\end{align*}
For the functions \(C_{s}^{(\beta)}\) and \(S_{s}^{(\beta)}\) we will
mainly make use of \(C_{2}^{(1)}(x)\) and \(C_{3}^{(1)}(x)\), for
which we have the following expansions \cite[Eq.~16]{Bailey2015},
valid for \(|x| < 2\pi\),
\begin{align*}
  C_{2}^{(1)}(x) =& \zeta^{(1)}(2) - \frac{\pi}{2}|x|\log|x| - (\gamma - 1)\frac{\pi}{2}|x|
                   + \sum_{m = 1}^{\infty}(-1)^{m}\zeta^{(1)}(2 - 2m)\frac{x^{2m}}{(2m)!};\\
  C_{3}^{(1)}(x) =& \zeta^{(1)}(3) - \frac{1}{4}x^{2}\log^2|x|
                    + \frac{3 - 2\gamma}{4}x^2\log|x|
                    - \frac{36\gamma - 12\gamma^2 - 24\gamma_{1} - 42 + \pi^{2}}{48}x^2\\
                 &+ \sum_{m = 2}^{\infty}(-1)^{m}\zeta^{(1)}(3 - 2m)\frac{x^{2m}}{(2m)!}.
\end{align*}
Where \(\gamma_{n}\) is the Stieltjes constant and
\(\gamma = \gamma_{0}\). Bounds for the tails are given in Lemmas
\ref{lemma:clausen-tails} and \ref{lemma:clausen-derivative-tails}.

\section{Construction of the approximate solution}
\label{sec:construction}

In this section we give the construction of the approximate solution
\(u_{0}\). As a first step we determine the coefficient for the
leading term in the asymptotic expansion.

\begin{lemma}
  \label{lemma:asymptotic-coefficient}
  Let \(u\) be a solution of \eqref{eq:main} with the asymptotic
  behaviour
  \begin{equation*}
    u(x) = \nu|x|\log|x| + o(|x|\log|x|)
  \end{equation*}
  close to zero, with \(\nu \not= 0\). Then the coefficient is given
  by \(\nu = -\frac{1}{\pi}\).
\end{lemma}
\begin{proof}
  For the asymptotic behaviour of the left hand side in
  \eqref{eq:main} we directly get
  \begin{equation*}
    \frac{1}{2}u(x)^{2} = \frac{\nu^{2}}{2}|x|^{2}\log^{2}|x| + o(|x|^{2}\log^{2}|x|)
  \end{equation*}
  To get the asymptotic behaviour of \(\HopGeneric[u]\) in the right
  hand side we go through the Clausen function
  \(\tilde{C}_{2}^{(1)}(x)\), which has the asymptotic behaviour
  \begin{equation*}
    \tilde{C}^{(1)}_{2}(x) = -\frac{\pi}{2}|x|\log|x| + o(|x|\log|x|).
  \end{equation*}
  Since \(\HopGeneric[\tilde{C}^{(1)}_{2}] = -\tilde{C}^{(1)}_{3}\) we
  have
  \begin{equation*}
    \HopGeneric{\tilde{C}_{2}^{(1)}(x)} = -\frac{1}{4}|x|^{2}\log^{2}|x| + o(|x|^{2}\log^{2}|x|),
  \end{equation*}
  giving us
  \begin{equation*}
    \HopGeneric[u] = \frac{1}{2\pi^{2}}|x|^{2}\log^{2}|x| + o(|x|^{2}\log^{2} |x|).
  \end{equation*}
  For \(\frac{1}{2}u^{2}\) and \(-\HopGeneric[u]\) to have the same
  asymptotic behaviour we must therefore have
  \begin{equation*}
    \frac{\nu^{2}}{2} = -\nu\frac{1}{2\pi},
  \end{equation*}
  which gives us the result.
\end{proof}

In addition to having the correct asymptotic behaviour we want
\(u_{0}\) to be a good approximate solution of \eqref{eq:main}, in the
sense that we want the defect,
\begin{equation*}
  F(x) = \frac{1}{w(x)u_{0}(x)}\left(\HopGeneric[u_{0}](x) + \frac{1}{2}u_{0}(x)^{2}\right),
\end{equation*}
to be small for \(x \in \mathbb{T}\). The hardest part is to make
\(F(x)\) small locally near the singularity at \(x = 0\), this is done
by studying the asymptotic behaviour of
\(\HopGeneric[u_{0}](x) + \frac{1}{2}u_{0}(x)^{2}\). Ones the defect
is sufficiently small near \(x = 0\) it can be made small globally by
adding a suitable trigonometric polynomial.

The construction is similar to that in \cite{enciso2018convexity}, the
main difference is that the asymptotic behaviour is more complicated
in our case. We take \(u_{0}\) to be a combination of three parts:
\begin{enumerate}
\item The first part is the term
  \(-\frac{2}{\pi^{2}}\tilde{C}^{(1)}_{2}\), where the coefficient is
  chosen to give the right asymptotic behaviour according to Lemma
  \ref{lemma:asymptotic-coefficient}.
\item The second part is chosen to make the defect small near
  \(x = 0\) and, similarly to in \cite{enciso2018convexity}, it is
  given by a sum of Clausen functions.
\item The third part is chosen to make the defect small globally and
  is given by a trigonometric polynomial.
\end{enumerate}
More precisely the approximation is given by
\begin{equation}
  \label{eq:u0}
  u_{0}(x) = \frac{2}{\pi^{2}}\tilde{C}_{2}^{(1)}(x)
  + \sum_{j = 1}^{N_{0}} a_{j}\tilde{C}_{s_{j}}(x)
  + \sum_{n = 1}^{N_{1}} b_{n}(\cos(nx) - 1),
\end{equation}
To ensure that the leading asymptotics are determined by
\(\tilde{C}_{2}^{(1)}(x)\) we require that \(s_{j} \geq 2\).

We want to chose the values of \(a_{j}\) and \(s_{j}\) to make the
defect small near \(x = 0\). Taking
\(u_{0} = \frac{2}{\pi^{2}}\tilde{C}_{2}^{(1)}\) gives us that the
leading term in the expansion of
\(\HopGeneric[u_{0}] + \frac{1}{2}u_{0}^{2}\) is of order
\(|x|^{2}\log|x|\). A natural choice would then be to take the next
term as a multiple of \(\tilde{C}_{2}\), for which
\(\HopGeneric[\tilde{C}_{2}]\) behaves like \(|x|^{2}\log|x|\).
However, its contribution to the \(|x|^{2}\log|x|\) term of
\(\HopGeneric[u_{0}]\) and \(\frac{1}{2}u_{0}^{2}\) turns out to
exactly cancel out and we are left with no improvement to the
asymptotic behaviour.

Instead we take inspiration from the fractional KdV equations
\eqref{eq:fkdv-wave}, which for \(\alpha = -1\) reduces to the
Burgers-Hilbert equation. To chose \(a_{j}\) and \(s_{j}\) we study
the limit \(\alpha \to -1^{+}\). For \(-1 < \alpha < 0\) the
fractional KdV equations, like the Burgers-Hilbert equation, admits a
highest cusped traveling wave solution, as recently proved in
\cite{orke22:_highes_kortew}. In a forthcoming
\cite{private-dahne2022fkdv} work we show that the traveling waves
asymptotically at \(x = 0\) behave like
\(c - \nu_{\alpha}|x|^{-\alpha}\), with
\begin{equation*}
  \nu_{\alpha} = \frac{2\Gamma(2\alpha)\cos(\pi\alpha)}{\Gamma(\alpha)\cos\left(\frac{\pi}{2}\alpha\right)}.
\end{equation*}
Following the same approach as in Section \ref{sec:reduct-fixed-point}
proving the existence of a highest cusped wave for the fractional KdV
equations can be reduced to studying the equation
\begin{equation*}
  \frac{1}{2}u^{2} = -\Hop{\alpha}[u],
\end{equation*}
where
\(\Hop{\alpha}[u](x) = |D|^{\alpha}[u](x) - |D|^{\alpha}[u](0)\). By
studying the asymptotic behaviour of
\(\Hop{\alpha}[u] + \frac{1}{2}u^{2}\) and following the same
reasoning as in \cite{enciso2018convexity} a suitable approximation
for this equation is given by
\begin{equation*}
  a_{\alpha,0}\tilde{C}_{1 - \alpha}(x) + \sum_{j = 1}^{N_{0}}a_{\alpha,j}\tilde{C}_{1 - \alpha + jp_{\alpha}}(x)
\end{equation*}
with
\begin{equation*}
  a_{\alpha,0} = \frac{2\Gamma(2\alpha)\cos\left(\pi\alpha\right)}{\Gamma(\alpha)^{2}\cos\left(\frac{\pi}{2}\alpha\right)^{2}}
\end{equation*}
and \(p_{\alpha}\) a solution of
\begin{equation*}
  \Gamma(\alpha)\cos\left(\frac{\pi}{2}\alpha\right)a_{\alpha,0}
  \Gamma(\alpha - p_{\alpha})\cos\left(\frac{\pi}{2}(\alpha - p_{\alpha})\right)
  - \Gamma(2\alpha - p_{\alpha})\cos\left(\frac{\pi}{2}(2\alpha - p_{\alpha})\right) = 0.
\end{equation*}
One can numerically solve for \(a_{\alpha,j}\) by considering the
asymptotic expansion of the defect. As \(\alpha \to -1\) we have
\(a_{\alpha,0} \to -\infty\) and \(p_{\alpha} \to 0\). Numerically one
also sees that \(a_{\alpha,1} \to \infty\), in such a way that
\(a_{\alpha,0} + a_{\alpha,1}\) remains bounded. The other
coefficients, \(a_{\alpha,j}\) for \(j \geq 2\), all remain bounded.
Furthermore we note that as \(\alpha\) approaches \(-1\), the function
\(a_{\alpha,0}(\tilde{C}_{1 - \alpha}(x) - \tilde{C}_{1 - \alpha +
  p_{\alpha}})\) approaches
\(\frac{2}{\pi^{2}}\tilde{C}_{2}^{(1)}(x)\). The remaining part,
\begin{equation*}
  (a_{\alpha,0} + a_{\alpha,1})\tilde{C}_{1 - \alpha + p_{\alpha}}(x) + \sum_{j = 2}^{N_{0}}a_{\alpha,j}\tilde{C}_{1 - \alpha + jp_{\alpha}}(x),
\end{equation*}
is numerically seen to converge to a function as \(\alpha \to -1\), as
long as \(N_{0}\) is increased as we approach \(\alpha = -1\). This
gives a hint on how to choose \(a_{j}\) and \(s_{j}\) for \(u_{0}\),
we take it according to
\begin{equation*}
  (a_{\alpha,0} + a_{\alpha,1})\tilde{C}_{1 - \alpha + p_{\alpha}}(x) + \sum_{j = 2}^{N_{0}}a_{\alpha,j}\tilde{C}_{1 - \alpha + jp_{\alpha}}(x).
\end{equation*}
For this we fix some \(\alpha\) close to \(-1\) and compute
\(p_{\alpha}\) as well as the coefficients \(a_{\alpha,j}\). We then
take \(a_{1} = a_{\alpha,0} + a_{\alpha,1}\), \(a_{j} = a_{\alpha,j}\)
for \(j \geq 2\) and
\begin{equation*}
  s_{j} = 1 - \alpha + jp_{\alpha}.
\end{equation*}
As long as \(\alpha\) is sufficiently close to \(-1\) and \(N_{0}\) is
sufficiently high this gives a small enough defect close to \(x = 0\).
While not obvious from the expression of \(p_{\alpha}\) we do have
\(1 - \alpha + p_{\alpha} > 2\), as required.

The coefficients \(b_{n}\) are then taken to make the defect small
globally. This is done by taking \(N_{1}\) equally spaced points
\(\{x_{n}\}_{1 \leq n \leq N_{1}}\)on the interval \((0, \pi)\) and
numerically solving the non-linear system
\begin{equation*}
  \HopGeneric[u_{0}](x_{n}) + \frac{1}{2}u_{0}(x_{n})^{2} = 0\quad
  \text{for}\quad 1 \leq n \leq N_{1}.
\end{equation*}

In the computations we use \(\alpha = -0.9997\) (giving
\(p_{\alpha} \approx 3.00045 \cdot 10^{-4}\)), \(N_{0} = 1929\) and
\(N_{1} = 16\).
\begin{remark}
  The approximation \(u_{0}\) only needs to be an approximate solution
  to the equation, as opposed to the coefficient for
  \(\tilde{C}_{2}^{(1)}\), which has to match exactly to obtain the
  right asymptotic behaviour.
\end{remark}

With this approximation we get
\begin{equation}
  \label{eq:Hu0}
  \HopGeneric[u_{0}](x) = -\frac{2}{\pi^{2}}\tilde{C}_{3}^{(1)}(x)
  - \sum_{j = 1}^{N_{0}} a_{j}\tilde{C}_{1 + s_{j}}(x)
  - \sum_{n = 1}^{N_{1}} b_{n}\frac{(\cos(nx) - 1)}{n}.
\end{equation}
Furthermore we have the following asymptotic expansions for the
approximation, which we give without proof since they follow directly
from the expansions of the Clausen and trigonometric functions.
\begin{lemma}
  \label{lemma:asymptotic-u0}
  The approximation \(u_{0}\) given by equation \eqref{eq:u0} has the
  asymptotic expansions
  \begin{align*}
    u_{0}(x) =& -\frac{1}{\pi}|x|\log|x| - \frac{\gamma - 1}{\pi}|x|
              + \sum_{j = 1}^{N_{0}}a_{j}^{0}|x|^{-\alpha + jp_{\alpha}}\\
              &+ \sum_{m = 1}^{\infty}\frac{(-1)^{m}}{(2m)!}\left(
                \frac{2}{\pi^{2}}\zeta^{(1)}(2 - 2m)
                + \sum_{j = 1}^{N_{0}}a_{j}\zeta(1 - \alpha + jp_{\alpha} - 2m)
                + \sum_{n = 1}^{N_{1}}b_{n}n^{2m}
                \right)x^{2m}
  \end{align*}
  and
  \begin{align*}
    \HopGeneric[u_{0}(x)] =& -\frac{1}{2\pi^{2}}x^{2}\log^{2}|x| + \frac{3 - 2\gamma}{2\pi^{2}}x^{2}\log|x|
                             -\sum_{j = 1}^{N_{0}}A_{j}^{0}|x|^{1 - \alpha + jp_{\alpha}}\\
                           &+\frac{1}{2}\left(
                             \frac{36\gamma - 12\gamma^2 - 24\gamma_{1} - 42 + \pi^{2}}{12\pi^{2}}
                             + \sum_{j = 1}^{N_{0}}a_{j}\zeta(-\alpha + jp_{\alpha})
                             + \sum_{n = 1}^{N_{1}}b_{n}n
                             \right)x^{2}\\
                           &-\sum_{m = 2}^{\infty}\frac{(-1)^{m}}{(2m)!}\left(
                             \frac{2}{\pi^{2}}\zeta^{(1)}(3 - 2m)
                             + \sum_{j = 1}^{N_{0}}a_{j}\zeta(2 - \alpha + jp_{\alpha} - 2m)
                             + \sum_{n = 1}^{N_{1}}b_{n}n^{2m - 1}
                             \right)x^{2m}
  \end{align*}
  valid for \(|x| < 2\pi\), where
  \begin{align*}
    a_{j}^{0} &= \Gamma(\alpha - jp_{\alpha})\cos\left((\alpha - jp_{\alpha})\frac{\pi}{2}\right)a_{j};\\
    A_{j}^{0} &= \Gamma(\alpha - 1 - jp_{\alpha})\cos\left((\alpha - 1 - jp_{\alpha})\frac{\pi}{2}\right)a_{j}.
  \end{align*}
\end{lemma}

\section{Bounding \(n_{0}\) and \(\delta_{0}\)}
\label{sec:bounding-n-delta}
Recall that
\begin{align*}
  n_{0} &:= \|N\|_{L^{\infty}(\mathbb{T})}
          = \sup_{x \in [0, \pi]} \left|N(x)\right|,\\
  \delta_{0} &:= \|F\|_{L^{\infty}(\mathbb{T})}
               = \sup_{x \in [0, \pi]} |F(x)|,
\end{align*}
where
\begin{equation*}
  N(x) = \frac{x\sqrt{\log\left(1 + 1 / x\right)}}{2u_{0}(x)}
  \quad \text{and} \quad
  F(x) = \frac{\HopGeneric[u_{0}](x) + \frac{1}{2}u_{0}(x)^{2}}{x\sqrt{\log(1 + 1 / x)}u_{0}(x)}.
\end{equation*}
For fixed \(x > 0\) which is not too small we can compute accurate
enclosures of \(N(x)\) as well as \(F(x)\) using interval arithmetic.
This allows us to compute
\begin{equation*}
  \sup_{x \in [\epsilon, \pi]} \left|N(x)\right|
  \quad \text{and} \quad
  \sup_{x \in [\epsilon, \pi]} \left|F(x)\right|
\end{equation*}
for some fixed \(\epsilon > 0\). As \(x \to 0\) the numerators and
denominators tend to zero in both \(N\) and \(F\) and we need to
handle the removable singularities. We start with the following lemma:
\begin{lemma}
  \label{lemma:asymptotic-inv-u0}
  The function \(\frac{-|x|\log|x|}{u_{0}(x)}\) is positive and
  bounded at \(x = 0\) and for \(|x| < 1\) it has the expansion
  \begin{align*}
    \frac{|x|\log|x|}{u_{0}(x)} =& -\Bigg(
                                   -\frac{1}{\pi} - \frac{\gamma - 1}{\pi}\frac{1}{\log|x|}
                                   + \sum_{j = 1}^{N_{0}}a_{j}^{0}\frac{|x|^{-1 - \alpha + jp_{\alpha}}}{\log|x|}\\
                                 &+ \sum_{m = 1}^{\infty}\frac{(-1)^{m}}{(2m)!}\left(
                                   \frac{2}{\pi^{2}}\zeta^{(1)}(2 - 2m)
                                   + \sum_{j = 1}^{N_{0}}a_{j}\zeta(1 - \alpha + jp_{\alpha} - 2m)
                                   + \sum_{n = 1}^{N_{1}}b_{n}n^{2m}
                                   \right)\frac{|x|^{2m - 1}}{\log|x|}
                                   \Bigg)^{-1}.
  \end{align*}
\end{lemma}
\begin{proof}
  The expansion follows directly from the expansion of \(u_{0}\) given
  in Lemma \ref{lemma:asymptotic-u0} and canceling the \(|x|\log|x|\)
  factor.

  The function \(\frac{1}{\log|x|}\) goes to zero at \(x = 0\) and is
  bounded nearby and by construction \(\alpha\) and \(p\) are taken
  such that \(-1 - \alpha + jp_{\alpha} > 0\) for \(j \geq 1\). This
  means that all terms except \(\frac{1}{\pi}\) tend to zero as
  \(x \to 0\). The value at \(x = 0\) will hence be given by
  \(-\left(-\frac{1}{\pi}\right)^{-1} = \pi\).
\end{proof}
In light of this lemma we do the split
\begin{equation*}
  N(x) = \frac{\sqrt{\log\left(1 + 1 / x\right)}}{-2\log x} \cdot \frac{-x\log x}{u_{0}(x)}.
\end{equation*}
The second factor can then be bounded using the lemma. For the first
factor it is enough to notice that it tends to zero at \(x = 0\) and
is increasing in \(x\) on \((0, 1)\). For \(F(x)\) we do the split
\begin{equation*}
  F(x) = -\frac{1}{\sqrt{\log(1 + 1 / x)}}
  \cdot \frac{-x\log x}{u_{0}(x)}
  \cdot \frac{\HopGeneric[u_{0}](x) + \frac{1}{2}u_{0}(x)^{2}}{x^{2}\log x}
\end{equation*}
The first factor can be handled by noticing that it tends to zero at
\(x = 0\) and is increasing in \(x\), the second factor is handled
using the above lemma. What remains is to handle the third factor,
which is done in the following lemma:
\begin{lemma}
  \label{lemma:asymptotic-delta}
  The function
  \begin{equation*}
    \frac{\HopGeneric[u_{0}](x) + \frac{1}{2}u_{0}(x)^{2}}{x^{2}\log x}
  \end{equation*}
  is non-zero and bounded at \(x = 0\) and for \(|x| < 1\) it has the
  expansion
  \begin{multline*}
    \frac{\HopGeneric[u_{0}](x) + \frac{1}{2}u_{0}(x)^{2}}{x^{2}\log x}
    =  \frac{1}{2\pi^{2}}\\
    + \frac{1}{2}\left(
      \frac{(\gamma - 1)^{2}}{\pi^{2}}
      + \frac{36\gamma - 12\gamma^2 - 24\gamma_{1} - 42 + \pi^{2}}{12\pi^{2}}
      + \sum_{j = 1}^{N_{0}}a_{j}\zeta(-\alpha + jp_{\alpha})
      + \sum_{n = 1}^{N_{1}}b_{n}n
    \right)\frac{1}{\log|x|}\\
    - \frac{1}{\pi}\left(1 + \frac{\gamma - 1}{\log|x|}\right)\sum_{j = 1}^{N_{0}}a_{j}^{0}|x|^{-1 - \alpha + jp_{\alpha}}
    + \frac{1}{2\log|x|}\left(\sum_{j = 1}^{N_{0}}a_{j}^{0}|x|^{-1 - \alpha + jp_{\alpha}}\right)^{2}
    - \frac{1}{\log|x|}\sum_{j = 1}^{N_{0}}A_{j}^{0}|x|^{-1 - \alpha + jp_{\alpha}}\\
    - \frac{1}{\pi}\left(1 + \frac{\gamma - 1}{\log|x|}\right)\frac{S_{1}}{|x|}
    + \frac{1}{\log|x|}\left(\sum_{j = 1}^{N_{0}}a_{j}^{0}|x|^{-1 - \alpha + jp_{\alpha}}\right)\frac{S_{1}}{|x|}
    + \frac{1}{2}\frac{S_{1}^{2}}{x^{2}\log|x|} - \frac{S_{2}}{x^{2}\log|x|}.
  \end{multline*}
  where
  \begin{align*}
    S_{1} &= \sum_{m = 1}^{\infty}\frac{(-1)^{m}}{(2m)!}\left(
            \frac{2}{\pi^{2}}\zeta^{(1)}(2 - 2m)
            + \sum_{j = 1}^{N_{0}}a_{j}\zeta(1 - \alpha + jp_{\alpha} - 2m)
            + \sum_{n = 1}^{N_{1}}b_{n}n^{2m}
            \right)|x|^{2m}\\
    S_{2} &= \sum_{m = 2}^{\infty}\frac{(-1)^{m}}{(2m)!}\left(
            \frac{2}{\pi^{2}}\zeta^{(1)}(3 - 2m)
            + \sum_{j = 1}^{N_{0}}a_{j}\zeta(2 - \alpha + jp_{\alpha} - 2m)
            + \sum_{n = 1}^{N_{1}}b_{n}n^{2m - 1}
            \right)x^{2m}
  \end{align*}
\end{lemma}
\begin{proof}
  From the expansions of \(u_{0}\) in Lemma \ref{lemma:asymptotic-u0}
  we get
  \begin{multline*}
    \frac{1}{2}u_{0}(x)^{2} = \frac{1}{2\pi^{2}}x^{2}\log^{2}|x|
    + \frac{\gamma - 1}{\pi^{2}}x^{2}\log |x|
    + \frac{(\gamma - 1)^{2}}{2\pi^{2}}x^{2}\\
    - \frac{1}{\pi}(\log|x| + \gamma - 1)\sum_{j = 1}^{N_{0}}a_{j}^{0}|x|^{1 - \alpha + jp_{\alpha}}
    + \frac{1}{2}\left(\sum_{j = 1}^{N_{0}}a_{j}^{0}|x|^{-\alpha + jp_{\alpha}}\right)^{2}\\
    - \frac{1}{\pi}(\log|x| + \gamma - 1)|x|S_{1}
    + \left(\sum_{j = 1}^{N_{0}}a_{j}^{0}|x|^{-\alpha + jp_{\alpha}}\right)S_{1}
    + \frac{1}{2}S_{1}^{2}.
  \end{multline*}
  This together with the expansion for \(\HopGeneric[u_{0}]\) gives us
  \begin{multline*}
    \HopGeneric[u_{0}](x) + \frac{1}{2}u_{0}(x)^{2}
    =  \frac{1}{2\pi^{2}}x^{2}\log |x|\\
    + \frac{1}{2}\left(
      \frac{(\gamma - 1)^{2}}{\pi^{2}}
      + \frac{36\gamma - 12\gamma^2 - 24\gamma_{1} - 42 + \pi^{2}}{12\pi^{2}}
      + \sum_{j = 1}^{N_{0}}a_{j}\zeta(-\alpha + jp_{\alpha})
      + \sum_{n = 1}^{N_{1}}b_{n}n
    \right)x^{2}\\
    - \frac{1}{\pi}(\log|x| + \gamma - 1)\sum_{j = 1}^{N_{0}}a_{j}^{0}|x|^{1 - \alpha + jp_{\alpha}}
    + \frac{1}{2}\left(\sum_{j = 1}^{N_{0}}a_{j}^{0}|x|^{-\alpha + jp_{\alpha}}\right)^{2}
    - \sum_{j = 1}^{N_{0}}A_{j}^{0}|x|^{1 - \alpha + jp_{\alpha}}\\
    - \frac{1}{\pi}(\log|x| + \gamma - 1)|x|S_{1}
    + \left(\sum_{j = 1}^{N_{0}}a_{j}^{0}|x|^{-\alpha + jp_{\alpha}}\right)S_{1}
    + \frac{1}{2}S_{1}^{2} - S_{2}.
  \end{multline*}
  Division by \(x^{2}\log|x|\) gives the required expansion. Note that
  the two leading terms in \(\frac{1}{2}u_{0}\) and
  \(\HopGeneric[u_{0}]\) exactly cancel out, this is needed for the
  result to be bounded near \(x = 0\).
\end{proof}

\section{Analysis of \(T\) and bounding \(D_{0}\)}
\label{sec:analysis-T}
In this section we give more details about the operator \(T\) defined
by
\begin{equation*}
  T[v] = -\frac{1}{wu_{0}}\HopGeneric[w v]
\end{equation*}
and show how to bound \(D_{0} := \|T\|\).

For an even function \(v(x)\) and \(0 < x < \pi\) we can write
\eqref{eq:H} as
\begin{equation}
  \label{eq:H-integral-representation}
  \HopGeneric[v](x) = \frac{1}{\pi}\int_{0}^{\pi}\log\left(\frac{\sin(|x - y| / 2)\sin((x + y) / 2)}{\sin(y / 2)^{2}}\right)v(y)\ dy.
\end{equation}
Using that \(C_{1}(x) = -\log(2\sin(|x| / 2))\) this can alternatively
be written as
\begin{equation}
  \label{eq:H-integral-representation-clausen}
  \HopGeneric[v](x) = \frac{1}{\pi}\int_{0}^{\pi}-\left(C_{1}(x - y) + C_{1}(x + y) - 2C_{1}(y)\right)v(y)\ dy.
\end{equation}
From \eqref{eq:H-integral-representation} we have
\begin{equation*}
  T[v](x) = \frac{1}{\pi w(x) u_{0}(x)}
  \int_{0}^{\pi}\log\left(\frac{\sin(|x - y| / 2)\sin((x + y) / 2)}{\sin(y / 2)^{2}}\right)w(y)v(y)\ dy.
\end{equation*}

The norm of \(T\) is then given by
\begin{equation}
  \label{eq:D}
  D_{0} = \|T\| =
  \sup_{x \in [0, \pi]} \left|\frac{1}{\pi w(x) u_{0}(x)}
    \int_{0}^{\pi}\left|\log\left(\frac{\sin(|x - y| / 2)\sin((x + y) / 2)}{\sin(y / 2)^{2}}\right)w(y)\right|\ dy
  \right|.
\end{equation}
Let
\begin{equation*}
  I(x, y) = \log\left(\frac{\sin(|x - y| / 2)\sin((x + y) / 2)}{\sin(y / 2)^{2}}\right),
\end{equation*}
and
\begin{equation*}
  U(x) = \int_{0}^{\pi}|I(x, y)w(y)|\ dy
  = \int_{0}^{\pi}|I(x, y)|y\sqrt{\log(1 + 1 / y)}\ dy.
\end{equation*}
We are then interested in computing
\begin{equation*}
  D_{0} = \|T\| = \sup_{x \in [0, \pi]} \left|\frac{U(x)}{\pi w(x) u_{0}(x)}\right|.
\end{equation*}
We use the notation
\begin{equation*}
  \mathcal{T}(x) = \frac{U(x)}{\pi w(x) u_{0}(x)}.
\end{equation*}

The integrand of \(U(x)\) has a singularity at \(y = x\). It is
therefore natural to split the integral into two parts
\begin{align*}
  U_{1}(x) &= \int_{0}^{x}|I(x, y)|y\sqrt{\log(1 + 1 / y)}\ dy = x^{2} \int_{0}^{1}|\hat{I}(x, t)|t\sqrt{\log(1 + 1 / (xt))}\ dt,\\
  U_{2}(x) &= \int_{x}^{\pi}|I(x, y)|y\sqrt{\log(1 + 1 / y)}\ dy,
\end{align*}
where
\begin{equation*}
  \hat{I}(x, t) = \log\left(\frac{\sin(|x(1 - t)| / 2)\sin(x(1 + t) / 2)}{\sin(xt / 2)^{2}}\right).
\end{equation*}
The following two lemmas give information about the sign of
\(\hat{I}(x, t)\) and \(I(x, y)\), allowing us to remove the absolute
value.
\begin{lemma}
  \label{lemma:I-0-x}
  For all \(x \in (0, \pi)\) the function \(\hat{I}(x, t)\) is
  decreasing and continuous in \(t\) for \(t \in (0, 1)\) and has the
  limits
  \begin{align*}
    \lim_{t \to 0^{+}} \hat{I}(x, t) &= \infty,\\
    \lim_{t \to 1^{-}} \hat{I}(x, t) &= -\infty.
  \end{align*}
  Moreover, the unique root, \(r_{x}\), is decreasing in \(x\) and
  satisfies the inequality
  \begin{equation*}
    \frac{1}{2} < r_{x} < \frac{1}{\sqrt{2}}.
  \end{equation*}
\end{lemma}
\begin{proof}
  The left and right limits are easily checked and it is clear that
  the function is continuous in \(t\) on the interval. To show that it
  is decreasing in \(t\) we split the \(\log\) into three terms and
  differentiate, giving us
  \begin{equation*}
    \frac{d}{dt}\hat{I}(x, t)
    = \frac{1}{2}x\left(-\cot(x(1 - t) / 2) + \cot(x(1 + t) / 2) - 2\cot(xt / 2)\right).
  \end{equation*}
  We want to prove that this is negative. Note that
  \(0 < xt / 2 < \pi / 2\) and that \(\cot\) is positive on the
  interval \((0, \pi / 2)\), hence it is enough to check
  \begin{equation*}
    -\cot(x(1 - t) / 2) + \cot(x(1 + t) / 2) < 0,
  \end{equation*}
  which follows immediately from the monotonicity of \(\cot\) on the
  interval \([0, \pi]\). This proves the existence of a unique root
  \(r_{x}\) on the interval \((0, 1)\).

  To prove that \(r_{x}\) is decreasing in \(x\) it is enough to prove
  that \(\hat{I}(x, t)\) is decreasing in \(x\). Differentiating with
  respect to \(x\) gives us
  \begin{equation*}
    \frac{d}{dx}\hat{I}(x, t)
    = \frac{1}{2}\left(\left((1 - t)\cot(x(1 - t) / 2) + (1 + t)\cot(x(1 + t) / 2) - 2t\cot(xt / 2)\right)\right),
  \end{equation*}
  which we want to prove is negative. Letting
  \(g(t) = t \cot(xt / 2)\) and using that \(g(1) \leq g(t)\) for
  \(t \in (0, 1)\) it is enough to show that
  \begin{equation*}
    g(1 - t) + g(1 + t) \leq 2g(1).
  \end{equation*}
  This holds since \(g\) is concave, as
  \begin{equation*}
    g''(t) = \frac{x}{2\sin^{2}(xt / 2)}(xt\cot(xt / 2) - 2) < 0
  \end{equation*}
  following from \(\frac{x}{2\sin^{2}(xt / 2)} > 0\) and
  \(xt\cot(xt / 2) \leq 2\).

  Finally, to see that \(\frac{1}{2} < r_{x} < \frac{1}{\sqrt{2}}\) it
  is enough to check that for \(x = \pi\) the root is given by
  \(t = \frac{1}{2}\) and that
  \begin{equation*}
    \lim_{x \to 0^{+}} \hat{I}(x, t) = \lim_{x \to 0^{+}} \log\left(\frac{x(1 - t) / 2 \cdot x(1 + t) / 2}{(xt / 2)^{2}}\right)
    = -\log\left(\frac{1}{t^{2}} - 1\right)
  \end{equation*}
  together with the monotonicity in \(x\).
\end{proof}
\begin{lemma}
  \label{lemma:I-x-pi}
  For \(x \in (0, \pi)\) we have \(I(x, y) < 0\) for all
  \(y \in (x, \pi)\).
\end{lemma}
\begin{proof}
  The function \(f(y) = -\log(\sin(y / 2))\) is strictly convex on the
  domain \((0, 2\pi)\). This follows from the fact that \(-\log(y)\)
  is strictly convex and decreasing and \(\sin(y / 2)\) is strictly
  concave on the interval. It immediately follows that
  \begin{equation*}
    I(x, y) = \log((\sin((y - x) / 2))) + \log((\sin((y + x) / 2))) - 2\log((\sin(y / 2))) < 0.
  \end{equation*}
\end{proof}
With these two lemmas we can slightly simplify \(U_{1}\) and
\(U_{2}\), we get
\begin{multline*}
  U_{1}(x) = x^{2}\int_{0}^{r_{x}} \log\left(\frac{\sin(x(1 - t) / 2)\sin(x(1 + t) / 2)}{\sin(xt / 2)^{2}}\right)t\sqrt{\log(1 + 1 / (xt))}\ dt\\
  -x^{2}\int_{r_{x}}^{1} \log\left(\frac{\sin(x(1 - t) / 2)\sin(x(1 + t) / 2)}{\sin(xt / 2)^{2}}\right)t\sqrt{\log(1 + 1 / (xt))}\ dt\\
  = x^{2}(U_{1,1}(x) + U_{1,2}(x))
\end{multline*}
and
\begin{equation*}
  U_{2}(x) = -\int_{x}^{\pi}\log\left(\frac{\sin((y - x) / 2)\sin((x + y) / 2)}{\sin(y / 2)^{2}}\right)y\sqrt{\log(1 + 1 / y)}\ dy
\end{equation*}

Similarly to in the previous section we divide the interval
\([0, \pi]\), on which we take the supremum, into two parts,
\([0, \epsilon]\) and \([\epsilon, \pi]\).

For the interval \([\epsilon, \pi]\) we split \(\mathcal{T}(x)\) as
\begin{equation*}
  \mathcal{T}(x) = \frac{U(x)}{\pi x \sqrt{\log(1 + 1 / x)} u_{0}(x)}
  = \frac{1}{u_{0}(x)}\left(
    \frac{x(U_{1,1}(x) + U_{1,2}(x))}{\pi \sqrt{\log(1 + 1 / x)}}
    + \frac{U_{2}(x)}{\pi x \sqrt{\log(1 + 1 / x)}}
  \right).
\end{equation*}
See Appendix \ref{sec:rigorous-integration} for details on how
\(U_{1,1}\), \(U_{1,2}\) and \(U_{2}\) are computed.

For the interval \([0, \epsilon]\) write \(\mathcal{T}(x)\) as
\begin{equation*}
  \mathcal{T}(x) = \frac{1}{\pi} \cdot \frac{-x\log x}{u_{0}(x)} \cdot \left(
    \frac{U_{1}(x)}{-x^{2}\log x \sqrt{\log(1 + 1 / x)}}
    + \frac{U_{2}(x)}{-x^{2}\log x \sqrt{\log(1 + 1 / x)}}
  \right).
\end{equation*}
The factor \(\frac{-x\log x}{u_{0}(x)}\) is handled by Lemma
\ref{lemma:asymptotic-inv-u0} as before. For the two remaining terms
we have the following two lemmas:
\begin{lemma}
  \label{lemma:asymptotic-U1}
  For \(x \in [0, \epsilon]\) with \(\epsilon < 1\) we have
  \begin{multline*}
    \frac{U_{1}(x)}{-x^{2}\log x \sqrt{\log(1 + 1 / x)}} \leq
    \frac{1}{\sqrt{\log(1 + 1/x)}}\Bigg(
    \frac{\log 2}{\sqrt{\log(1/x)}}
    + \frac{c_{1} + \log 2\sqrt{\log(1 + x)}}{\log(1/x)}\\
    + \frac{3R_{1}}{8}x^{2}\left(
      \frac{2}{\sqrt{\log(1 / x)}}
      + \frac{\sqrt{\pi / 2} + 2\sqrt{\log(1 + x)}}{\log(1 / x)}
    \right)
    \Bigg)
  \end{multline*}
  where
  \begin{align*}
    c_{1} &= \int_{0}^{1}|\log(1 / t^{2} - 1)|t\sqrt{\log(1 / t)}\ dt,\\
    R_{1} &= \sup_{y \in [0, \epsilon]} \frac{1}{2}\left|\frac{d^{2}}{dy^{2}}\log\left(\frac{\sin(y)}{y}\right)\right|.
  \end{align*}
\end{lemma}
\begin{proof}
  As a first step we split \(\hat{I}(x, t)\) into one main term and
  one remainder term. We can write \(\hat{I}(x, t)\) as
  \begin{equation*}
    \hat{I}(x, t) = \log\left(\sin\left(\frac{x(1 - t)}{2}\right)\right)
    + \log\left(\sin\left(\frac{x(1 + t)}{2}\right)\right)
    - 2\log\left(\sin\left(\frac{xt}{2}\right)\right).
  \end{equation*}
  Using that
  \begin{equation*}
    \log\left(\sin\left(\frac{x(1 - t)}{2}\right)\right)
    = \log\left(\sinc\left(\frac{x(1 - t)}{2}\right)\right) + \log\left(\frac{x(1 - t)}{2}\right),
  \end{equation*}
  where \(\sinc x = \frac{\sin x}{x}\), and similarly for the other
  log-sin terms, we can split \(\hat{I}(x, t)\) as
  \begin{multline*}
    \hat{I}(x, t) = \log\left(\frac{x(1 - t)}{2}\right) + \log\left(\frac{x(1 + t)}{2}\right) - 2\log\left(\frac{xt}{2}\right)\\
    + \log\left(\sinc\left(\frac{x(1 - t)}{2}\right)\right) + \log\left(\sinc\left(\frac{x(1 + t)}{2}\right)\right) - 2\log\left(\sinc\left(\frac{xt}{2}\right)\right).
  \end{multline*}
  where
  \begin{equation*}
    \log\left(\frac{x(1 - t)}{2}\right) + \log\left(\frac{x(1 + t)}{2}\right) - 2\log\left(\frac{xt}{2}\right)
    = \log(1 - t) + \log(1 + t) - 2\log(t) = \log(1 / t^{2} - 1).
  \end{equation*}
  Note that for small \(x\), \(\sinc(x)\) is close to one, and the
  corresponding log-terms will therefore be small. We split
  \(U_{1}(x)\) as
  \begin{multline*}
    U_{1}(x) \leq x^{2}\int_{0}^{1}\left|\log(1 / t^{2} - 1)\right|t\sqrt{\log(1 + 1/(xt))}\ dt\\
    + x^{2}\int_{0}^{1}\left|
      \log\left(\sinc\left(\frac{x(1 - t)}{2}\right)\right) + \log\left(\sinc\left(\frac{x(1 + t)}{2}\right)\right) - 2\log\left(\sinc\left(\frac{xt}{2}\right)\right)
    \right|t\sqrt{\log(1 + 1/(xt))}\ dt\\
    = U_{1,m}(x) + U_{1,r}(x).
  \end{multline*}

  Focusing in \(U_{1,m}\) we note that for \(t \in (0, 1)\) we have
  \begin{multline}
    \label{eq:norm-BH-sqrt-inequality}
    \sqrt{\log(1 + 1/(xt))} = \sqrt{\log(1 / x) + \log(1 / t) + \log(1 + xt)}\\
    \leq \sqrt{\log(1 / x)} + \sqrt{\log(1 / t)} + \sqrt{\log(1 + xt)}\\
    \leq \sqrt{\log(1 / x)} + \sqrt{\log(1 / t)} + \sqrt{\log(1 + x)}.
  \end{multline}
  Hence
  \begin{multline*}
    U_{1,m}(x) \leq x^{2}\Bigg(
      \sqrt{\log(1 / x)} \int_{0}^{1}\left|\log(1 / t^{2} - 1)\right|t\ dt
      + \int_{0}^{1}\left|\log(1 / t^{2} - 1)\right|t\sqrt{\log(1 / t)}\ dt\\
      + \sqrt{\log(1 + x)} \int_{0}^{1}\left|\log(1 / t^{2} - 1)\right|t\ dt
    \Bigg).
  \end{multline*}
  We have
  \(\int_{0}^{1}\left|\log(1 / t^{2} - 1)\right|t\ dt = \log 2\) and
  if we let
  \begin{align*}
    c_{1} &= \int_{0}^{1}\left|\log(1 / t^{2} - 1)\right|t\sqrt{\log(1 / t)}\ dt
  \end{align*}
  this gives us
  \begin{equation*}
    U_{1,m}(x) \leq x^{2}\left(\sqrt{\log(1 / x)}\log 2 + c_{1} + \log 2 \sqrt{\log(1 + x)}\right).
  \end{equation*}

  For \(U_{1,r}\) we will give a uniform bound of
  \begin{equation*}
    \log\left(\sinc\left(\frac{x(1 - t)}{2}\right)\right) + \log\left(\sinc\left(\frac{x(1 + t)}{2}\right)\right) - 2\log\left(\sinc\left(\frac{xt}{2}\right)\right)
  \end{equation*}
  in \(t\) on the interval \([0, 1]\) and use this to simplify the
  integrand. Note that \(\frac{x(1 - t)}{2}\), \(\frac{x(1 + t)}{2}\)
  and \(\frac{xt}{2}\) all lie on the interval \([0, x]\). The
  function \(\log(\sinc(y))\) is analytic around \(y = 0\) and the
  first two terms in the Taylor expansion are zero, by Taylor's
  Theorem we hence have
  \begin{equation*}
    |\log(\sinc(y))| \leq R_{1}y^{2}
  \end{equation*}
  where
  \begin{equation*}
    R_{1} = \sup_{y \in [0, x]} \frac{1}{2}\left|\frac{d}{dy^{2}}\log(\sinc(y))\right|.
  \end{equation*}
  This gives us
  \begin{multline*}
    \left|\log\left(\sinc\left(\frac{x(1 - t)}{2}\right)\right)
      + \log\left(\sinc\left(\frac{x(1 + t)}{2}\right)\right)
      - 2\log\left(\sinc\left(\frac{xt}{2}\right)\right)\right|\\
    \leq R_{1}x^{2}\left(\frac{(1 - t)^{2}}{4} + \frac{(1 + t)^{2}}{4} + \frac{t^{2}}{2}\right)
    \leq \frac{3R_{1}}{2}x^{2}.
  \end{multline*}
  With this we get
  \begin{equation*}
    U_{1,r}(x) \leq \frac{3R_{1}}{2}x^{4}\int_{0}^{1}t\sqrt{\log(1 + 1/(xt))}\ dt.
  \end{equation*}
  Using the inequality \eqref{eq:norm-BH-sqrt-inequality} we have
  \begin{equation*}
    \int_{0}^{1}t\sqrt{\log(1 + 1/(xt))}\ dt
    \leq \sqrt{\log(1 / x)}\int_{0}^{1}t\ dt\\
    + \int_{0}^{1}t\sqrt{\log(1 / t)}\ dt
    + \sqrt{\log(1 + x)}\int_{0}^{1}t\ dt.
  \end{equation*}
  With
  \begin{equation*}
    \int_{0}^{1}t\ dt = \frac{1}{2},\quad \int_{0}^{1}t\sqrt{\log(1 / t)}\ dt = \frac{\sqrt{\pi / 2}}{4}
  \end{equation*}
  this gives
  \begin{equation*}
    U_{1,r}(x) \leq \frac{3R_{1}}{8}x^{4}\left(
      2\sqrt{\log(1 / x)}
      + \sqrt{\pi / 2}
      + 2\sqrt{\log(1 + x)}
    \right)
  \end{equation*}

  Combining the bound for \(U_{1,m}\) with that for \(U_{1,r}\) gives
  us the result.
\end{proof}

\begin{lemma}
  \label{lemma:asymptotic-U2}
  For \(x \in [0, \epsilon]\) with \(\epsilon < \frac{1}{2}\) we have
  \begin{multline*}
    \frac{U_{2}(x)}{-x^{2}\log x \sqrt{\log(1 + 1 / x)}} \leq
    \frac{\log\left(\frac{16}{3\sqrt{3}}\right)}{\log(1/x)}\\
    +\Bigg(
    \frac{\frac{2}{3}\log(1/(2x))^{\frac{3}{2}}}{\log(1/x)\sqrt{\log(1 + 1/x)}}
    + \frac{R_{2}\sqrt{\log(1/(2x))}}{8\log(1/x)\sqrt{\log(1 + 1/x)}}
    - \frac{R_{2}(1 - 4x^{2})}{16\sqrt{\log(1 / (2x))}\log(1/x)\sqrt{\log(1 + 1/x)}}\\
    + \frac{\sqrt{\log 2}}{\sqrt{\log(1 + 1/x)}}
    - \frac{\log(2)^{\frac{3}{2}}}{\log(1/x)\sqrt{\log(1 + 1/x)}}
    + \frac{R_{2}\sqrt{\log 2}(1 - 4x^{2})}{8\log(1/x)\sqrt{\log(1 + 1/x)}}
    \Bigg)\\
    + \frac{\sqrt{\log 2}}{\log(1/x)\sqrt{\log(1 + 1/x)}}\left(
      \frac{1}{2}\log\left(\frac{\pi^{2} - x^{2}}{1 - x^{2}}\right)
      + \frac{\log(1 - x^{2})}{2x^{2}}
      - \frac{\log\left(1 - \frac{x^{2}}{\pi^{2}}\right)}{2x^{2}}
    \right)
    + \frac{D_{1}c_{2}}{\log(1/x)\sqrt{\log(1 + 1/x)}}
  \end{multline*}
  where
  \begin{align*}
    c_{2} &= \int_{0}^{\pi}y\sqrt{\log(1 + 1/y)}\ dy,\\
    R_{2} &= \sup_{y \in [0, 1/4]} \frac{1}{2}\left|\frac{d^{2}}{dy^{2}}\log(1 - y)\right|,\\
    D_{1} &= \sup_{x \in [0, \epsilon]} -\frac{\log\left(\sinc\left(\frac{\pi - x}{2}\right)\right) + \log\left(\sinc\left(\frac{\pi + x}{2}\right)\right) - 2\log\left(\sinc\left(\frac{\pi}{2}\right)\right)}{x^{2}}.
  \end{align*}
\end{lemma}
\begin{proof}
  Similarly as in the previous lemma we split \(I(x, y)\) into one
  main term and one remainder term. In this case we get
  \begin{align*}
    I(x, y) =& \log((y - x) / 2) + \log((y + x) / 2) - 2\log(y / 2)\\
             &+ \log(\sinc((y - x) / 2)) + \log(\sinc((y + x) / 2)) - 2\log(\sinc(y / 2)).
  \end{align*}
  where
  \begin{equation*}
    \log((y - x) / 2) + \log((y + x) / 2) - 2\log(y / 2) = \log(1 - (x / y)^{2}).
  \end{equation*}
  Note that the log-sinc terms are not individually small since \(y\)
  is not in general small, but for small values of \(x\) they mostly
  cancel out. We split \(U_{2}\) as
  \begin{multline*}
    U_{2}(x) = \int_{x}^{\pi}-\log(1 - (x / y)^{2})y\sqrt{\log(1 + 1 / y)}\ dy\\
    + \int_{x}^{\pi}-(\log(\sinc((y - x) / 2)) + \log(\sinc((y + x) / 2)) - 2\log(\sinc(y / 2)))y\sqrt{\log(1 + 1 / y)}\ dy\\
    = U_{2,m}(x) + U_{2,r}(x)
  \end{multline*}

  Due to the occurrence of \(x / y\) in \(U_{2,m}\) it is natural to
  switch coordinates to \(t = y / x\), as was done for \(U_{1}\). This
  gives us
  \begin{equation*}
    U_{2,m}(x) = x^{2}\int_{1}^{\pi / x}-\log(1 - 1 / t^{2})t\sqrt{\log(1 + 1/(xt))}\ dt.
  \end{equation*}
  Next we split the interval \([1, \pi / x]\) into three parts,
  \([1, 2]\), \([2, 1 / x]\) and \([1 / x, \pi / x]\), and treat each
  of them separately. Let
  \begin{align*}
    U_{2,m,1}(x) &= x^{2}\int_{1}^{2}-\log(1 - 1 / t^{2})t\sqrt{\log(1 + 1/(xt))}\ dt,\\
    U_{2,m,2}(x) &= x^{2}\int_{2}^{1 / x}-\log(1 - 1 / t^{2})t\sqrt{\log(1 + 1/(xt))}\ dt,\\
    U_{2,m,3}(x) &= x^{2}\int_{1 / x}^{\pi / x}-\log(1 - 1 / t^{2})t\sqrt{\log(1 + 1/(xt))}\ dt.
  \end{align*}

  For \(U_{2,m,1}\) we note that \(\sqrt{\log(1 + 1/(xt))}\) is
  decreasing in \(t\) and hence upper bounded by the value at
  \(t = 1\), which is \(\sqrt{\log(1 + 1 / x)}\). We therefore have
  \begin{equation*}
    U_{2,m,1}(x) \leq x^{2}\sqrt{\log(1 + 1 / x)} \int_{1}^{2}-\log(1 - 1 / t^{2}) t\ dt
    = x^{2}\sqrt{\log(1 + 1 / x)} \log\left(\frac{16}{3\sqrt{3}}\right).
  \end{equation*}

  For \(U_{2,m,2}\) we use a Taylor expansion of
  \(\log(1 - 1 / t^{2})\) at \(t = \infty\) and explicitly bound the
  remainder term. For \(y \in [0, 1 / 4]\) we have
  \begin{equation*}
    -\log(1 - y) \leq y + R_{2}y^{2}
  \end{equation*}
  with
  \begin{equation*}
    R_{2} = \sup_{y \in [0, 1/4]} \frac{1}{2}\left|\frac{d^{2}}{dy^{2}}\log(1 - y)\right|.
  \end{equation*}
  This gives us that for \(t > 2\) we have
  \begin{equation*}
    -\log(1 - 1 / t^{2}) \leq \frac{1}{t^{2}} + \frac{R_{2}}{t^{4}}.
  \end{equation*}
  We also use the inequality
  \begin{equation*}
    \sqrt{\log(1 + 1 / (xt))}
    = \sqrt{\log(1 + xt) + \log(1 / (xt))}
    \leq \sqrt{\log(1 + xt)} + \sqrt{\log(1 / (xt))}
  \end{equation*}
  together with the bound \(\sqrt{\log(1 + xt)} < \sqrt{\log 2}\) for
  \(2 \leq t \leq 1 / x\) to split \(U_{2,m,2}\) as
  \begin{multline*}
    U_{2,m,2}(x) \leq x^{2}\Bigg(
    \int_{2}^{1 / x}\frac{\sqrt{\log(1 / (xt))}}{t}\ dt
    + R_{2}\int_{2}^{1 / x}\frac{\sqrt{\log(1/(xt))}}{t^{3}}\ dt\\
    + \sqrt{\log 2}\int_{2}^{1 / x}\frac{1}{t}\ dt
    + R_{2}\sqrt{\log 2}\int_{2}^{1 / x}\frac{1}{t^{3}}\ dt
    \Bigg).
  \end{multline*}
  For the integrals we have
  \begin{align*}
    \int_{2}^{1 / x}\frac{\sqrt{\log(1 / (xt))}}{t}\ dt &= \frac{2}{3}\log(1/(2x))^{\frac{3}{2}},\\
    \int_{2}^{1 / x}\frac{\sqrt{\log(1 / (xt))}}{t^{3}}\ dt
                                                        &= \frac{\sqrt{\log(1/(2x))} - \sqrt{2\pi}x^{2}\erfi(\sqrt{\log(1/(4x^{2}))})}{8},\\
    \int_{2}^{1 / x}\frac{1}{t}\ dt &= \log(1 / x) - \log 2,\\
    \int_{2}^{1 / x}\frac{1}{t^{3}}\ dt &= \frac{1 - 4x^{2}}{8}.
  \end{align*}
  We can avoid the \(\erfi\) function by using that
  \begin{multline*}
    x^{2}\erfi(\sqrt{\log(1/(4x^{2}))})
    = \frac{2}{\sqrt{\pi}}x^{2}\int_{0}^{\sqrt{\log(1 / (4x^{2}))}} e^{t^{2}}\ dt\\
    \leq \frac{2}{\sqrt{\pi}}x^{2} \frac{e^{\left(\sqrt{\log(1 / (4x^{2}))}\right)^{2}} - 1}{\sqrt{\log(1 / (4x^{2}))}}
    = \frac{1}{2\sqrt{2\pi}} \frac{1 - 4x^{2}}{\sqrt{\log(1 / (2x))}}.
  \end{multline*}
  Where we have used \cite[Eq.~7.8.7]{NIST:DLMF}. This gives us
  \begin{multline*}
    U_{2,m,2}(x) \leq x^{2}\Bigg(
    \frac{2}{3}\log(1/(2x))^{\frac{3}{2}}
    + R_{2}\frac{\sqrt{\log(1/(2x))}}{8}
    - R_{2}\frac{1 - 4x^{2}}{16\sqrt{\log(1 / (2x))}}\\
    + \sqrt{\log 2}\log(1 / x)
    - \log(2)^{\frac{3}{2}}
    + R_{2}\sqrt{\log 2}\frac{1 - 4x^{2}}{8}
    \Bigg).
  \end{multline*}

  For \(U_{2,m,3}\) we use that
  \(\sqrt{\log(1 + 1/(xt))} < \sqrt{\log 2}\) for
  \(1 / x < t < \pi / x\), hence
  \begin{equation*}
    U_{2,m,3}(x) \leq x^{2}\sqrt{\log 2} \int_{1 / x}^{\pi / x}-\log(1 - 1 / t^{2})t\ dt.
  \end{equation*}
  The integral can be explicitly computed to be
  \begin{multline*}
    \int_{1 / x}^{\pi / x}-\log(1 - 1 / t^{2})t\ dt\\
    = \left(
      \frac{1}{2}\left(1 - \frac{\pi^{2}}{x^{2}}\right)\log\left(1 - \frac{x^{2}}{\pi^{2}}\right)
      + \log\left(\frac{\pi}{x}\right)
    \right)
    - \left(
      \frac{1}{2}\left(1 - \frac{1}{x^{2}}\right)\log\left(1 - x^{2}\right)
      + \log\left(\frac{1}{x}\right)
    \right).
  \end{multline*}
  Reordering and simplifying this gives us
  \begin{equation*}
    \int_{1 / x}^{\pi / x}-\log(1 - 1 / t^{2})t\ dt =
    \frac{1}{2}\log\left(\frac{\pi^{2} - x^{2}}{1 - x^{2}}\right)
    + \frac{\log(1 - x^{2})}{2x^{2}}
    - \frac{\log\left(1 - \frac{x^{2}}{\pi^{2}}\right)}{2x^{2}}
  \end{equation*}
  and hence
  \begin{equation*}
    U_{2,m,3}(x) \leq x^{2}\sqrt{\log 2}
    \left(
      \frac{1}{2}\log\left(\frac{\pi^{2} - x^{2}}{1 - x^{2}}\right)
      + \frac{\log(1 - x^{2})}{2x^{2}}
      - \frac{\log\left(1 - \frac{x^{2}}{\pi^{2}}\right)}{2x^{2}}
    \right).
  \end{equation*}

  Putting \(U_{2,m,1}\), \(U_{2,m,2}\) and \(U_{2,m,3}\) together we
  arrive at
  \begin{multline*}
    U_{2,m}(x) \leq x^{2}\Bigg(
    \sqrt{\log(1 + 1 / x)} \log\left(\frac{16}{3\sqrt{3}}\right)\\
    + \Bigg(
    \frac{2}{3}\log(1/(2x))^{\frac{3}{2}}
    + R_{2}\frac{\sqrt{\log(1/(2x))}}{8}
    - R_{2}\frac{\sqrt{2\pi}x^{2}\erfi(\sqrt{\log(1/(4x^{2}))})}{8}\\
    + \sqrt{\log 2}\log(1 / x)
    - \log(2)^{\frac{3}{2}}
    + R_{2}\sqrt{\log 2}\frac{1 - 4x^{2}}{8}
    \Bigg)\\
    + \sqrt{\log 2}
    \left(
      \log\left(\frac{\pi^{2} - x^{2}}{1 - x^{2}}\right)
      + \frac{\log(1 - x^{2})}{2x^{2}}
      - \frac{\pi^{2}\log\left(1 - \frac{x^{2}}{\pi^{2}}\right)}{2x^{2}}
    \right)
    \Bigg).
  \end{multline*}

  For \(U_{2,r}\) we use the fact that
  \begin{equation*}
    -\frac{
      \log\left(\sinc\left(\frac{y - x}{2}\right)\right)
      + \log\left(\sinc\left(\frac{y + x}{2}\right)\right)
      - 2\log\left(\sinc\left(\frac{y}{2}\right)\right)
    }{x^{2}}
  \end{equation*}
  is bounded for \(t \in [x, \pi]\) uniformly in \(x\). If we let
  \begin{equation}
    \label{eq:D1-x-y}
    D_{1} = \sup_{x \in [0, \epsilon]} \sup_{y \in [x, \pi]}
    -\frac{
      \log\left(\sinc\left(\frac{y - x}{2}\right)\right)
      + \log\left(\sinc\left(\frac{y + x}{2}\right)\right)
      - 2\log\left(\sinc\left(\frac{y}{2}\right)\right)
    }{x^{2}}
  \end{equation}
  we get
  \begin{equation*}
    U_{2,r}(x) \leq D_{1}x^{2}\int_{x}^{\pi}y\sqrt{\log(1 + 1/y)}\ dy.
  \end{equation*}
  If we also let
  \begin{equation*}
    c_{2} = \int_{0}^{\pi}y\sqrt{\log(1 + 1/y)}\ dy
  \end{equation*}
  we have \(U_{2,r}(x) \leq D_{1}c_{2}x^{2}\).

  To easier bound \(D_{1}\) we note that the function in
  \eqref{eq:D1-x-y} is increasing in \(y\) for \(y \in [x, \pi]\). To
  see this focus on the part
  \begin{equation*}
    \log\left(\sinc\left(\frac{y - x}{2}\right)\right) + \log\left(\sinc\left(\frac{y + x}{2}\right)\right) - 2\log\left(\sinc\left(\frac{y}{2}\right)\right).
  \end{equation*}
  which we want to show is decreasing in \(y\). If we let
  \(f(y) = \log\left(\sinc\left(\frac{y}{2}\right)\right)\) we can
  write the above as
  \begin{equation*}
    f(y - x) + f(y + x) - 2f(y).
  \end{equation*}
  Differentiating we have
  \begin{equation*}
    f'(y - x) + f'(y + x) - 2f'(y),
  \end{equation*}
  since \(f'\) is concave on \((0, 2\pi)\) this is non-positive. To
  see that \(f'\) indeed is concave on this interval it is enough to
  check that
  \begin{equation*}
    f'''(y) = \frac{\cot\left(\frac{y}{2}\right)}{4\sin^{2}\left(\frac{y}{2}\right)} - \frac{2}{y^{3}}
  \end{equation*}
  is negative on \((0, 2\pi)\), which follows from that the
  coefficients in the Taylor expansion at \(y = 0\) all are negative.
  This means that the supremum for \(D_{1}\) is attained at
  \(y = \pi\) and it reduces to
  \begin{equation*}
    D_{1} = \sup_{x \in [0, \epsilon]} -\frac{\log\left(\sinc\left(\frac{\pi - x}{2}\right)\right) + \log\left(\sinc\left(\frac{x + \pi}{2}\right)\right) - 2\log\left(\sinc\left(\frac{\pi}{2}\right)\right)}{x^{2}}.
  \end{equation*}
\end{proof}

To use these lemmas in the computer assisted proof the first step is
to compute enclosures of \(c_{1}\), \(c_{2}\), \(R_{1}\), \(R_{2}\)
and \(D_{1}\). For \(R_{2}\) this can be done directly using Taylor
arithmetic. For \(R_{1}\) and \(D_{1}\) we have removable
singularities that need to be dealt with, this can be done using the
approach described in Appendix \ref{sec:removable-singularities}. For
\(c_{1}\) and \(c_{2}\) we refer to Appendix
\ref{sec:rigorous-integration}.

For Lemma~\ref{lemma:asymptotic-U1} it is straightforward to compute
enclosures of all the terms in the upper bound by using monotonicity
properties in \(x\). For Lemma~\ref{lemma:asymptotic-U2} the
expression for the upper bound is more complicated. In particular
several of the terms contain removable singularities at \(x = 0\),
these are handled using the approach described in Appendix
\ref{sec:removable-singularities}. As an example we show how to bound
\begin{equation*}
  \frac{\frac{2}{3}\log(1/(2x))^{\frac{3}{2}}}{\log(1/x)\sqrt{\log(1 + 1/x)}},
\end{equation*}
the other terms can be done in a similar way.
\begin{lemma}
  The function
  \begin{equation*}
    f(x) = \frac{\frac{2}{3}\log(1/(2x))^{\frac{3}{2}}}{\log(1/x)\sqrt{\log(1 + 1/x)}}
  \end{equation*}
  is bounded from above by \(2 / 3\) and is decreasing in \(x\) on the
  interval \((0, 1 / 2)\).
\end{lemma}
\begin{proof}
  Taking the limit \(x \to 0^{+}\) gives the value \(\frac{2}{3}\), it
  is hence enough to prove that it is decreasing in \(x\).
  Differentiating with respect to \(x\) gives us
  \begin{equation*}
    \frac{\sqrt{\log(1 / (2x))}}{3x(x + 1)\log(1 + 1 / x)^{3/2}\log(1 / x)^{2}}\left(
      \log(1 / (2x))\log(1 / x) - (x + 1)\log(1 + 1 / x)\log(4 / x)
    \right).
  \end{equation*}
  The sign is given by that of
  \begin{equation*}
    \log(1 / (2x))\log(1 / x) - (x + 1)\log(1 + 1 / x)\log(4 / x).
  \end{equation*}
  For \(x \in (0, 1 / 2)\) we get the upper bound
  \begin{align*}
    \log(1 / (2x))\log(1 / x) - (x + 1)\log(1 + 1 / x)\log(4 / x)
    &\leq \log(1 / (2x))\log(1 / x) - \log(1 + 1 / x)\log(4 / x)\\
    &\leq \log(1 / (2x))\log(1 / x) - \log(1 / x)\log(4 / x)\\
    &= \log(1 / x)(\log(1 / (2x)) - \log(4 / x))\\
    &= \log(1 / x)\log(1 / 8),
  \end{align*}
  which is negative.
\end{proof}

\section{Bounds for \(D_{0}\), \(\delta_{0}\) and \(n_{0}\)}
\label{sec:bounds-for-values}
We are now ready to give bounds for \(n_{0}\), \(\delta_{0}\) and
\(D_{0}\). Recall that they are given by
\begin{equation*}
  n_{0} = \sup_{x \in [0, \pi]} |N(x)|,\quad
  \delta_{0} = \sup_{x \in [0, \pi]} |F(x)|,\quad
  D_{0} = \sup_{x \in [0, \pi]} |\mathcal{T}(x)|.
\end{equation*}
In each case we split the interval \([0, \pi]\) into two parts,
\([0, \epsilon]\) and \([\epsilon, \pi]\), with \(\epsilon\) varying
for the different cases, and threat them separately. For the interval
\([0, \epsilon]\) we use the asymptotic bounds for the different
functions that were introduced in the previous two sections. For the
interval \([\epsilon, \pi]\) we evaluate the functions directly using
interval arithmetic. For the direct evaluation the only complicated
part is the computation of \(U(x)\) in \(\mathcal{T}(x)\), where we
proceed as discussed in Section \ref{sec:analysis-T} and
Appendix~\ref{sec:rigorous-integration}.

Consider the problem of enclosing the maximum of \(f\) on some
interval \(I\) to some predetermined tolerance. The main idea is to
iteratively bisect the interval \(I\) into smaller and smaller
subintervals. At every iteration we compute an enclosure of \(f\) on
each subinterval. From these enclosures a lower bound of the maximum
can be computed. We then discard all subintervals for which the
enclosure is less than the lower bound of the maximum, the maximum
cannot be attained there. For the remaining subintervals we check if
their enclosure satisfies the required tolerance, in that case we
don't bisect them further. If there are any subintervals left we
bisect them and continue with the next iteration. In the end, either
when there are no subintervals left to bisect or we have reached some
maximum number of iterations (to guarantee that the procedure
terminates), we return the maximum of all subintervals that were not
discarded. This is guaranteed to give an enclosure of the maximum of
\(f\) on the interval.

If we are able to compute Taylor series of the function \(f\) we can
improve the performance of this procedure significantly (see e.g.
\cite{dahne20_comput_tight_enclos_laplac_eigen,
  dahne21_count_to_paynes_nodal_line} where a similar approach is
used). Consider a subinterval \(I_{i}\), instead of computing an
enclosure of \(f(I_{i})\) we compute a Taylor polynomial \(P\) at the
midpoint and an enclosure \(R\) of the remainder term such that
\(f(x) \in P(x) + R\) for \(x \in I_{i}\). We then have
\begin{equation}
  \label{eq:taylor-bound}
  \sup_{x \in I_{i}} f(x) \in \sup_{x \in I_{i}} P(x) + R.
\end{equation}
To compute \(\sup_{x \in I_{i}} P(x)\) we isolate the roots of \(P'\)
on \(I_{i}\) and evaluate \(P\) on the roots as well as the endpoints
of the interval. In practice the computation of \(R\) involves
computing an enclosure of the Taylor series of \(f\) on the full
interval \(I_{i}\). Since this includes the derivative we can as an
extra optimization check if the derivative is non-zero, in which case
\(f\) is monotone and it is enough to evaluate \(f\) on either the
left or the right endpoint of \(I_{i}\), depending on the sign of the
derivative.

The above procedures can easily be adapted to instead compute the
minimum of \(f\) on the interval, joining them together we can thus
compute the extrema on the interval. In some cases we don't care about
computing an enclosure of the maximum, but only to prove that it is
bounded by some value. Instead of using a tolerance we then discard
any subintervals for which the enclosure of the maximum is less than
the bound.

In most cases we bisect the subintervals at the midpoint, meaning that
the interval \([\lo{x}, \hi{x}]\) would be bisected into the two
intervals \([\lo{x}, (\lo{x} + \hi{x}) / 2]\) and
\([(\lo{x} + \hi{x}) / 2, \hi{x}]\). However, when the magnitude of
\(\lo{x}\) and \(\hi{x}\) are very different it can be beneficial to
bisect at the geometric midpoint (see e.g. \cite{GmezSerrano2014}), in
that case we split the interval into
\(\left[\lo{x}, \sqrt{\lo{x}\hi{x}}\right]\) and
\(\left[\sqrt{\lo{x}\hi{x}}, \hi{x}\right]\), where we assume that
\(\lo{x} > 0\).

We split the computation of the bounds for \(n_{0}\), \(\delta_{0}\)
and \(D_{0}\) into three lemmas. The proof of the lemmas are computer
assisted and we give some details about the process. The lemmas are
stated as upper bounds, but in the process of proving them we do
compute actual enclosures of the values.

The code \footnote{Available at
  \url{https://github.com/Joel-Dahne/BurgersHilbertWave.jl}, the
  results in the paper are from commit
  d7863475d1c4e9d8f49da716b4ef2dc936c3dab9.} for the computer assisted
part is implemented in Julia \cite{Julia-2017}. The main tool for the
rigorous numerics is Arb \cite{Johansson2013arb} which we use through
the Julia wrapper Arblib.jl
\footnote{\url{https://github.com/kalmarek/Arblib.jl}}. Many of the
basic interval arithmetic algorithms, such as isolating roots or
enclosing maximum values, are implemented in a separate package,
ArbExtras.jl
\footnote{\url{https://github.com/Joel-Dahne/ArbExtras.jl}}. For
finding the coefficients \(\{a_{j}\}\) and \(\{b_{n}\}\) of \(u_{0}\)
we make use of non-linear solvers from NLsolve.jl
\cite{patrick_kofod_mogensen_2020_4404703}.

The computations were done on an AMD Ryzen 9 5900X processor with 32
GB of RAM using 20 threads and when timings are given it refers to
this configuration. In most cases the computations are multithreaded
and make use of all available threads. The computations were done
using 100 bits of precision.

\begin{lemma}
  \label{lemma:bh-bounds-n}
  The constant \(n_{0}\) satisfies the inequality
  \(n_{0} \leq \bar{n}_{0} = 0.53682\).
\end{lemma}
\begin{proof}
  A plot of \(N(x)\) on the interval \([0, \pi]\) is given in Figure
  \ref{fig:BH-N} and hints at the maximum being attained at
  \(x = \pi\). A good guess for the maximum value is thus given by
  \(N(\pi)\).

  We take \(\epsilon = 0.5\). For the interval \([0, \epsilon]\) we
  don't compute an enclosure of the maximum but only prove that it is
  bounded by \(N(\pi)\).

  For the interval \([\epsilon, \pi]\) we compute an enclosure of the
  maximum. This gives us
  \begin{equation*}
    n_{0} \in [0.5368150155330973217537 \pm 5.95 \cdot 10^{-23}],
  \end{equation*}
  which is upper bounded by \(\bar{n}_{0}\).

  We are able to compute Taylor expansions of \(N(x)\) in both the
  asymptotic and non-asymptotic case as long as the subinterval
  doesn't contain zero. This allows us to use the better version of
  the algorithm, based on the Taylor polynomial, for enclosing the
  maximum and fall back to the naive version, where no information
  about the derivatives is used, for the subintervals containing zero.
  We use a Taylor expansion of degree \(0\), which is enough to pick
  up the monotonicity after only a few bisections in most cases. The
  runtime is about 5 seconds, most of it for handling the interval
  \([\epsilon, \pi]\).
\end{proof}

\begin{lemma}
  \label{lemma:bh-bounds-delta}
  The constant \(\delta_{0}\) satisfies the inequality
  \(\delta_{0} \leq \bar{\delta}_{0} = 8.4976 \cdot 10^{-4}\).
\end{lemma}
\begin{proof}
  A plot of \(F(x)\) on the interval \([0, \pi]\) is given in Figure
  \ref{fig:BH-F}, however this plot doesn't reveal the full picture of
  what happens close to \(x = 0\). Figures \ref{fig:BH-F-asymptotic-1}
  and \ref{fig:BH-F-asymptotic-2} show a log-plot of \(F(x)\) on the
  intervals \([10^{-100}, 10^{-1}]\) and \([10^{-20000}, 10^{-100}]\)
  respectively. In the latter of these figures we can see that \(F\)
  has an extrema between \(10^{-10000}\) and \(10^{-5000}\), this
  turns out to be where the maximum of \(|F|\) is attained. Note that
  these numbers are extremely small, too small to be represented even
  in standard quadruple precision (\emph{binary128}), though Arb has
  no problem handling them.

  We take \(\epsilon = 0.1\). Handling the interval
  \([\epsilon, \pi]\) is straightforward and we get the enclosure
  \([0.00022669 \pm 3.45 \cdot 10^{-9}]\). The interval
  \([0, \epsilon]\) is more delicate due to the extrema being attained
  for such an extremely small \(x\).

  In general evaluation with the asymptotic version of \(F\) is much
  faster than the non-asymptotic version. However, for the interval
  \([0, \epsilon]\) we have to do a huge number of subdivisions and it
  is therefore beneficial to optimize it slightly more. The terms in
  the expansions used for computing \(F\) are given in Lemmas
  \ref{lemma:asymptotic-inv-u0} and~\ref{lemma:asymptotic-delta}, they
  contain factors of the form \(c|x|^{d}\) and
  \(c\frac{|x|^{d}}{\log|x|}\) with varying coefficients \(c\) and
  exponents \(d\). When \(x\) gets smaller more and more of these
  terms become negligible. To reduce the number of terms in the
  expansions we can collapse all negligible terms into one remainder
  term. If we fix some \(d_{0} > 0\) then on the interval
  \([0, \epsilon]\) with \(\epsilon < 1\) we have for \(d \geq d_{0}\)
  and \(c > 0\)
  \begin{equation*}
    cx^{d} \in [0, c] \cdot x^{d_{0}},\quad
    c \frac{x^{d}}{\log x} \in \left[\frac{c}{\log \epsilon}, 0\right] \cdot x^{d_{0}},
  \end{equation*}
  with a similar expression for \(c < 0\). In this way we can take all
  terms with an exponent greater than \(d_{0}\) and put them into one
  single term with the exponent \(d_{0}\). If \(x\) is small enough so
  that \(x^{d_{0}}\) is negligible this will still give a good
  enclosure.

  For this we split the interval \([0, \epsilon]\) into four parts,
  \([0, \epsilon_{1}]\), \([\epsilon_{1}, \epsilon_{2}]\),
  \([\epsilon_{2}, \epsilon_{3}]\) and \([\epsilon_{4}, \epsilon]\),
  with \(\epsilon_{1} = 10^{-10000000}\),
  \(\epsilon_{2} = 10^{-100000}\) and \(\epsilon_{3} = 10^{-100}\).
  For the first two intervals we take \(d_{0} = 10^{-4}\), on the
  third we take it to be \(1 / 4\) and on the fourth we keep all
  terms. For the expansion from Lemma \ref{lemma:asymptotic-inv-u0}
  this leaves us with \(4\), \(837\) and \(1936\) terms respectively.
  For the expansion from Lemma \ref{lemma:asymptotic-delta} we get
  \(6\), \(2505\) and \(17381\) terms.

  The interval \([0, \epsilon_{1}]\) is taken such that a single
  evaluation of \(F\) gives a good enough enclosure, we have
  \(f([0, \epsilon_{1}]) \subseteq [\pm 2.09 \cdot 10^{-4}]\). The
  bulk of the work is for the second interval, it needs to be split
  into more than \(2^{24}\) subintervals and the final enclosure for
  the maximum is \([0.0003448 \pm 4.98 \cdot 10^{-8}]\). The third
  interval, \([\epsilon_{2}, \epsilon_{3}]\), is split into more than
  \(2^{17}\) subintervals, with the enclosure
  \([0.0008497 \pm 5.37 \cdot 10^{-8}]\). For
  \([\epsilon_{3}, \epsilon]\) it suffices to split it in around
  \(2^{10}\) subintervals and we get the enclosure
  \([0.00038604 \pm 6.73 \cdot 10^{-9}]\).

  We are able to compute Taylor expansions of \(F(x)\) in both the
  asymptotic and non-asymptotic case as long as the subinterval
  doesn't contain zero, this allows us to use the better version of
  the algorithm for enclosing the maximum and fall back to the naive
  version for the subintervals containing zero. We use a Taylor
  expansion of degree \(4\) in all cases. On the intervals
  \([\epsilon_{1}, \epsilon_{2}]\), \([\epsilon_{2}, \epsilon_{3}]\)
  and \([\epsilon_{3}, \epsilon]\) we bisect at the geometric midpoint
  instead of bisecting at the arithmetic midpoint. The runtime is
  about 15 minutes.
\end{proof}

\begin{lemma}
  \label{lemma:bh-bounds-D}
  The constant \(D_{0}\) satisfies the inequality
  \(D_{0} \leq \bar{D}_{0} = 0.94589\).
\end{lemma}
\begin{proof}
  A plot of \(\mathcal{T}(x)\) on the interval \([0, \pi]\) is given
  in Figure \ref{fig:BH-T}. It hints at the maximum being attained
  around \(x \approx 2.3\).

  We take \(\epsilon = 0.01\) and start by computing the maximum on
  \([\epsilon, \pi]\), this gives us the enclosure
  \([0.945126, 0.94589]\).

  For the interval \([0, \epsilon]\) we make use of the asymptotic
  expansions from Lemmas \ref{lemma:asymptotic-U1} and
  \ref{lemma:asymptotic-U2}. Since these only give upper bounds we are
  not able to compute an enclosure of \(\mathcal{T}\) on this
  interval. However, since the maximum is attained on the interval
  \([\epsilon, \pi]\) we only need prove that \(\mathcal{T}\) is
  bounded by this value on \([0, \epsilon]\). Hence the maximum for
  the full interval is given by that on \([\epsilon, \pi]\) and we get
  \begin{equation*}
    D_{0} \in [0.945126, 0.94589],
  \end{equation*}
  which is upper bounded by \(\bar{D}_{0}\).

  In this case we do not have access to Taylor series of
  \(\mathcal{T}(x)\) in either the asymptotic or non-asymptotic case.
  This means we have to rely on the naive version for bounding the
  maximum. On the interval \([\epsilon, \pi]\) there is one
  optimization that we can do. \(\mathcal{T}(x)\) involves a division
  by \(u_{0}(x)\) and for this function we have access to Taylor
  series, we therefore compute a tighter enclosure of this using a
  \(C^{1}\) bound.

  The total runtime for the computation is around 90 seconds, the
  majority for handling the interval \([\epsilon, \pi]\).
\end{proof}

\begin{figure}
  \centering
  \begin{subfigure}[t]{0.45\textwidth}
    \includegraphics[width=\textwidth]{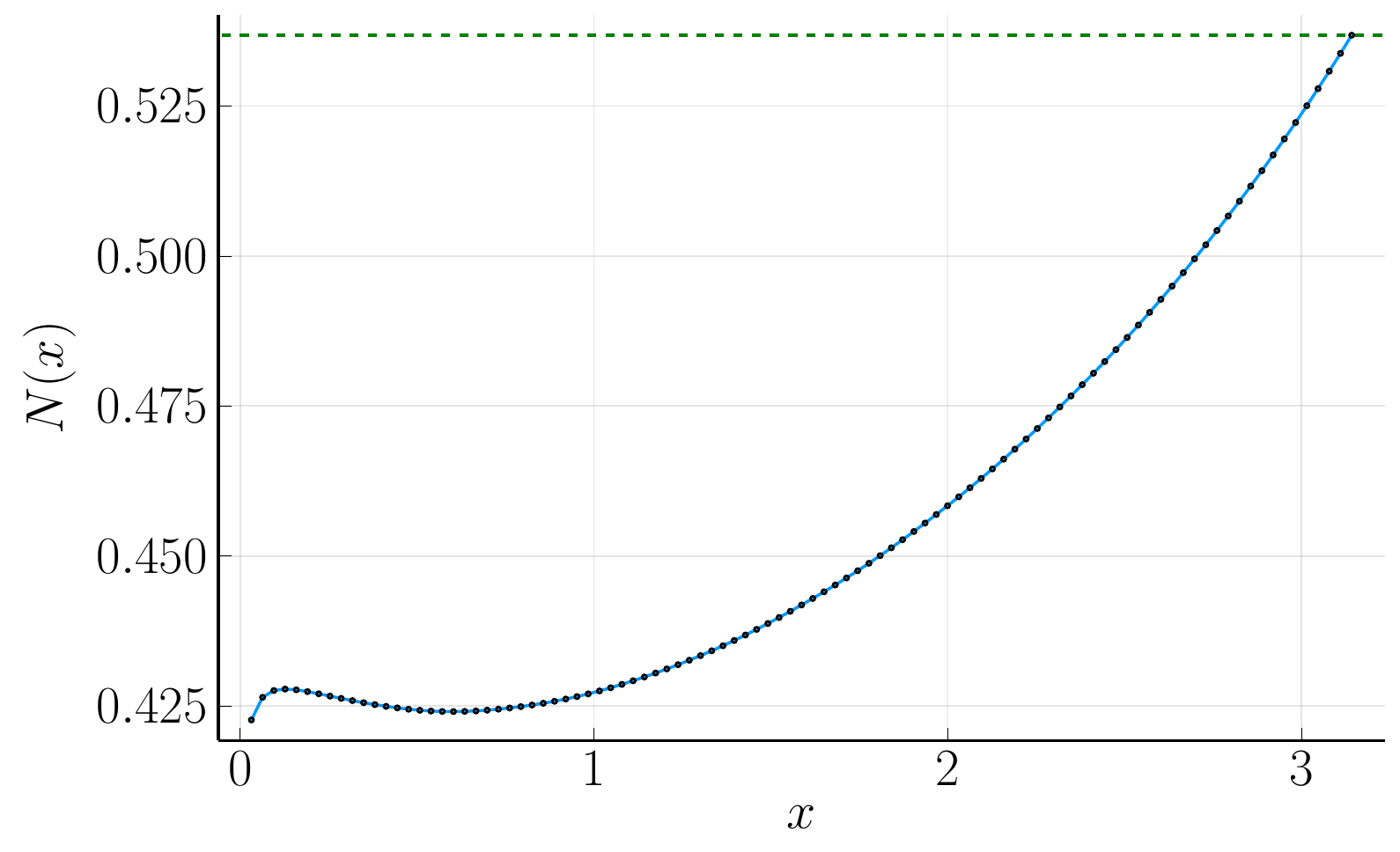}
    \caption{}
    \label{fig:BH-N}
  \end{subfigure}
  \begin{subfigure}[t]{0.45\textwidth}
    \includegraphics[width=\textwidth]{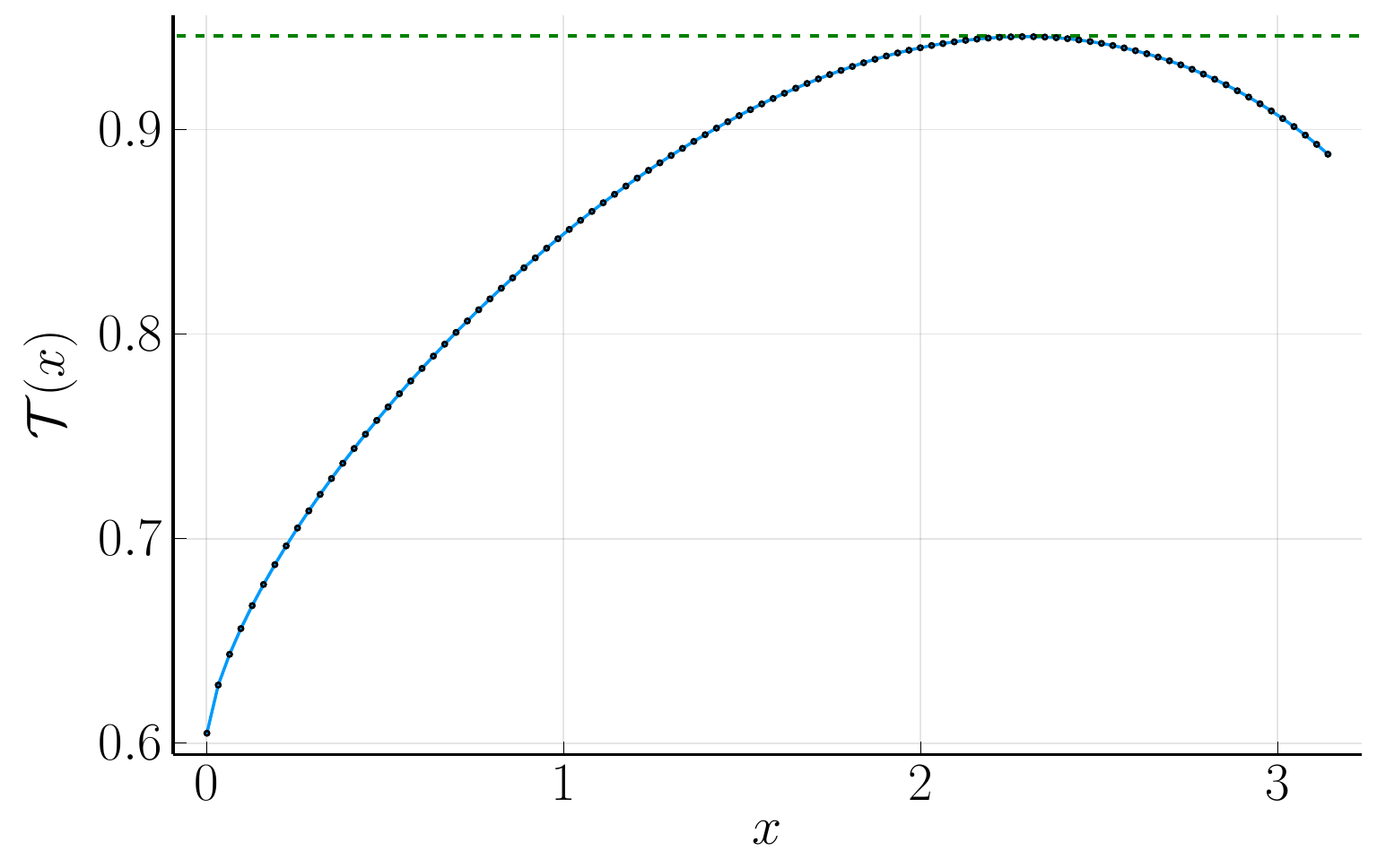}
    \caption{}
    \label{fig:BH-T}
  \end{subfigure}
  \caption{Plot of the functions \(N\) and \(\mathcal{T}\) on the
    interval \([0, \pi]\). The dashed green lines show the upper
    bounds \(\bar{n}_{0}\) and \(\bar{D}_{0}\) as given in Lemmas
    \ref{lemma:bh-bounds-n} and \ref{lemma:bh-bounds-D}}
\end{figure}

\begin{figure}
  \centering
  \begin{subfigure}[t]{0.3\textwidth}
    \includegraphics[width=\textwidth]{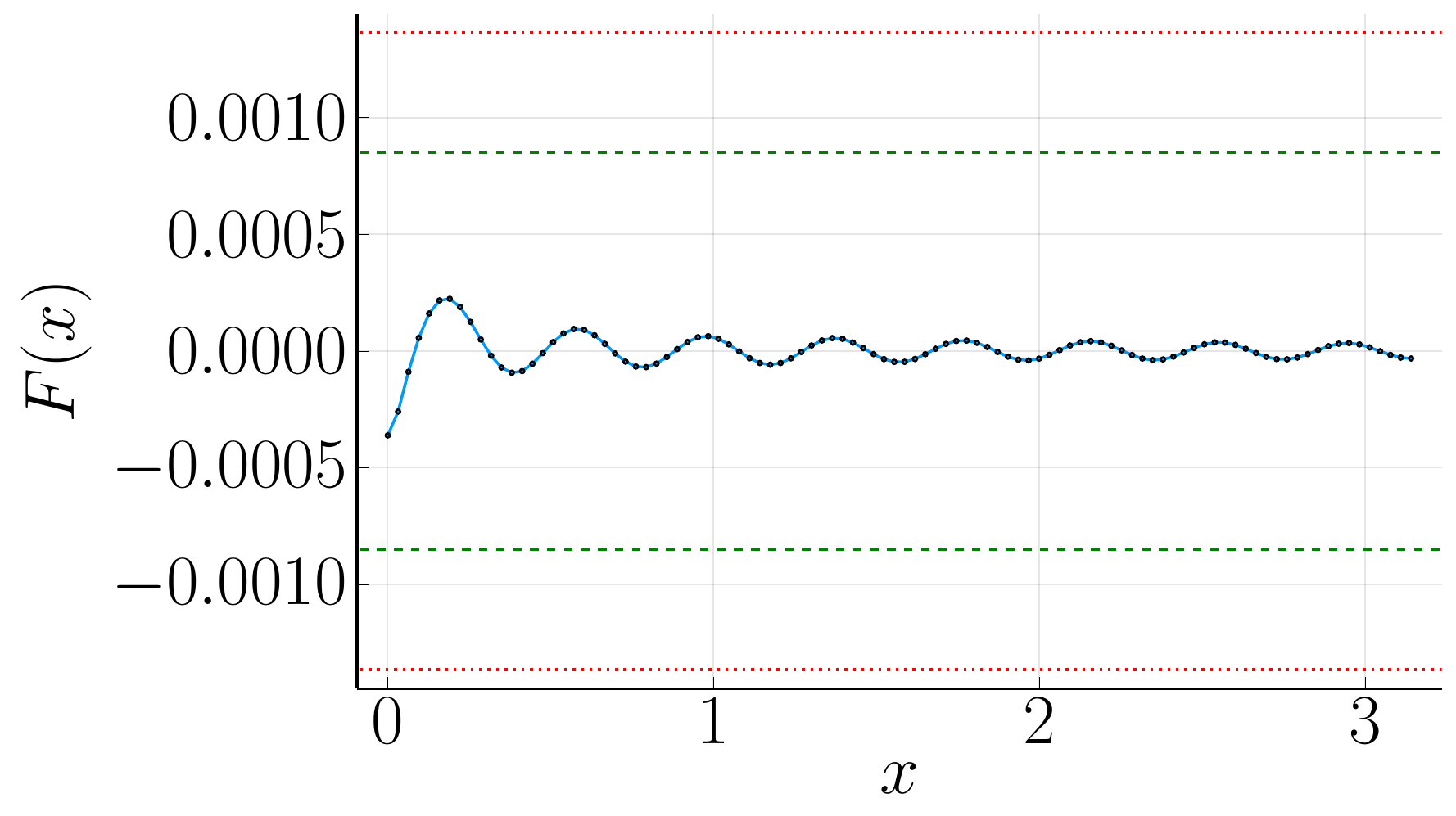}
    \caption{}
    \label{fig:BH-F}
  \end{subfigure}
  \begin{subfigure}[t]{0.3\textwidth}
    \includegraphics[width=\textwidth]{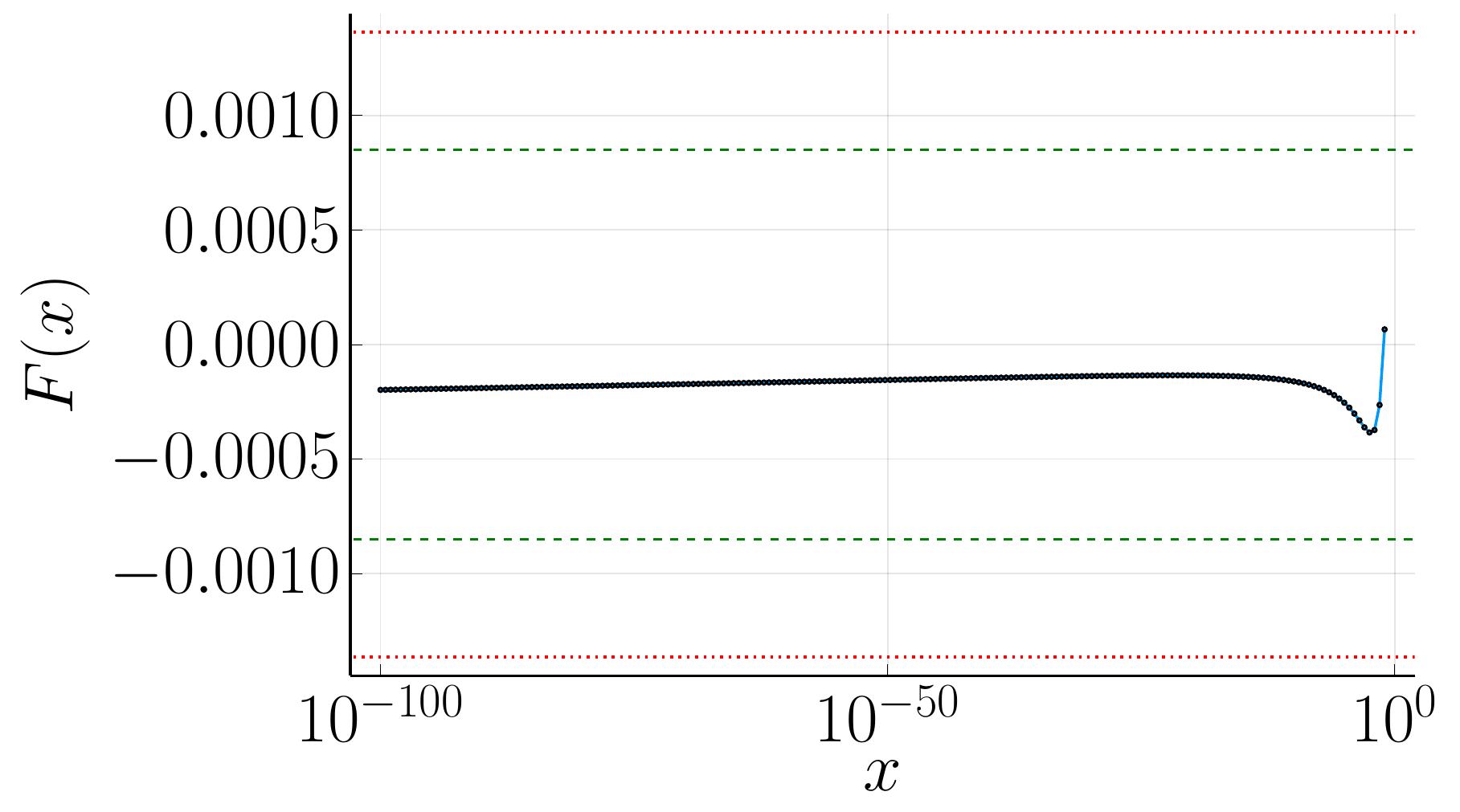}
    \caption{}
    \label{fig:BH-F-asymptotic-1}
  \end{subfigure}
  \begin{subfigure}[t]{0.3\textwidth}
    \includegraphics[width=\textwidth]{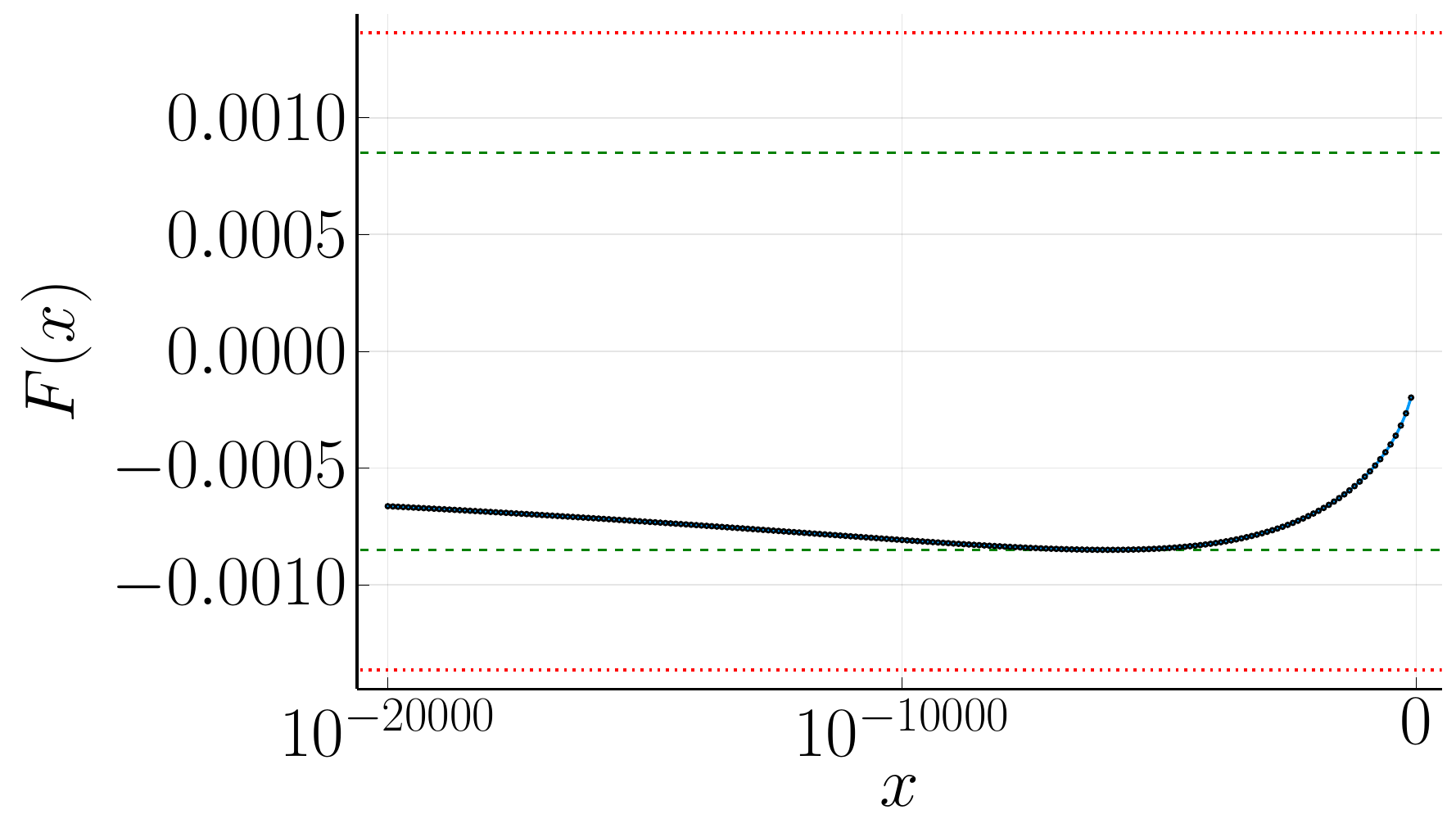}
    \caption{}
    \label{fig:BH-F-asymptotic-2}
  \end{subfigure}
  \caption{Plot of the function \(F\) on the intervals \([0, \pi]\),
    \([10^{-100}, 10^{-1}]\) and \([10^{-20000}, 10^{-100}]\). The
    dashed green line shows the upper bound \(\bar{\delta}_{0}\) as
    given in Lemma \ref{lemma:bh-bounds-delta}. The dotted red line
    shows \(\frac{(1 - \bar{D}_{0})^{2}}{4\bar{n}_{0}}\), which is the
    value we want the defect to be smaller than.}
\end{figure}

\section{Proof of Theorem \ref{thm:main}}
\label{sec:proof-main-theorem}
We are now ready to give the proof of Theorem \ref{thm:main}.

\begin{proof}[Proof of Theorem \ref{thm:main}]
  Consider the operator \(G\) from \eqref{eq:G} given by
  \begin{equation*}
    G[v] = (I - T)^{-1}(-F - Nv^{2}).
  \end{equation*}
  By Lemma \ref{lemma:bh-bounds-D} we have
  \(\|T\| \leq \bar{D}_{0} < 1\) so the inverse of the operator
  \(I - T\) is well defined. Combining Lemmas \ref{lemma:bh-bounds-n},
  \ref{lemma:bh-bounds-delta} and \ref{lemma:bh-bounds-D} gives us the
  inequality
  \begin{equation*}
    \delta_{0} \leq \bar{\delta}_{0} < \frac{(1 - \bar{D}_{0})^{2}}{4\bar{n}_{0}} \leq \frac{(1 - D_{0})^{2}}{4n_{0}}.
  \end{equation*}
  This, together with Proposition \ref{prop:contraction} and Banach
  fixed-point theorem proves that for
  \begin{equation*}
    \epsilon = \frac{1 - D_{0} - \sqrt{(1 - D_{0})^{2} - 4\delta_{0}n_{0}}}{2n_{0}}
  \end{equation*}
  the operator \(G\) has a unique fixed-point \(v_{0}\) in
  \(X_{\epsilon} \subseteq L^{\infty}(\mathbb{T})\).

  By the construction of the operator \(G\) this means that the function
  \begin{equation*}
    u(x) = u_{0}(x) + w(x)v_{0}(x)
  \end{equation*}
  solves \eqref{eq:main}, given by
  \begin{equation*}
    \frac{1}{2}u^{2} = -\HopGeneric[u].
  \end{equation*}
  For any wavespeed \(c \in \mathbb{R}\) we then have that the
  function
  \begin{equation*}
    \varphi(x) = c - u(x)
  \end{equation*}
  is a traveling wave solution to \eqref{eq:BH}. This proves the
  existence of a \(2\pi\)-periodic highest cusped traveling wave.

  To get the asymptotic behaviour we note that
  \begin{equation*}
    u_{0}(x) = -\frac{1}{\pi}|x|\log|x| + \mathcal{O}(|x|)
    \quad\text{and}\quad
    w(x)v_{0}(x) = \mathcal{O}(|x|\sqrt{\log(|x|)}).
  \end{equation*}
  Hence
  \begin{equation*}
    u(x) = -\frac{1}{\pi}|x|\log|x| + \mathcal{O}(|x|\sqrt{\log(|x|)})
  \end{equation*}
  and
  \begin{equation*}
    \varphi(x) = c + \frac{1}{\pi}|x|\log|x| + \mathcal{O}(|x|\sqrt{\log(|x|)}),
  \end{equation*}
  as we wanted to show.
\end{proof}

\appendix

\section{Removable singularities}
\label{sec:removable-singularities}

In several cases we have to compute enclosures of functions with
removable singularities. For example the function
\begin{equation*}
  \Gamma(1 - s)\cos(\pi(1 - s) / 2)
\end{equation*}
comes up when computing \(C_{s}\) through equation
\eqref{eq:clausenc-periodic-zeta} and has a removable singularity
whenever \(s\) is a positive even integer. In this appendix we explain
how to compute rigorous enclosures at and around these points. For
this we need a way to handle the removable singularity. Let
\begin{equation*}
  f_{n}(x) = \frac{f^{(n)}(x)}{n!}.
\end{equation*}
We have the following lemma for handling functions with removable
singularities:
\begin{lemma}
  Let \(m \in \mathbb{Z}_{\geq 0}\) and let \(I\) be an interval
  containing zero. Consider a function \(f(x)\) with a zero of order
  \(n\) at \(x = 0\) and such that \(f^{(m + n)}(x)\) is absolutely
  continuous on \(I\). Then for all \(x \in I\) we have
  \begin{equation*}
    \frac{f(x)}{x^{n}} = \sum_{k = 0}^{m}f_{k + n}(0)x^{k} + f_{m + n + 1}(\xi)x^{m + 1}
  \end{equation*}
  for some \(\xi\) between \(0\) and \(x\). Furthermore, if
  \(f^{m + n + p}(x)\) is absolutely continuous for
  \(p \in \mathbb{Z}_{\geq 0}\) we have
  \begin{equation*}
    \frac{d^{p}}{dx^{p}}\frac{f(x)}{x^{n}}
    = \sum_{k = 0}^{m}\frac{(k + p)!}{k!}f_{k + n + p}(0)x^{k}
    + \frac{(m + p + 1)!}{(m + 1)!}f_{m + n + p + 1}(\xi)x^{m + 1}
  \end{equation*}
  for some \(\xi\) between \(0\) and \(x\).
\end{lemma}
\begin{proof}
  The first statement follows directly from expanding \(f\) in a
  Taylor series with a remainder term on Lagrange form and dividing by
  \(x^{n}\), using that \(f_{k}(0) = 0\) for \(k < n\).

  For the second statement we start by noting that
  \begin{equation}
    \label{eq:derivate-sum}
    \frac{d^{p}}{dx^{p}} \frac{f(x)}{x^{n}}
    = \sum_{l = 0}^{p}\binom{p}{l}f^{(p - l)}(x)(-1)^{l}\frac{(n + l - 1)!}{(n - 1)!}\frac{1}{x^{n + l}}
    = \frac{1}{x^{n + p}}\sum_{l = 0}^{p}(-1)^{l}\binom{p}{l}\frac{(n + l - 1)!}{(n - 1)!}f^{(p - l)}(x)x^{p - l}.
  \end{equation}
  The Taylor expansion of \(f^{(p - l)}(x)\) with the remainder term
  in integral form is
  \begin{equation}
    \label{eq:f-taylor-integral}
    f^{(p - l)}(x) = \sum_{k = 0}^{m + n + l}(f^{(p - l)})_{k}(0)x^{k} + x^{m + n + l + 1}(m + n + l + 1)\int_{0}^{1} (f^{(p - l)})_{m + n + l + 1}(tx)(1 - t)^{m + n + l}\ dt.
  \end{equation}
  Using that
  \begin{equation*}
    (f^{(p - l)})_{k}(x) = \frac{f^{(p - l + k)}(x)}{k!} = \frac{(p - l + k)!}{k!}\frac{f^{(p - l + k)}(x)}{(p - l + k)!} = \frac{(p - l + k)!}{k!}f_{p - l + k}(x)
  \end{equation*}
  we can write \eqref{eq:f-taylor-integral} as
  \begin{align*}
    f^{(p - l)}(x) =& \sum_{k = 0}^{m + n + l}\frac{(p - l + k)!}{k!}f_{p - l + k}(0)x^{k}\\
                    &+ x^{m + n + l + 1}(m + l + n + 1)\int_{0}^{1} \frac{(p - l + (m + n + l + 1))!}{(m + n + l + 1)!}f_{p - l + (m + n + l + 1)}(tx)(1 - t)^{m + n + l}\ dt\\
    =& \sum_{k = 0}^{m + n + l}\frac{(p - l + k)!}{k!}f_{p - l + k}(0)x^{k} + x^{m + n + l + 1}\frac{(p + m + n + 1)!}{(m + n + l)!}\int_{0}^{1} f_{p + m + n + 1}(tx)(1 - t)^{m + n + l}\ dt.
  \end{align*}
  Multiplying by \(x^{p - l}\) we get
  \begin{equation}
    \label{eq:derivative-expansion}
    f^{(p - l)}(x)x^{p - l} = \sum_{k = p - l}^{m + n + p}\frac{k!}{(k - p + l)!}f_{k}(0)x^{k} + x^{m + n + p + 1}\frac{(p + m + n + 1)!}{(m + n + l)!}\int_{0}^{1} f_{p + m + n + 1}(tx)(1 - t)^{m + n + l}\ dt.
  \end{equation}
  Inserting the main term of \eqref{eq:derivative-expansion} into
  \eqref{eq:derivate-sum} gives us
  \begin{equation*}
    \frac{1}{x^{n + p}}\sum_{l = 0}^{p}(-1)^{l}\binom{p}{l}\frac{(n + l - 1)!}{(n - 1)!}\sum_{k = p - l}^{m + n + p}\frac{k!}{(k - p + l)!}f_{k}(0)x^{k}
  \end{equation*}
  Splitting into \(k < p\) and \(k \geq p\) we have
  \begin{multline}
    \label{eq:main-term-split}
    \sum_{k = 0}^{p - 1}\sum_{l = p-k}^{p}(-1)^{l}\binom{p}{l}\frac{(n + l - 1)!}{(n - 1)!}\frac{k!}{(k - p + l)!}f_{k}(0)x^{k - n - p}\\
    + \sum_{k = p}^{m + n + p}\sum_{l = 0}^{p}(-1)^{l}\binom{p}{l}\frac{(n + l - 1)!}{(n - 1)!}\frac{k!}{(k - p + l)!}f_{k}(0)x^{k - n - p}
  \end{multline}
  Since \(f_{k}(0) = 0\) for \(k < n\) the first sum reduces to
  \begin{equation*}
    \sum_{k = n}^{p - 1}\sum_{l = p-k}^{p}(-1)^{l}\binom{p}{l}\frac{(n + l - 1)!}{(n - 1)!}\frac{k!}{(k - p + l)!}f_{k}(0)x^{k - n - p}
  \end{equation*}
  Using \cite[Sec.~15.2.4]{NIST:DLMF} we have
  \begin{equation*}
    \sum_{l = p-k}^{p}(-1)^{l}\binom{p}{l}\frac{(n + l - 1)!}{(n - 1)!}\frac{k!}{(k - p + l)!}
    = (-1)^{p - k}\binom{p}{p - k}\frac{k!(n + p - k - 1)!}{(n - 1)!}{}_{2}F_{1}(-k, n + p - k, p - k + 1, 1).
  \end{equation*}
  From \cite[Sec.~15.4.24]{NIST:DLMF} we get that this is zero for
  \(k \geq n\). The second sum in \eqref{eq:main-term-split} we can
  rewrite as
  \begin{equation*}
    \sum_{k = -n}^{m}\sum_{l = 0}^{p}(-1)^{l}\binom{p}{l}\frac{(n + l - 1)!}{(n - 1)!}\frac{(k + n + p)!}{(k + n + l)!}f_{k + n + p}(0)x^{k}
  \end{equation*}
  For \(k < -p\) the terms are zero since \(f_{k + n + p}(0) = 0\).
  For \(k \geq -p\) we use \cite[Sec.~15.2.4]{NIST:DLMF} to write the
  inner sum as
  \begin{equation*}
    \sum_{l = 0}^{p}(-1)^{l}\binom{p}{l}\frac{(n + l - 1)!}{(n - 1)!}\frac{(k + n + p)!}{(k + n + l)!}
    = \frac{(k + n + p)!}{(k + n)!}{}_{2}F_{1}(-p, n, k + n + 1, 1).
  \end{equation*}
  From \cite[Sec.~15.4.24]{NIST:DLMF} we get that this is zero for
  \(-p \leq k < 0\) and for \(0 \leq k \leq m\) it is given by
  \(\frac{(k + p)!}{k!}\). This gives us that
  \eqref{eq:main-term-split} can be written as
  \begin{equation*}
    \sum_{k = 0}^{m}\frac{(k + p)!}{k!}f_{k + n + p}(0)x^{k},
  \end{equation*}
  which is what we wanted to show.

  We are now interested in the remainder term of
  \eqref{eq:derivative-expansion} when inserted into
  \eqref{eq:derivate-sum}, we get
  \begin{multline*}
    \frac{1}{x^{n + p}}\sum_{l = 0}^{p}(-1)^{l}\binom{p}{l}\frac{(n + l - 1)!}{(n - 1)!}x^{m + n + p + 1}\frac{(p + m + n + 1)!}{(m + n + l)!}\int_{0}^{1} f_{p + m + n + 1}(tx)(1 - t)^{m + n + l}\ dt\\
    = x^{m + 1}\int_{0}^{1} f_{p + m + n + 1}(tx)\sum_{l = 0}^{p}(-1)^{l}\binom{p}{l}\frac{(n + l - 1)!}{(n - 1)!}\frac{(p + m + n + 1)!}{(m + n + l)!}(1 - t)^{m + n + l}\ dt.
  \end{multline*}
  We first show that the sum is non-negative. Using
  \cite[Sec.~15.2.4]{NIST:DLMF} we have
  \begin{multline*}
    \sum_{l = 0}^{p}(-1)^{l}\binom{p}{l}\frac{(n + l - 1)!}{(n - 1)!}\frac{(p + m + n + 1)!}{(m + n + l)!}(1 - t)^{m + n + l}\\
    = (1 - t)^{m + n}\frac{(p + m + n + 1)!}{(m + n)!}\sum_{l = 0}^{p}(-1)^{l}\binom{p}{l}\frac{(n + l - 1)!}{(n - 1)!}\frac{(m + n)!}{(m + n + l)!}(1 - t)^{l}\\
    = (1 - t)^{m + n}\frac{(p + m + n + 1)!}{(m + n)!}{}_{2}F_{1}(-p, n, m + n + 1, 1 - t).
  \end{multline*}
  By \cite[Theorem 3.2]{hud2510} (see also \cite[Theorem
  2]{Dominici2013}), \({}_{2}F_{1}(-p, n, m + n + 1, 1 - t)\) have no
  roots on the interval \([0, 1]\) and for \(t = 1\) it is positive,
  it is hence positive on \([0, 1]\). Since the sum is positive the
  integral can be written as
  \begin{equation*}
    f_{p + m + n + 1}(\xi)x^{m + 1}\int_{0}^{1} \sum_{l = 0}^{p}(-1)^{l}\binom{p}{l}\frac{(n + l - 1)!}{(n - 1)!}\frac{(p + m + n + 1)!}{(m + n + l)!}(1 - t)^{m + n + l}\ dt.
  \end{equation*}
  for some \(\xi\) between \(0\) and \(x\). Simplifying further we get
  \begin{multline*}
    f_{m + n + p + 1}(\xi)x^{m + 1}\sum_{l = 0}^{p}(-1)^{l}\binom{p}{l}\frac{(n + l - 1)!}{(n - 1)!}\frac{(p + m + n + 1)!}{(m + n + l)!}\int_{0}^{1} (1 - t)^{m + n + l}\ dt\\
    = f_{m + n + p + 1}(\xi)x^{m + 1}\sum_{l = 0}^{p}(-1)^{l}\binom{p}{l}\frac{(n + l - 1)!}{(n - 1)!}\frac{(p + m + n + 1)!}{(m + n + l)!}\frac{1}{m + n + l + 1}\\
    = f_{m + n + p + 1}(\xi)x^{m + 1}\sum_{l = 0}^{p}(-1)^{l}\binom{p}{l}\frac{(n + l - 1)!}{(n - 1)!}\frac{(p + m + n + 1)!}{(m + n + l + 1)!}\\
    = \frac{(m + p + 1)!}{(m + 1)!}f_{m + n + p + 1}(\xi)x^{m + 1},
  \end{multline*}
  which is exactly the expression given for the remainder term.
\end{proof}

Using this lemma we can compute enclosures at and around removable
singularities as long as we have good control over \(f\) and its
derivatives.
\begin{example}
  \label{ex:removables-singularity}
  Consider the function \(\Gamma(1 - s)\cos(\pi(1 - s) / 2)\) with a
  removable singularity at \(s = 2\). If we let \(t = s - 2\) we can
  write the function as
  \begin{equation*}
    \frac{\cos(\pi(-1 - t) / 2)}{t} \cdot (t\Gamma(-1 - t))
  \end{equation*}
  If we let \(f(t) = \cos(\pi(-1 - t) / 2)\) and take \(m \geq 0\),
  then for the second factor the lemma then tells us that
  \begin{equation*}
    \frac{\cos(\pi(-1 - t) / 2)}{t} = \sum_{k = 0}^{m}(k + 1)f_{k + 1}(0)t^{k} + (m + 2)f_{m + 2}(\xi)t^{m + 1}
  \end{equation*}
  for some \(\xi\) between \(0\) and \(x\). Using interval arithmetic
  we can easily compute enclosure of the coefficients for the
  polynomial as well as \(f_{m + 2}(\xi)\). For the second factor we
  can't directly apply the lemma, instead we rewrite it in terms of
  the reciprocal gamma function, \(1 / \Gamma(s)\). If we let
  \(g(t) = 1 / \Gamma(-1 - t)\) we have
  \begin{equation*}
    t \Gamma(-1 - t) = \left(\frac{g(t)}{t}\right)^{-1}.
  \end{equation*}
  The function \(g(t)\) has a root at \(t = 0\) and we can thus apply
  the lemma on \(\frac{g(t)}{t}\). An enclosure of the coefficients
  for the polynomial as well as the remainder term can be computed
  using the implementation of the reciprocal gamma function in Arb.
\end{example}

\section{Computing enclosures of Clausen functions}
\label{sec:comp-encl-claus}
To be able to compute bounds of \(D_{\alpha}\), \(\delta_{\alpha}\)
and \(n_{\alpha}\) it is critical that we can compute accurate
enclosures of \(C_{s}(x)\) and \(S_{s}(x)\), including expansions in
the argument and derivatives in the parameter. We here go through how
these enclosures are computed. We make use of several different
special functions, most of them with implementations in Arb (see e.g.
\cite{Johansson2014hurwitz, Johansson2014thesis}). In several cases we
encounter removable singularities, they are all dealt with as
explained in Appendix \ref{sec:removable-singularities}, see in
particular example \ref{ex:removables-singularity}.

We start by going through how to compute \(C_{s}(x)\) and \(S_{s}(x)\)
for \(s, x \in \mathbb{R}\). Since both \(C_{s}(x)\) and \(S_{s}(x)\)
are \(2\pi\)-periodic we can reduce it to \(x = 0\) or
\(0 < x < 2\pi\).

For \(x = 0\) and \(s > 1\) we get directly from the defining sum that
\(C_{s}(0) = \zeta(s)\) and \(S_{s}(0) = 0\). For \(s \leq 1\) both
functions typically diverge at \(x = 0\).

For \(0 < x < 2\pi\) we can compute the Clausen functions by going
through the polylog function,
\begin{equation*}
  C_{s}(x) = \real\left(\polylog_{s}(e^{ix})\right),\quad S_{s}(x) = \imag\left(\polylog_{s}(e^{ix})\right).
\end{equation*}
However it is computationally beneficial (about 40\% faster in
general) to instead go through the periodic zeta function
\cite[Sec.~25.13]{NIST:DLMF},
\begin{equation*}
  F(x, s) := \polylog_{s}(e^{2\pi i x}) = \sum_{n = 1}^{\infty}\frac{e^{2\pi i nx}}{n^{s}},
\end{equation*}
for which we have
\begin{equation*}
  C_{s}(x) = \real F\left(\frac{x}{2\pi}, s\right),\quad S_{s}(x) = \imag F\left(\frac{x}{2\pi}, s\right).
\end{equation*}
For \(0 < x < 1\) the periodic zeta function can be written as
\cite[Eq.~25.13.2]{NIST:DLMF}
\begin{equation*}
  F(x, s) = \frac{\Gamma(1 - s)}{(2\pi)^{1 - s}}\left(
    e^{\pi i (1 - s) / 2}\zeta(1 - s, x) + e^{-\pi i (1 - s) / 2}\zeta(1 - s, 1 - x)
  \right).
\end{equation*}
Taking the real and imaginary part we get, for \(0 < x < 2\pi\),
\begin{align}
  \label{eq:clausenc-periodic-zeta}
  C_{s}(x) &= \frac{\Gamma(1 - s)}{(2\pi)^{1 - s}}\cos(\pi(1 - s) / 2)
             \left(\zeta\left(1 - s, \frac{x}{2\pi}\right) + \zeta\left(1 - s, 1 - \frac{x}{2\pi}\right)\right),\\
  \label{eq:clausens-periodic-zeta}
  S_{s}(x) &= \frac{\Gamma(1 - s)}{(2\pi)^{1 - s}}\sin(\pi(1 - s) / 2)
             \left(\zeta\left(1 - s, \frac{x}{2\pi}\right) - \zeta\left(1 - s, 1 - \frac{x}{2\pi}\right)\right).
\end{align}
This formulation works well as long as \(s\) is not a non-negative
integer. For non-negative integers we have to handle some removable
singularities.

For \(s = 0\) the functions
\(\zeta\left(1 - s, \frac{x}{2\pi}\right)\) and
\(\zeta\left(1 - s, 1 - \frac{x}{2\pi}\right)\) diverge and needs to
be handled differently. The Laurent series of the zeta function gives
us
\begin{equation}
  \label{eq:deflated-zeta}
  \zeta(s, x) = \frac{1}{s - 1} + \sum_{n = 0}^{\infty} \frac{(-1)^{n}}{n!}\gamma_{n}(x)(s - 1)^{n},
\end{equation}
where \(\gamma_{n}\) is the Stieltjes constant. The sum is referred to
as the deflated zeta function, it has an implementation in Arb and we
denote it by \(\underline{\zeta}(s, x)\). Writing the Clausen
functions in terms of the deflated zeta function gives us
\begin{align*}
  C_{s}(x) &= \frac{\Gamma(1 - s)}{(2\pi)^{1 - s}}\cos(\pi(1 - s) / 2)
             \left(\underline{\zeta}\left(1 - s, \frac{x}{2\pi}\right) + \underline{\zeta}\left(1 - s, 1 - \frac{x}{2\pi}\right) - \frac{2}{s}\right),\\
  S_{s}(x) &= \frac{\Gamma(1 - s)}{(2\pi)^{1 - s}}\sin(\pi(1 - s) / 2)
             \left(\underline{\zeta}\left(1 - s, \frac{x}{2\pi}\right) - \underline{\zeta}\left(1 - s, 1 - \frac{x}{2\pi}\right)\right).
\end{align*}
The expression for \(S_{s}(x)\) is now well defined, all terms are
finite, for \(s\) around zero. For \(C_{s}(x)\) we also have to handle
the removable singularity of \(\frac{2\cos(\pi(1 - s) / 2)}{s}\).

For \(s\) equal to a positive integer, \(\Gamma(1 - s)\) has a pole.
For \(C_{s}(x)\) with even \(s\) we get a removable singularity in
\(\Gamma(1 - s)\cos(\pi(1 - s) / 2)\) and for \(S_{s}(x)\) with odd
\(s\) we get that \(\Gamma(1 - s)\sin(\pi(1 - s) / 2)\) has a
removable singularity. For the other parity of \(s\) we instead have
to look at the zeta function. For integer \(s \geq 1\) we have
\begin{equation*}
  \zeta(1 - s, x) = -\frac{B_{s}(x)}{s}
\end{equation*}
where \(B_{s}(n)\) is the Bernoulli polynomial
\cite[Eq.~25.6.3]{NIST:DLMF}. Since \(B_{s}(1 - x) = (-1)^{s}B_{s}\)
\cite[Eq.~24.4.3]{NIST:DLMF} we have
\begin{equation*}
  \zeta\left(1 - s, \frac{x}{2\pi}\right) + \zeta\left(1 - s, 1 - \frac{x}{2\pi}\right) = 0
\end{equation*}
for odd \(s\) and
\begin{equation*}
  \zeta\left(1 - s, \frac{x}{2\pi}\right) - \zeta\left(1 - s, 1 - \frac{x}{2\pi}\right) = 0
\end{equation*}
for even \(s\). This means that
\begin{equation*}
  \Gamma(1 - s)\left(\zeta\left(1 - s, \frac{x}{2\pi}\right) + \zeta\left(1 - s, 1 - \frac{x}{2\pi}\right)\right)
\end{equation*}
has a removable singularity for odd \(s\) and
\begin{equation*}
  \Gamma(1 - s)\left(\zeta\left(1 - s, \frac{x}{2\pi}\right) - \zeta\left(1 - s, 1 - \frac{x}{2\pi}\right)\right)
\end{equation*}
has a removable singularity for even \(s\). This allows us to handle
all removable singularities in \eqref{eq:clausenc-periodic-zeta} and
\eqref{eq:clausens-periodic-zeta} when \(s\) is a positive integer.

\subsection{Interval arguments}
\label{sec:interval-arguments}
We are now ready to compute enclosures for interval arguments. Let
\(\inter{x} = [\lo{x}, \hi{x}]\) and \(\inter{s} = [\lo{s}, \hi{s}]\)
be two finite intervals, we are interested in computing an enclosure
of \(C_{\inter{s}}(\inter{x})\) and \(S_{\inter{s}}(\inter{x})\). Due
to the periodicity we can reduce it to three different cases for
\(\inter{x}\)
\begin{enumerate}
\item \(\inter{x}\) doesn't contain a multiple of \(2\pi\), by adding
  or subtracting a suitable multiple of \(2\pi\) we can assume that
  \(0 < \inter{x} < 2\pi\);
\item \(\inter{x}\) has a diameter of at least \(2\pi\), it then
  covers a full period and can without loss of generality be taken as
  \(\inter{x} = [0, 2\pi]\);
\item \(\inter{x}\) contains a multiple of \(2\pi\) but has a diameter
  less than \(2\pi\), by adding or subtracting a suitable multiple of
  \(2\pi\) we can take \(\inter{x}\) such that
  \(-2\pi < \lo{x} \leq 0 \leq \hi{x} < 2\pi\).
\end{enumerate}

We begin by considering the case when \(0 < \inter{x} < 2\pi\).
It is possible to evaluate \eqref{eq:clausenc-periodic-zeta} and
\eqref{eq:clausens-periodic-zeta} directly, treating \(x\) as an
interval, however this gives huge overestimations as soon as
\(\inter{x}\) is not a very tight interval. For \(C_{s}\) we make use
of the following lemma, see also \cite[Lemma B.1]{enciso2018convexity}.
\begin{lemma}
  For all \(s \in \mathbb{R}\) the Clausen function \(C_{s}(x)\) is
  monotone in \(x\) on the interval \((0, \pi)\). For \(s > 0\) it is
  non-increasing.
\end{lemma}
\begin{proof}
  We have \(C'_{s}(x) = -S_{s - 1}(x)\). For \(s > 1\) we have
  \cite[Eq.~25.12.11]{NIST:DLMF}
  \begin{align*}
    S_{s - 1}(x) &= \imag L_{s - 1}(e^{ix})\\
                 &= \imag \frac{e^{ix}}{\Gamma(s - 1)}\int_{0}^{\infty}t^{s - 2}\frac{1}{e^{t} - e^{ix}}\ dt\\
                 &= \imag \frac{e^{ix}}{\Gamma(s - 1)}\int_{0}^{\infty}t^{s - 2}\frac{e^{t} - e^{-ix}}{(e^{t} - \cos(x))^{2} + \sin^{2}(x)}\ dt\\
                 &= \frac{\sin(x)}{\Gamma(s - 1)}\int_{0}^{\infty}t^{s - 2}\frac{e^{t}}{(e^{t} - \cos(x))^{2} + \sin^{2}(x)}\ dt.
  \end{align*}
  Which for \(s - 1 > 0\) and \(x \in (0, \pi)\) is positive,
  \(C_{s}(x)\) is hence decreasing.

  For \(s < 1\) we use equation \eqref{eq:clausens-periodic-zeta}
  together with \cite[Eq.~25.11.25]{NIST:DLMF} to get
  \begin{equation*}
    S_{s - 1}(x) = \frac{\sin\left(\frac{\pi}{2}(2 - s)\right)}{(2\pi)^{2 - s}}\int_{0}^{\infty}t^{1 - s}e^{-\frac{xt}{2\pi}}\frac{1 - e^{(x / \pi - 1)t}}{1 - e^{-t}}\ dt.
  \end{equation*}
  which sign depends only on the value \(s\). In particular, for
  \(0 < s \leq 1\) we have
  \(\sin\left(\frac{\pi}{2}(2 - s)\right) > 0\) so \(C_{s}(x)\) is
  decreasing for \(s > 0\).
\end{proof}
Since \(C_{s}(x)\) is even around \(x = \pi\) it follows that if
\(0 < \inter{x} < 2\pi\) then the extrema of
\(C_{\inter{s}}(\inter{x})\) are attained at either \(x = \lo{x}\),
\(x = \hi{x}\) or, if \(\pi \in \inter{x}\), \(x = \pi\). This allows
us to compute very tight enclosures of \(C_{\inter{s}}(\inter{x})\) in
the argument \(\inter{x}\). The function \(S_{s}(x)\) is in general
not monotone. We handle it by computing and enclosure of the
derivative
\(S_{\inter{s}}'(\inter{x}) = C_{\inter{s} - 1}(\inter{x})\), if the
enclosure of the derivative doesn't contain zero then the function is
monotone and we evaluate it at the endpoints, if the derivative
contains zero we instead use the midpoint approximation
\(S_{\inter{s}}(\inter{x}) = S_{\inter{s}}(x_{0}) + (\inter{x} -
x_{0})C_{\inter{s} - 1}(\inter{x})\) where \(x_{0}\) is the midpoint
of \(\inter{x}\).

For \(\inter{x} = [0, 2\pi]\) we assume that \(\inter{s} > 1\) since
the value is typically unbounded otherwise (in which case we return an
indeterminate result). We get an enclosure of
\(C_{\inter{s}}(\inter{x})\) by evaluating at the critical points
\(x = 0\) and \(x = \pi\). For \(S_{\inter{s}}(\inter{x})\) we use the
trivial bound
\begin{equation*}
  |S_{s}(x)| = \left|\sum_{n = 1}^{\infty}\frac{\sin(nx)}{n^{s}}\right|
  \leq \sum_{n = 1}^{\infty}\frac{|\sin(nx)|}{n^{s}}
  \leq \sum_{n = 1}^{\infty}\frac{1}{n^{s}} = \zeta(s)
\end{equation*}
to get the enclosure
\(S_{\inter{s}}(\inter{x}) \subseteq [-\zeta(\inter{s}),
\zeta(\inter{s})]\).

For the final case, when \(-2\pi < \lo{x} \leq 0 \leq \hi{x} < 2\pi\),
we also assume that \(\inter{s} > 1\). We handle
\(C_{\inter{s}}(\inter{x})\) by using the eveness and monotonicity in
\(x\). If \(\max(-\lo{x}, \hi{x}) < \pi\) then the extrema are
attained at \(x = 0\) and \(x = \max(-\lo{x}, \hi{x}) < \pi\),
otherwise the extrema are attained at \(x = 0\) and \(x = \pi\). For
\(S_{\inter{s}}(\inter{x})\) we check if it is increasing on
\(\inter{x}\) by checking if \(\max(-\lo{x}, \hi{x}) < \pi\) and
\(S_{\inter{s}}'(\max(-\lo{x}, \hi{x})) = C_{\inter{s} -
  1}(\max(-\lo{x}, \hi{x})) \geq 0\), where we have used that
\(C_{s}(x)\) is decreasing for \(s > 0\). If it is increasing we
evaluate at the endpoints, otherwise we have the trivial bound
\([-\zeta(\inter{s}), \zeta(\inter{s})]\)

In all the above cases we could reduce the problem of enclosing the
Clausen functions on \(\inter{x}\) to evaluating them on, a subset of,
\(x = 0\), \(x = \lo{x}\) (or \(-\lo{x}\) if \(\lo{x} < 0\)),
\(x = \hi{x}\), \(x = x_{0}\) (the midpoint) and \(x = \pi\). In all
cases except \(x = 0\) we have \(0 < x < 2\pi\) so we can use
\eqref{eq:clausenc-periodic-zeta} and
\eqref{eq:clausens-periodic-zeta}.

For \(\inter{s}\) we get two different cases depending on if
\(\inter{s}\) contains a non-negative integer or not.

When \(\inter{s}\) doesn't contain a non-negative integer the
formulations \eqref{eq:clausenc-periodic-zeta} and
\eqref{eq:clausens-periodic-zeta} are well defined. However, similarly
as for \(\inter{x}\), direct evaluation gives large overestimations
when \(\inter{s}\) is not very tight. In this case there is no
monotonicity to use, instead we compute a tighter enclosure using a
Taylor expansion as in \eqref{eq:taylor-bound}. One exception to this
is the modified Clausen function \(\hat{C}_{s}(x)\), for which we have
the following result
\begin{lemma}
  For \(s > 1\) and \(x \in \mathbb{R}\) the function
  \(\hat{C}_{s}(x)\) is decreasing in \(s\).
\end{lemma}
\begin{proof}
  We have
  \begin{equation*}
    \hat{C}_{s}(x) = C_{s}(x) - \zeta(s) = \sum_{n = 1}^{\infty}\frac{\cos(nx) - 1}{n^{s}} .
  \end{equation*}
  Since \(\cos(nx) - 1 \leq 0\) for \(x \in \mathbb{R}\) all terms in
  the sum have the same sign. That it is decreasing in \(s\) then
  follows easily from the monotonicity of \(n^{s}\).
\end{proof}
This means that for \(\inter{s} > 1\) we only have to evaluate on
\(\lo{s}\) and \(\hi{s}\) to get an enclosure. In practice
\(\hat{C}_{s}(x)\) is only every used with \(s > 1\) so this is enough
go get good bounds.

When \(\inter{s}\) contains a non-negative integer we have to handle
the removable singularities in \eqref{eq:clausenc-periodic-zeta} and
\eqref{eq:clausens-periodic-zeta} as described in the beginning of the
section. The required derivatives can generally be computed directly
with Arb, this is the case for the reciprocal gamma function,
\(\sin\), \(\cos\) as well as \(\zeta(s, x)\). There is an
implementation of the deflated zeta function
\(\underline{\zeta}(s, a)\) in Arb, however it only supports \(s\)
exactly equal to \(1\), and not intervals containing \(1\), and can
therefore not be used directly. In our application this case does
however not occur, so it is not an issue.

\subsection{Expansion in \(x\)}
\label{sec:expansion-x}
We now go through how to compute expansions of the Clausen functions
in the argument \(x\).

Around any point \(0 < x < 2\pi\) the functions are analytic and it is
straightforward to compute the Taylor expansions by differentiating
directly, using that \(\frac{d}{dx}C_{s}(x) = -S_{s - 1}(x)\) and
\(\frac{d}{dx}S_{s}(x) = C_{s - 1}(x)\).

At \(x = 0\) we then have the following asymptotic expansions
\cite{enciso2018convexity}
\begin{align}
  \label{eq:clausenc-asymptotic-x}
  C_{s}(x) &= \Gamma(1 - s)\sin\left(\frac{\pi}{2}s\right)|x|^{s - 1}
             + \sum_{m = 0}^{\infty} (-1)^{m}\zeta(s - 2m)\frac{x^{2m}}{(2m)!};\\
  \label{eq:clausens-asymptotic-x}
  S_{s}(x) &= \Gamma(1 - s)\cos\left(\frac{\pi}{2}s\right)\sign(x)|x|^{s - 1}
             + \sum_{m = 0}^{\infty} (-1)^{m}\zeta(s - 2m - 1)\frac{x^{2m + 1}}{(2m + 1)!}.
\end{align}
These expressions work well as long as \(s\) is not a positive
integer. For positive integers we have to handle the poles of
\(\Gamma(s)\) at non-positive integers and the pole of \(\zeta(s)\) at
\(s = 1\).

For \(C_{s}(x)\) with positive even integers \(s\) the only
problematic term is
\begin{equation*}
  \Gamma(1 - s)\sin\left(\frac{\pi}{2}s\right)|x|^{s - 1}
\end{equation*}
which has a removable singularity. Similarly for \(S_{s}(x)\) with
positive odd integers \(s\). For \(C_{s}(x)\) with positive odd
integers \(s\) and \(S_{s}\) with positive even integers \(s\) the
singularities are not removable, however in this paper we don't
encounter this case.

With the above we are able to compute expansions at \(x = 0\) to
arbitrarily high degree. What remains is to bound the tails of the
sums in \eqref{eq:clausenc-asymptotic-x} and
\eqref{eq:clausens-asymptotic-x}. For this we have the following
lemma, see also \cite[Lemma 2.1]{enciso2018convexity}. We omit the
proof since it is very similar to that in \cite{enciso2018convexity}.
\begin{lemma}
  \label{lemma:clausen-tails}
  Let \(s \geq 0\), \(2M \geq s + 1\) and \(|x| < 2\pi\), we then have
  the following bounds for the tails in equations
  \eqref{eq:clausenc-asymptotic-x} and
  \eqref{eq:clausens-asymptotic-x}
  \begin{align*}
    \left|\sum_{m = M}^{\infty} (-1)^{m}\zeta(s - 2m)\frac{x^{2m}}{(2m)!}\right|
    &\leq 2(2\pi)^{1 + s - 2M}\left|\sin\left(\frac{\pi}{2}s\right)\right|\zeta(2M + 1 - s) \frac{x^{2M}}{4\pi^{2} - x^{2}},\\
    \left|\sum_{m = M}^{\infty} (-1)^{m}\zeta(s - 2m - 1)\frac{x^{2m + 1}}{(2m + 1)!}\right|
    &\leq 2(2\pi)^{s - 2M}\left|\cos\left(\frac{\pi}{2}s\right)\right|\zeta(2M + 2 - s) \frac{x^{2M + 1}}{4\pi^{2} - x^{2}}.
  \end{align*}
\end{lemma}

\subsection{Derivatives in \(s\)}
\label{sec:clausen-derivatives-s}
For \(C_{s}^{(\beta)}(x)\) and \(S_{s}^{(\beta)}(x)\) we use
\eqref{eq:clausenc-periodic-zeta} and
\eqref{eq:clausens-periodic-zeta} and differentiate directly in \(s\).
When \(s\) is not an integer this is handled directly using Taylor
arithmetic, for integers we use the approach in Appendix
\ref{sec:removable-singularities} to handle the removable
singularities.

To get asymptotic expansions in \(x\) we take the expansions
\eqref{eq:clausenc-asymptotic-x} and \eqref{eq:clausens-asymptotic-x}
and differentiate them with respect to \(s\). Giving us
\begin{align}
  \label{eq:clausenc-derivative-asymptotic-x}
  C_{s}^{(\beta)}(x) &= \frac{d}{ds^{\beta}}\left(\Gamma(1 - s)\sin\left(\frac{\pi}{2}s\right)|x|^{s - 1}\right)
             + \sum_{m = 0}^{\infty} (-1)^{m}\zeta^{(\beta)}(s - 2m)\frac{x^{2m}}{(2m)!};\\
  \label{eq:clausens-derivative-asymptotic-x}
  S_{s}^{(\beta)}(x) &= \frac{d}{ds^{\beta}}\left(\Gamma(1 - s)\cos\left(\frac{\pi}{2}s\right)\sign(x)|x|^{s - 1}\right)
             + \sum_{m = 0}^{\infty} (-1)^{m}\zeta^{(\beta)}(s - 2m - 1)\frac{x^{2m + 1}}{(2m + 1)!}.
\end{align}
These formulas work well when \(s\) is not a positive odd integer for
\(C_{s}^{(\beta)}(x)\) or a positive even integer for
\(S_{s}^{(\beta)}(x)\) and the derivatives can be computed using
Taylor series. We mostly make use of the functions \(C_{2}^{(1)}(x)\)
and \(C_{3}^{(1)}(x)\), in which case the expansions can be computed
explicitly using \cite[Eq.~16]{Bailey2015}, for \(|x| < 2\pi\) we have
\begin{align*}
  C_{2}^{(1)}(x) =& \zeta^{(1)}(2) - \frac{\pi}{2}|x|\log|x| - (\gamma - 1)\frac{\pi}{2}|x|
                   + \sum_{m = 1}^{\infty}(-1)^{m}\zeta^{(1)}(2 - 2m)\frac{x^{2m}}{(2m)!}\\
  C_{3}^{(1)}(x) =& \zeta^{(1)}(3) - \frac{1}{4}x^{2}\log^2|x|
                    + \frac{3 - 2\gamma}{4}x^2\log|x|
                    - \frac{36\gamma - 12\gamma^2 - 24\gamma_{1} - 42 + \pi^{2}}{48}x^2\\
                 &+ \sum_{m = 2}^{\infty}(-1)^{m}\zeta^{(1)}(3 - 2m)\frac{x^{2m}}{(2m)!}.
\end{align*}
Where \(\gamma_{n}\) is the Stieltjes constant and
\(\gamma = \gamma_{0}\). To bound the tails we have the following
lemma:
\begin{lemma}
  \label{lemma:clausen-derivative-tails}
  Let \(\beta \geq 1\), \(2M \geq s + 1\) and \(|x| < 2\pi\), we then
  have the following bounds:
  \begin{multline*}
    \left|\sum_{m = M}^{\infty} (-1)^{m}\zeta^{(\beta)}(s - 2m)\frac{x^{2m}}{(2m)!}\right|\\
    \leq \sum_{j_{1} + j_{2} + j_{3} = \beta} \binom{\beta}{j_{1},j_{2},j_{3}}
    2\left(\log(2\pi) + \frac{\pi}{2}\right)^{j_{1}}(2\pi)^{s - 1}|\zeta^{(j_{3})}(1 + 2M - s)|
    \sum_{m = M}^{\infty} \left|p_{j_{2}}(1 + 2m - s)\left(\frac{x}{2\pi}\right)^{2m}\right|,
  \end{multline*}
  \begin{multline*}
    \left|\sum_{m = M}^{\infty} (-1)^{m}\zeta^{(\beta)}(s - 2m - 1)\frac{x^{2m + 1}}{(2m + 1)!}\right|\\
    \leq \sum_{j_{1} + j_{2} + j_{3} = \beta} \binom{\beta}{j_{1},j_{2},j_{3}}
    2\left(\log(2\pi) + \frac{\pi}{2}\right)^{j_{1}}(2\pi)^{s - 2}|\zeta^{(j_{3})}(2 + 2M - s)|
    \sum_{m = M}^{\infty} \left|p_{j_{2}}(2 + 2m - s)\left(\frac{x}{2\pi}\right)^{2m + 1}\right|.
  \end{multline*}
  Here \(p_{j_{2}}\) is given recursively by
  \begin{equation*}
    p_{k + 1}(s) = \psi^{(0)}(s)p_{k}(s) + p'_{k}(s),\quad p_{0} = 1,
  \end{equation*}
  where \(\psi^{(0)}\) is the polygamma function. It is given by a sum
  of terms of the form
  \begin{equation*}
    c (\psi^{(0)}(s))^{q_{0}}(\psi^{(1)}(s))^{q_{1}}\cdots (\psi^{(j_{2} - 1)}(s))^{q_{j_{2} - 1}}.
  \end{equation*}
  If \(x^{2} < e^{-\frac{q_{0}}{1 + 2M}}\) we have the following bounds
  \begin{multline*}
    \sum_{m = M}^{\infty}\left|
      c (\psi^{(0)}(1 + 2m - s))^{q_{0}}\cdots (\psi^{(j_{2} - 1)}(1 + 2m - s))^{q_{j_{2} - 1}}
      \left(\frac{x}{2\pi}\right)^{2m}
    \right|\\
    \leq |c(\psi^{(0)}(1 + 2M - s))^{q_{0}}\cdots (\psi^{(j_{2} - 1)}(1 + 2M - s))^{q_{j_{2} - 1}}|
    (1 + 2M)^{q_{0} / 2}(2\pi)^{2 - 2M}\frac{x^{2m}}{4\pi^{2} - Cx^{2}},
  \end{multline*}
  with \(C = e^{\frac{q_{0}}{1 + 2M}}\). If
  \(x^{2} < e^{-\frac{q_{0}}{2 + 2M}}\) we have
  \begin{multline*}
    \sum_{m = M}^{\infty}\left|
      c (\psi^{(0)}(2 + 2m - s))^{q_{0}}\cdots (\psi^{(j_{2} - 1)}(2 + 2m - s))^{q_{j_{2} - 1}}
      \left(\frac{x}{2\pi}\right)^{2m + 1}
    \right|\\
    \leq |c(\psi^{(0)}(2 + 2M - s))^{q_{0}}\cdots (\psi^{(j_{2} - 1)}(2 + 2M - s))^{q_{j_{2} - 1}}|
    (2 + 2M)^{q_{0} / 2}(2\pi)^{1 - 2M}\frac{x^{2m + 1}}{4\pi^{2} - Dx^{2}},
  \end{multline*}
  with \(D = e^{\frac{q_{0}}{2 + 2M}}\).
\end{lemma}
\begin{proof}
  We give the proof for the first case, the other one is similar. We
  have the functional identity
  \begin{equation*}
    \zeta(s) = 2(2\pi)^{s - 1}\sin\left(\frac{\pi}{2}s\right)\Gamma(1 - s)\zeta(1 - s).
  \end{equation*}
  If we let
  \begin{align*}
    f(s) &= 2(2\pi)^{s - 1}\sin\left(\frac{\pi}{2}s\right),\\
    g(s) &= \Gamma(1 - s),\\
    h(s) &= \zeta(1 - s),
  \end{align*}
  we have
  \begin{equation*}
    \zeta^{(\beta)}(s) =
    \sum_{j_{1} + j_{2} + j_{3} = \beta} \binom{k}{j_{1},j_{2},j_{3}} f^{(j_{1})}(s)g^{(j_{2})}(s)h^{(j_{3})}(s).
  \end{equation*}
  This gives us
  \begin{multline*}
    \sum_{m = M}^{\infty} (-1)^{m}\zeta^{(\beta)}(s - 2m)\frac{x^{2m}}{(2m)!}\\
    = \sum_{m = M}^{\infty} (-1)^{m}\frac{x^{2m}}{(2m)!} \sum_{j_{1} + j_{2} + j_{3} = \beta} \binom{\beta}{j_{1},j_{2},j_{3}} f^{(j_{1})}(s - 2m)g^{(j_{2})}(s - 2m)h^{(j_{3})}(s - 2m)\\
    = \sum_{j_{1} + j_{2} + j_{3} = \beta} \binom{\beta}{j_{1},j_{2},j_{3}} \sum_{m = M}^{\infty} (-1)^{m} f^{(j_{1})}(s - 2m)g^{(j_{2})}(s - 2m)h^{(j_{3})}(s - 2m) \frac{x^{2m}}{(2m)!}.
  \end{multline*}
  We are thus interested in bounding
  \begin{multline*}
    \left|\sum_{m = M}^{\infty} (-1)^{m} f^{(j_{1})}(s - 2m)g^{(j_{2})}(s - 2m)h^{(j_{3})}(s - 2m) \frac{x^{2m}}{(2m)!}\right|\\
    \leq \sum_{m = M}^{\infty} \left|f^{(j_{1})}(s - 2m)g^{(j_{2})}(s - 2m)h^{(j_{3})}(s - 2m) \frac{x^{2m}}{(2m)!}\right|.
  \end{multline*}
  for \(j_{1}, j_{2}, j_{3} \geq 0\).

  For \(f^{(j_{1})}(s - 2m)\) we start by noting that
  \begin{align*}
    f^{(j_{1})}(s) &= 2 \sum_{k = 0}^{j_{1}} \binom{j_{1}}{k} \frac{d^{k}(2\pi)^{s - 1}}{ds^{k}} \frac{d^{j_{1} - k}\sin\left(\frac{\pi}{2}s\right)}{ds^{j_{1} - k}}\\
                   &= 2(2\pi)^{s - 1}\sum_{k = 0}^{j_{1}} \binom{j_{1}}{k} \log(2\pi)^{k} \frac{d^{j_{1} - k}\sin\left(\frac{\pi}{2}s\right)}{ds^{j_{1} - k}}.
  \end{align*}
  Hence
  \begin{equation*}
    |f^{(j_{1})}(s)| \leq 2(2\pi)^{s - 1}\sum_{k = 0}^{j_{1}} \binom{j_{1}}{k} \log(2\pi)^{k}\left(\frac{\pi}{2}\right)^{k - j_{1}}
    = 2\left(\log(2\pi) + \frac{\pi}{2}\right)^{j_{1}}(2\pi)^{s - 1}
  \end{equation*}
  and, in particular,
  \begin{equation*}
    |f^{(j_{1})}(s - 2m)| \leq 2\left(\log(2\pi) + \frac{\pi}{2}\right)^{j_{1}}(2\pi)^{s - 1}(2\pi)^{-2m}.
  \end{equation*}

  For \(g^{(j_{2})}(s - 2m)\) we start from
  \(\Gamma^{(1)}(s) = \Gamma(s)\psi^{(0)}(s)\). Differentiating this
  we get
  \begin{equation*}
    \Gamma^{(j_{2})}(s) = \Gamma(s)p_{j_{2}}(s),
  \end{equation*}
  where \(p_{j_{2}}\) is given recursively by
  \begin{equation*}
    p_{k + 1} = \psi^{(0)}p_{k} + p'_{k},\quad p_{0} = 1.
  \end{equation*}
  This gives us
  \begin{equation*}
    |g^{(j_{2})}(s - 2m)| = \Gamma(1 + 2m - s)|p_{j_{2}}(1 + 2m - s)|.
  \end{equation*}

  For \(h^{(j_{3})}(s - 2m)\) we have
  \begin{equation*}
    h^{(j_{3})}(s - 2m) = (-1)^{j_{3}}\zeta^{(j_{3})}(1 + 2m - s).
  \end{equation*}
  Since \(2m \geq 2M \geq s + 1\) we have \(1 + 2m - s > 1\) and hence
  \begin{equation*}
    \zeta^{(j_{3})}(1 + 2m - s) = \sum_{k = 1}^{\infty}\frac{\log^{k} k}{k^{1 + 2m - s}},
  \end{equation*}
  which is decreasing in \(m\). Giving us
  \begin{equation*}
    |h^{(j_{3})}(s - 2m)| \leq |\zeta^{(j_{3})}(1 + 2M - s)|.
  \end{equation*}

  Combining the bounds for \(|f^{(j_{1})}|\), \(|g^{(j_{2})}|\) and
  \(|h^{j_{3}}|\) we have
  \begin{multline*}
    \sum_{m = M}^{\infty} \left|f^{(j_{1})}(s - 2m)g^{(j_{2})}(s - 2m)h^{(j_{3})}(s - 2m) \frac{x^{2m}}{(2m)!}\right|\\
    \leq 2\left(\log(2\pi) + \frac{\pi}{2}\right)^{j_{1}}(2\pi)^{s - 1}|\zeta^{(j_{3})}(1 + 2M - s)|\sum_{m = M}^{\infty} \left|(2\pi)^{-2m}\Gamma(1 + 2m - s) p_{j_{2}}(1 + 2m - s) \frac{x^{2m}}{(2m)!}\right|\\
    = 2\left(\log(2\pi) + \frac{\pi}{2}\right)^{j_{1}}(2\pi)^{s - 1}|\zeta^{(j_{3})}(1 + 2M - s)|\sum_{m = M}^{\infty} \left|p_{j_{2}}(1 + 2m - s)\frac{\Gamma(1 + 2m - s)}{(2m)!}\left(\frac{x}{2\pi}\right)^{2m}\right|.
  \end{multline*}
  Using that \((2m)! \geq \Gamma(2m + 1 - s)\) this simplifies to
    \begin{multline*}
    \sum_{m = M}^{\infty} \left|f^{(j_{1})}(s - 2m)g^{(j_{2})}(s - 2m)h^{(j_{3})}(s - 2m) \frac{x^{2m}}{(2m)!}\right|\\
    \leq 2\left(\log(2\pi) + \frac{\pi}{2}\right)^{j_{1}}(2\pi)^{s - 1}|\zeta^{(j_{3})}(1 + 2M - s)|\sum_{m = M}^{\infty} \left|p_{j_{2}}(1 + 2m - s)\left(\frac{x}{2\pi}\right)^{2m}\right|.
  \end{multline*}
  This proves the first part of the statement.

  Using that \(\frac{d}{ds}\psi^{(k)}(s) = \psi^{(k + 1)}(s)\) we can
  compute \(p_{k}\) for any fixed \(k\). It is clear that all terms of
  \(p_{j_{2}}(s)\) will be on the form
  \begin{equation*}
    c (\psi^{(0)}(s))^{q_{0}}(\psi^{(1)}(s))^{q_{1}}\cdots (\psi^{(j_{2} - 1)}(s))^{q_{j_{2} - 1}}.
  \end{equation*}
  From the integral representation \cite[Eq. 5.9.15]{NIST:DLMF} it
  follows that for \(s > 0\) we have \(\psi^{(0)}(s) < \log s\), using
  that \(\log s < \sqrt{s}\) we get \(\psi^{(0)}(s) < \sqrt{s}\). For
  \(s > 0\) and \(k \geq 1\) it follows from the integral
  representation
  \begin{equation*}
    \psi^{(k)}(s) = (-1)^{k + 1}\int_{0}^{\infty}\frac{t^{k}e^{-st}}{1 - e^{-t}}\ dt
  \end{equation*}
  that \(|\psi^{(k)}(s)|\) is decreasing in \(s\). For \(m \geq M\) we
  then have the bound
  \begin{multline*}
    |c (\psi^{(1)}(1 + 2m - s))^{q_{1}}\cdots (\psi^{(j_{2} - 1)}(1 + 2m - s))^{q_{j_{2} - 1}}|\\
    \leq |c(1 + 2m)^{q_{0} / 2}(\psi^{(1)}(1 + 2M - s))^{q_{1}}\cdots (\psi^{(j_{2} - 1)}(1 + 2M - s))^{q_{j_{2} - 1}}|.
  \end{multline*}
  This gives us
  \begin{multline*}
    \sum_{m = M}^{\infty}\left|
      c (\psi^{(0)}(1 + 2m - s))^{q_{0}}\cdots (\psi^{(j_{2} - 1)}(1 + 2m - s))^{q_{j_{2} - 1}}
      \left(\frac{x}{2\pi}\right)^{2m}
    \right|\\
    \leq |c(\psi^{(0)}(1 + 2M - s))^{q_{0}}\cdots (\psi^{(j_{2} - 1)}(1 + 2M - s))^{q_{j_{2} - 1}}|
    \sum_{m = M}^{\infty}(1 + 2m)^{q_{0} / 2}\left(\frac{x}{2\pi}\right)^{2m}.
  \end{multline*}
  Focusing on the sum we rewrite it as
  \begin{equation*}
    \frac{1}{2^{q_{0} / 2}}\left(\frac{x}{2\pi}\right)^{2M}\sum_{m = 0}^{\infty}\left(m + M + \frac{1}{2}\right)^{q_{0} / 2}\left(\frac{x}{2\pi}\right)^{2m}.
  \end{equation*}
  Here it can be noted that the sum is given by the Lerch transcendent
  \(\Phi\left(\frac{x^{2}}{4\pi^{2}}, -\frac{q_{0}}{2}, M +
    \frac{1}{2}\right)\). Hence
  \begin{equation*}
    \sum_{m = M}^{\infty}(1 + 2m)^{q_{0} / 2}\left(\frac{x}{2\pi}\right)^{2m}
    \leq \frac{1}{2^{q_{0} / 2}}(2\pi)^{-2M}x^{2M}
    \Phi\left(\frac{x^{2}}{4\pi^{2}}, -\frac{q_{0}}{2}, M + \frac{1}{2}\right).
  \end{equation*}
  Consider a constant \(C\) such that
  \begin{equation}
    \label{eq:lerch-C}
    \left(1 + \frac{m}{M + 1/2}\right)^{q_{0} / 2} < C^{m}\quad \text{ and }\quad \frac{Cx^{2}}{4\pi^{2}} < 1,
  \end{equation}
  with this we can bound the Lerch transcendent as
  \begin{align*}
    \sum_{m = 0}^{\infty}\left(m + M + \frac{1}{2}\right)^{q_{0} / 2}\left(\frac{x}{2\pi}\right)^{2m}
    &= \left(M + \frac{1}{2}\right)^{q_{0} / 2}\sum_{m = 0}^{\infty}\left(1 + \frac{m}{M + 1/2}\right)^{q_{0} / 2}\left(\frac{x^{2}}{4\pi^{2}}\right)^{m}\\
    &\leq \left(M + \frac{1}{2}\right)^{q_{0} / 2}\sum_{m = 0}^{\infty}C^{m}\left(\frac{x^{2}}{4\pi^{2}}\right)^{m}\\
    &= \left(M + \frac{1}{2}\right)^{q_{0} / 2} 4\pi^{2}\frac{1}{4\pi^{2} - Cx^{2}},
  \end{align*}
  If we take
  \(C = e^{\frac{q_{0}}{2}(M + 1/2)} = e^{\frac{q_{0}}{1 + 2M}}\) then
  the first inequality in \eqref{eq:lerch-C} is satisfied. For the
  second inequality to be satisfied we must then have
  \(x^{2} < e^{-\frac{q_{0}}{1 + 2M}}\), if this does not hold we can
  sum the first few terms of the series explicitly to work with a
  larger \(M\). Note that this method for bounding sum is the same as
  that used by Arb in the, currently unreleased,
  implementation\footnote{\url{https://fredrikj.net/blog/2022/02/computing-the-lerch-transcendent/}}
  of the Lerch transcendent.

  Putting all of this together we arrived at the bound
  \begin{multline*}
    \sum_{m = M}^{\infty}\left|
      c (\psi^{(0)}(1 + 2m - s))^{q_{0}}\cdots (\psi^{(j_{2} - 1)}(1 + 2m - s))^{q_{j_{2} - 1}}
      \left(\frac{x}{2\pi}\right)^{2m}
    \right|\\
    \leq |c(\psi^{(0)}(1 + 2M - s))^{q_{0}}\cdots (\psi^{(j_{2} - 1)}(1 + 2M - s))^{q_{j_{2} - 1}}|
    \frac{1}{2^{q_{0} / 2}}\left(\frac{x}{2\pi}\right)^{2M}\left(M + \frac{1}{2}\right)^{q_{0} / 2} 4\pi^{2}\frac{1}{4\pi^{2} - Cx^{2}}\\
    = |c(\psi^{(0)}(1 + 2M - s))^{q_{0}}\cdots (\psi^{(j_{2} - 1)}(1 + 2M - s))^{q_{j_{2} - 1}}|
    \left(1 + 2M\right)^{q_{0} / 2}(2\pi)^{2 - 2M}\frac{x^{2M}}{4\pi^{2} - Cx^{2}}.
  \end{multline*}
\end{proof}

\section{Rigorous integration with singularities}
\label{sec:rigorous-integration}
We here explain how to compute enclosures of the integrals
\(U_{1,1}(x)\), \(U_{1,2}\) and \(U_{2(x)}\) in the non-asymptotic
case, as well as how to enclose \(c_{1}\) and \(c_{2}\) occurring in
Lemmas \ref{lemma:asymptotic-U1} and \ref{lemma:asymptotic-U2}. Recall
that
\begin{align*}
  U_{1,1}(x) &= \int_{0}^{r_{x}}\log\left(\frac{\sin(x(1 - t) / 2)\sin(x(1 + t) / 2)}{\sin(xt / 2)^{2}}\right)t\sqrt{\log(1 + 1 / (xt))}\ dt;\\
  U_{1,2}(x) &= -\int_{r_{x}}^{1}\log\left(\frac{\sin(x(1 - t) / 2)\sin(x(1 + t) / 2)}{\sin(xt / 2)^{2}}\right)t\sqrt{\log(1 + 1 / (xt))}\ dt;\\
  U_{2}(x) &= -\int_{x}^{\pi}\log\left(\frac{\sin((y - x) / 2)\sin((x + y) / 2)}{\sin(y / 2)^{2}}\right)y\sqrt{\log(1 + 1 / y)}\ dy.
\end{align*}
The integrand for \(U_{1,2}\) has a (integrable) singularity at
\(t = 1\) and the integrand for \(U_{2}\) has one at \(y = x\). As a
first step we split these off to handle them separately. Let
\begin{align*}
     U_{1,2}(x) =& -\int_{r_{x}}^{1 - \delta_{1}}\log\left(\frac{\sin(x(1 - t) / 2)\sin(x(1 + t) / 2)}{\sin(xt / 2)^{2}}\right)t\sqrt{\log(1 + 1 / (xt))}\ dt;\\
              &- \int_{1 - \delta_{1}}^{1}\log\left(\frac{\sin(x(1 - t) / 2)\sin(x(1 + t) / 2)}{\sin(xt / 2)^{2}}\right)t\sqrt{\log(1 + 1 / (xt))}\ dt;\\
  =& U_{1,2,1}(x) + U_{1,2,2}(x)\\
  U_{2}(x) =& -\int_{x}^{x + \delta_{2}}\log\left(\frac{\sin((y - x) / 2)\sin((x + y) / 2)}{\sin(y / 2)^{2}}\right)y\sqrt{\log(1 + 1 / y)}\ dy\\
              &- \int_{x + \delta_{2}}^{\pi}\log\left(\frac{\sin((y - x) / 2)\sin((x + y) / 2)}{\sin(y / 2)^{2}}\right)y\sqrt{\log(1 + 1 / y)}\ dy\\
  =& U_{2,1}(x) + U_{2,2}(x).
\end{align*}

For \(U_{1,2,2}\) and \(U_{2,1}\) we have the following lemma
\begin{lemma}
  We have
  \begin{equation*}
    U_{1,2,2}(x) = \frac{\xi_{1}}{x}\left(
      -S_{2}(0) + S_{2}(2x) - 2S_{2}(x)
      - \left(-S_{2}\!\left(x\delta_{1}\right) + S_{2}\!\left(x(2 - \delta_{1})\right) - 2S_{2}\!\left(x(1 - \delta_{1})\right)\right)
    \right)
  \end{equation*}
  for some \(\xi_{1}\) in the image of \(t\sqrt{\log(1 + 1 / (xt))}\)
  for \(t \in [1 - \delta_{1}, 1]\). Furthermore
  \begin{equation*}
    U_{2,1}(x) = \xi_{2}\left(
      (S_{2}(\delta_{2}) + S_{2}(2x + \delta_{2}) - 2S_{2}(x + \delta_{2}))
      - (S_{2}(0) + S_{2}(2x) - 2S_{2}(x))
    \right)
  \end{equation*}
  for some \(\xi_{2}\) in the image of \(y\sqrt{\log(1 + 1 / y)}\) for
  \(y \in [x, x + \delta_{2}]\).
\end{lemma}
\begin{proof}
  We give the proof for \(U_{1,2,2}(x)\), the other one is similar.

  The factor \(t\sqrt{\log(1 + 1 / (xt))}\) in the integrand of
  \(U_{1,2,2}\) is bounded in \(t\) on the interval
  \([1 - \delta_{1}, 1]\) and the integrand has a constant sign, hence
  there exists \(\xi_{1}\) in the range of
  \(t\sqrt{\log(1 + 1 / (xt))}\) on this interval such that
  \begin{align}
    U_{1,2,2}(x) = \xi_{1} \int_{1 - \delta_{1}}^{1}-\log\left(\frac{\sin(x(1 - t) / 2)\sin(x(1 + t) / 2)}{\sin(xt / 2)^{2}}\right)\ dt.
  \end{align}
  It is therefore enough to compute an enclosure of this integral. For
  this we use that \(C_{1}\!(x) = -\log(2\sin(|x|/2))\) to write it as
  \begin{equation*}
    \int_{1 - \delta_{1}}^{1} C_{1}\!\left(x(1 - t)\right) + C_{1}\!\left(x(1 + t)\right) - 2C_{1}\!\left(xt\right)\ dt.
  \end{equation*}
  Using that \(\int C_{1}\!(t)\ dt = S_{2}(t)\) we get
  \begin{equation*}
    \int C_{1}\!\left(x(1 - t)\right) + C_{1}\!\left(x(1 + t)\right) - 2C_{1}\!\left(xt\right)\ dt
    = \frac{1}{x}\left(-S_{2}\!\left(x(1 - t)\right) + S_{2}\!\left(x(1 + t)\right) - 2S_{2}\!\left(xt\right)\right).
  \end{equation*}
  Integrating from \(1 - \delta_{1}\) to \(1\) gives us
  \begin{equation*}
    U_{1,2,2}(x) = \xi_{1}\frac{1}{x}\left(
      -S_{2}(0) + S_{2}(2x) - 2S_{2}(x)
      - \left(-S_{2}\!\left(x\delta_{1}\right) + S_{2}\!\left(x(2 - \delta_{1})\right) - 2S_{2}\!\left(x(1 - \delta_{1})\right)\right)
    \right),
  \end{equation*}
  which proves the result.
\end{proof}

What remains is computing \(U_{1,1}\), \(U_{1,2,1}\) and \(U_{2,2}\).
In this case the integrands are bounded everywhere on the intervals of
integration and we enclose the integrals using the rigorous numerical
integrator implemented by Arb \cite{Johansson2018numerical}. For
functions that are analytic on the interval the integrator uses
Gaussian quadratures with error bounds computed through complex
magnitudes, we therefore need to evaluate the integrands on complex
intervals. When the function is not analytic it falls back to naive
enclosures using interval arithmetic. The only non-trivial part is
bounding the integrand for \(U_{1,1}\) near \(t = 0\), where it is
bounded but not analytic. For this we start by splitting it as
\begin{multline*}
  \log\left(\frac{\sin(x(1 - t) / 2)\sin(x(1 + t) / 2)}{\sin(xt / 2)^{2}}\right)t\sqrt{\log(1 + 1 / (xt))}
  = \log(\sin(x(1 - t) / 2))t\sqrt{\log(1 + 1 / (xt))}\\
  + \log(\sin(x(1 + t) / 2))t\sqrt{\log(1 + 1 / (xt))}
  - 2\log(\sin(xt / 2))t\sqrt{\log(1 + 1 / (xt))}.
\end{multline*}
For the first two terms the only problematic part is the factor
\(t\sqrt{\log(1 + 1 / (xt))}\), we note that it is zero at \(t = 0\)
and differentiating it gives us
\begin{equation*}
  \frac{d}{dt}t\sqrt{\log(1 + 1 / (xt))} =
  \frac{2xt\log(1 + 1 / (xt)) + 2\log(1 + 1/(xt)) - 1}{(2xt + 1)\sqrt{\log(1 + 1 / (xt))}}.
\end{equation*}
If
\begin{equation*}
   2\log(1 + 1/(xt)) - 1 > 0 \iff t < \frac{1}{x(e^{1 / 2} - 1)}.
\end{equation*}
then the derivative is positive and we can thus get an enclosure by
checking that this inequality holds and using monotonicity. This
leaves us with the third term,
\begin{equation*}
  \log(\sin(xt / 2))t\sqrt{\log(1 + 1 / (xt))},
\end{equation*}
which is also zero at \(t = 0\) and for the monotonicity we have the
following lemma
\begin{lemma}
  For \(0 < x < \pi\) and \(0 < t < \frac{1}{10x}\) the function
  \begin{equation*}
    \log(\sin(xt / 2))t\sqrt{\log(1 + 1 / (xt))}
  \end{equation*}
  is decreasing in \(t\).
\end{lemma}
\begin{proof}
  Differentiating we have
  \begin{multline*}
    \frac{d}{dt}\log(\sin(xt / 2))t\sqrt{\log(1 + 1 / (xt))}\\
    = \frac{(2(xt + 1)\log(1 + 1 / (xt)) - 1)\log(\sin(xt / 2)) + xt(xt + 1)\log(1 + 1 / (xt))\cot(xt / 2)}{2(xt + 1)\sqrt{\log(1 + 1 / (xt))}}.
  \end{multline*}
  Since the denominator is positive it is enough to show that
  \begin{equation*}
    (2(xt + 1)\log(1 + 1 / (xt)) - 1)\log(\sin(xt / 2)) + xt(xt + 1)\log(1 + 1 / (xt))\cot(xt / 2) \leq 0.
  \end{equation*}
  Since \(t < \frac{1}{10x} < \frac{1}{x(e^{1 / 2} - 1)}\) the factor
  \((2(xt + 1)\log(1 + 1 / (xt)) - 1)\) is positive, multiplication by
  \(\log(\sin(xt / 2))\) makes it negative. Using that
  \(\log(\sin(xt / 2)) < \log(xt / 2) < \log(xt)\) an upper bound is
  hence given by
  \begin{equation*}
    (2\log(1 + 1 / (xt)) - 1)\log(xt) + xt(xt + 1)\log(1 + 1 / (xt))\cot(xt / 2).
  \end{equation*}
  If we let \(r = xt \in \left[0, \frac{1}{10}\right]\) this becomes
  \begin{equation*}
    (2\log(1 + 1 / r) - 1)\log(r) + r(r + 1)\log(1 + 1 / r)\cot(r / 2).
  \end{equation*}
  Using that \(r\cot(r / 2) \leq 2\) for \(|r| < \pi\) and
  \(\log(1 + 1 / r) > -\log(r)\) we get the upper bound
  \begin{equation*}
    (-2\log(r) - 1)\log(r) + 2(r + 1)\log(1 + 1 / r).
  \end{equation*}
  Splitting \(\log(1 + 1 / r) = \log(1 + r) - \log(r)\) gives
  \begin{equation*}
    (-2\log(r) - 1)\log(r) + 2(r + 1)\log(1 + r) - 2(r + 1)\log(r).
  \end{equation*}
  Using that \(2(r + 1) \leq \frac{11}{5}\) and
  \(\log(1 + r) \leq \log(11 / 10)\) it reduces to
  \begin{equation*}
    (-2\log(r) - 1)\log(r) + \frac{11}{5}\log(11 / 10) - \frac{11}{5}\log(r)
  \end{equation*}
  which is a quadratic expression in \(\log(r)\), which can easily be
  checked to be negative on \(\left[0, \frac{1}{10}\right]\).
\end{proof}

For enclosing \(c_{1}\) and \(c_{2}\) recall that
\begin{align*}
  c_{1} &= \int_{0}^{1}|\log(1 / t^{2} - 1)|t\sqrt{\log(1 / t)}\ dt;\\
  c_{2} &= \int_{0}^{\pi}y\sqrt{\log(1 + 1/y)}\ dy.
\end{align*}
Both integrands are bounded, by splitting \(c_{1}\) as
\begin{equation*}
  c_{1} = \int_{0}^{\frac{1}{\sqrt{2}}}\log(1 / t^{2} - 1)t\sqrt{\log(1 / t)}\ dt
  - \int_{\frac{1}{\sqrt{2}}}^{1}\log(1 / t^{2} - 1)t\sqrt{\log(1 / t)}\ dt,
\end{equation*}
they are also analytic except at \(t = 0\) and \(t = 1\) for \(c_{1}\)
and \(y = 0\) for \(c_{2}\). For \(c_{2}\) the integrand is easily
seen to be increasing for \(y > 0\) and it can be bounded near
\(y = 0\) using that. The integrand for \(c_{1}\) is not monotone, to
bound it we split it into three terms
\begin{equation*}
  \log(1 / t^{2} - 1)t\sqrt{\log(1 / t)} = t f_{1}(t) + t f_{2}(t) - 2f_{3}(t),
\end{equation*}
with
\begin{equation*}
  f_{1}(t) = \log(1 - t)\sqrt{\log(1/t)},\quad
  f_{2}(t) = \log(1 + t)\sqrt{\log(1/t)},\quad
  f_{3}(t) = t\log(t)\sqrt{\log(1/t)}.
\end{equation*}
It is then enough to bound \(f_{1}\), \(f_{2}\) and \(f_{3}\).
Differentiating we have
\begin{align*}
  f_{1}'(t) &= \frac{1}{2\sqrt{\log(1 / t)}}\left(\frac{2\log t}{1 - t} - \frac{\log(1 - t)}{t}\right),\\
  f_{2}'(t) &= -\frac{1}{2\sqrt{\log(1 / t)}}\left(\frac{2\log t}{1 + t} + \frac{\log(1 + t)}{t}\right),\\
  f_{3}'(t) &= \frac{1}{2}\sqrt{\log(1 / t)}(3 + 2\log(t)).
\end{align*}
For \(f_{3}'\) we get the unique root \(e^{-3 / 2}\) on the interval
\((0, 1)\), we can thus use monotonicity of \(f_{3}\) as long as we
avoid this point. For \(f_{1}'\) and \(f_{2}'\) we are looking for the
zeros of
\begin{equation*}
  g_{1}(t) = \frac{2\log t}{1 - t} - \frac{\log(1 - t)}{t}\quad
  \text{ and }\quad
  g_{2}(t) = \frac{2\log t}{1 + t} + \frac{\log(1 + t)}{t}
\end{equation*}
respectively. For \(g_{1}\) both \(\frac{2\log t}{1 - t}\) and
\(-\frac{\log(1 - t)}{t}\) are increasing, hence \(g_{1}\) is
increasing. From the limits \(\lim_{t \to 0^{+}} g_{1}(t) = -\infty\)
and \(\lim_{t \to 1^{-}} g_{1}(t) = \infty\) it follows that
\(f_{1}'\) has exactly one root on \((0, 1)\). This root can easily be
isolated using interval arithmetic and we can then use monotonicity of
\(f_{1}\) as long as we avoid this root. The function \(g_{2}\) is
also monotone, to see this we differentiate, giving us
\begin{equation*}
  g_{2}'(t) = - \frac{2t^{2}\log(t) + (1 + t)(t(\log(1 + t) - 3) + \log(1 + t))}{t^{2}(1 + t)^{2}}.
\end{equation*}
The sign is determined by
\begin{equation*}
  -(2t^{2}\log(t) + (1 + t)(t(\log(1 + t) - 3) + \log(1 + t))),
\end{equation*}
A lower bound is given by
\begin{equation*}
  (1 + t)(3t - t\log(1 + t) - \log(1 + t)),
\end{equation*}
that this is positive follows from the inequality \(t > \log(1 + t)\).
This proves that \(f_{2}'\) is monotone, we can then use the same
approach as for \(f_{1}\) to isolate the unique critical point.

\section*{Acknowledgments}

JD and JGS were partially supported by the ERC Starting Grant
ERC-StG-CAPA-852741. JGS was partially supported by NSF through Grant
NSF DMS-1763356. This material is based upon work supported by the
National Science Foundation under Grant No. DMS-1929284 while JD and
JGS were in residence at the Institute for Computational and
Experimental Research in Mathematics in Providence, RI, during the
program “Hamiltonian Methods in Dispersive and Wave Evolution
Equations”. JD was partially supported by the Swedish-American
foundation for the visit. We are also thankful for the hospitality of
the Princeton Department of Mathematics, the Uppsala University
Department of Mathematics and the Brown University Department of
Mathematics where parts of this paper were done. This work is
supported by the Spanish State Research Agency, through the Severo
Ochoa and María de Maeztu Program for Centers and Units of Excellence
in R\&D (CEX2020-001084-M).

\printbibliography

\begin{tabular}{l}
  \textbf{Joel Dahne} \\
  {Department of Mathematics} \\
  {Uppsala University} \\
  {L\"agerhyddsv\"agen 1, 752 37, Uppsala, Sweden} \\
  {Email: joel.dahne@math.uu.se} \\ \\
  \textbf{Javier G\'omez-Serrano}\\
  {Department of Mathematics} \\
  {Brown University} \\
  {Kassar House, 151 Thayer St.} \\
  {Providence, RI 02912, USA} \\ \\
  {and} \\ \\
  {Departament de Matem\`atiques i Inform\`atica} \\
  {Universitat de Barcelona} \\
  {Gran Via de les Corts Catalanes, 585} \\
  {08007, Barcelona, Spain} \\ \\
  {and} \\ \\
  {Centre de Recerca Matem\`atica} \\
  {Edifici C} \\
  {Campus Bellaterra} \\
  {08193 Bellaterra, Spain} \\
  {Email: javier\_gomez\_serrano@brown.edu, jgomezserrano@ub.edu} \\ \\
\end{tabular}

\end{document}